\documentclass[11pt]{amsart}

\title[Stability of wave-packet scattering]{The stability landscape in wave-packet scattering: geometric rigidity and sharp Sobolev thresholds}

\usepackage[a4paper,top=2.5cm,bottom=2.5cm,hmargin=2.5cm,heightrounded]{geometry}
\usepackage[T1]{fontenc}
\usepackage[utf8]{inputenc}
\usepackage{lmodern}
\usepackage{microtype}
\usepackage{amsmath,amssymb,amsthm,mathtools}
\mathtoolsset{showonlyrefs}
\usepackage{enumitem}
\usepackage{xcolor,standalone}
\usepackage{tikz,tikz-cd}
\usetikzlibrary{arrows.meta,calc}
\usepackage[hidelinks]{hyperref}
\usepackage{orcidlink}

\author[Max Getter]{%
Max Getter\textsuperscript{$\dagger$}\,
\orcidlink{0009-0001-1269-4623}%
}

\author[S. Ivan Trapasso]{%
S. Ivan Trapasso\textsuperscript{$\ddagger$}\,
\orcidlink{0000-0003-1707-8966}%
}

\hypersetup{colorlinks=true,linkcolor=blue,citecolor=blue,urlcolor=blue}

\theoremstyle{plain}
\newtheorem{theorem}{Theorem}[section]
\newtheorem{proposition}[theorem]{Proposition}
\newtheorem{lemma}[theorem]{Lemma}
\newtheorem{corollary}[theorem]{Corollary}

\theoremstyle{definition}
\newtheorem{definition}[theorem]{Definition}
\newtheorem{assumption}[theorem]{Assumption}

\theoremstyle{remark}
\newtheorem{remark}[theorem]{Remark}

\newcommand{\R}{\mathbb R}
\newcommand{\C}{\mathbb C}
\newcommand{\N}{\mathbb N}
\newcommand{\Nzero}{\mathbb N_0}
\newcommand{\Z}{\mathbb Z}

\newcommand{\calA}{\mathcal A}
\newcommand{\calC}{\mathcal C}
\newcommand{\calD}{\mathcal D}
\newcommand{\calF}{\mathcal F}
\newcommand{\calH}{\mathcal H}
\newcommand{\calN}{\mathcal N}
\newcommand{\calO}{\mathcal O}
\newcommand{\calP}{\mathcal P}
\newcommand{\calQ}{\mathcal Q}
\newcommand{\calS}{\mathcal S}
\newcommand{\calT}{\mathcal T}
\newcommand{\calU}{\mathcal U}
\newcommand{\calW}{\mathcal W}

\newcommand{\Id}{\mathrm{Id}}
\newcommand{\eps}{\varepsilon}
\newcommand{\wh}{\widehat}
\newcommand{\iu}{\mathrm{i}} 
\newcommand{\one}{\mathbf 1}
\newcommand{\la}{\langle}
\newcommand{\ra}{\rangle}

\newcommand{\aff}{\operatorname{aff}}

\newcommand{\supp}{\operatorname{supp}}

\newcommand{\norm}[1]{\left\lVert #1 \right\rVert}
\newcommand{\abs}[1]{\left\lvert #1 \right\rvert}

\newcommand{\given}{\nonscript\,\delimsize\vert\nonscript\,}
\DeclarePairedDelimiterX{\set}[1]{\lbrace}{\rbrace}{#1}

\newcommand{\bvQ}{\calQ^{(\alpha,\beta)}}

\newcommand{\lp}{\ell^2(L^2)}

\newcommand{\PathDepth}[1]{I_\ast^{#1}}
\newcommand{\PathsAll}{I_\ast^\ast}
\newcommand{\PathsLe}[1]{I_\ast^{\leq #1}}

\newcommand{\Scat}{\mathsf{S}_{\chi}}
\newcommand{\ScatSet}[1]{\mathsf{S}_{\chi}[#1]}
\newcommand{\ScatLe}[1]{\mathsf{S}_{\chi,\leq #1}}
\newcommand{\ScatOne}{\mathsf{S}_{\chi,1}}

\newcommand{\cact}{c_{\rm act}}
\newcommand{\Creg}{C_{\rm reg}}

\newcommand{\rgap}{r_{\rm gap}}

\newcommand{\dd}{\,\mathrm{d}}

\newcommand{\dt}{\,\mathrm{d}t}
\newcommand{\du}{\,\mathrm{d}u}

\newcommand{\dx}{\,\mathrm{d}x}
\newcommand{\dy}{\,\mathrm{d}y}
\newcommand{\dz}{\,\mathrm{d}z}

\newcommand{\deta}{\,\mathrm{d}\eta}

\newcommand{\dxi}{\,\mathrm{d}\xi}
\newcommand{\dzeta}{\,\mathrm{d}\zeta}

\keywords{Nonlinear harmonic analysis, wave-packet representations, scattering transforms, resolution--robustness trade-off, commutator estimates, microlocal analysis\vspace{1ex}}

\begin{document}
\maketitle

\begingroup
\renewcommand{\thefootnote}{\fnsymbol{footnote}}

\footnotetext[2]{%
Chair for Geometry and Analysis,
RWTH Aachen University,
D-52062 Aachen, Germany; \\ 
E-mail:
\href{mailto:getter@mathga.rwth-aachen.de}
{\texttt{getter@mathga.rwth-aachen.de}} (Corresponding author).
}

\footnotetext[3]{%
Department of Mathematical Sciences ``G.\ L.\ Lagrange'',
Politecnico di Torino,
Turin, Italy; \\
E-mail:
\href{mailto:salvatoreivan.trapasso@polito.it}
{\texttt{salvatoreivan.trapasso@polito.it}}.
}

\endgroup

\begin{abstract}
    A central problem in harmonic analysis is to understand how frequency resolution constrains stability under nonlinear changes of coordinates. Motivated by Mallat's wavelet scattering theory, we establish a sharp resolution--robustness principle for nonlinear representations built from two-dimensional wave-packet systems, showing that the transverse geometry of the frequency channels determines the critical Sobolev regularity. Whenever the channels are thinner than in the wavelet regime, arbitrarily small smooth deformations can move high-frequency mass between adjacent channels, producing instability already at the level of the linear filter bank and the first scattering layer. We match this obstruction with scale-sharp commutator estimates propagated through the full scattering cascade, obtaining Lipschitz stability at the critical regularity threshold. Together, the results provide the first intrinsic deformation-stability theory for non-wavelet Euclidean scattering transforms and identify a microlocal transport mechanism governing robustness across nonlinear multiscale representations.
\end{abstract}

\setcounter{tocdepth}{1}
\tableofcontents

\section{Introduction}
\label{sec:introduction}

The core problem studied in this paper concerns the compatibility of finely resolved multiscale representations with smooth changes of coordinates. Directional systems such as curvelets, shearlets, and wave atoms gain their resolution and sparse-approximation power by decomposing high frequencies into increasingly thin channels \cite{CandesDonoho2004,DemanetYing2007,GuoLabate2007}. Somewhat paradoxically, we show that the very geometric refinement responsible for this gain generates an unavoidable obstruction when these systems are used to build deformation-robust representations: the thinner the channels relative to their carrier frequency, the more readily a small diffeomorphism can transport energy across them, thus requiring greater input regularity to achieve stability. This phenomenon can be quantified in the framework of two-dimensional wave-packet coverings \cite{BytchenkoffVoigtlaender2020}, where a generation-$j$ frequency tile has radial length of order $2^{\alpha j}$ and transverse width of order $2^{\beta j}$, with $0\leq\beta\leq\alpha\leq1$. The narrowest geometric extent of the frequency tiles, encoded by the transverse exponent $\beta$, determines the critical Sobolev threshold $H^{1-\beta}$: under the respective filter assumptions, our positive and negative stability results match throughout $0\leq\beta\leq\alpha\leq1$, $\beta<1$.

This resolution--robustness duality has the character of an uncertainty principle, and admits indeed a microlocal interpretation: the underlying question is whether a wave-packet system approximately intertwines the canonical phase-space action induced by pullback under a diffeomorphism, a basic Fourier integral operation \cite{CandesDemanetFIO,TataruPhaseSpace}. The principal ingredient behind our stability results is indeed a scale-sharp anisotropic commutator estimate between wave-packet multipliers and Lipschitz transport fields, obtained by the paradifferential methods used in nonlinear PDE \cite{Auscher1995,Bahouri2011,Coifman1986,Li2019,Taylor-PDC}. The resulting derivative cost is dictated by the same transverse geometry that produces the approximation gain, and its optimality is established by a matching channel-separation construction.

Remarkably, this trade-off is already visible at the linear level and is therefore not specific to scattering-type representations; nevertheless, they provide a mathematically controlled setting in which the resolution--robustness principle can be propagated through arbitrary depth and converted into sharp representation-level stability and instability results. Our focus on these transforms, first introduced by Mallat for wavelet systems \cite{Mallat2012}, is motivated in part by their demonstrated effectiveness across a broad range of real-world signal-processing applications. More importantly, however, they are singled out by the analytic problem itself: as we discuss below, requiring stability to persist through depth at the nonlinear level essentially forces a scattering-type architecture. This also reveals a striking first-principles route to the structure of convolutional neural networks, which emerged empirically from different but ultimately convergent design heuristics.

\subsection{Stable features by scattering}

A memorable turning point in the rise of multilayer representations was the performance of AlexNet in the 2012 ImageNet Large Scale Visual Recognition Challenge \cite{AlexNet}, which marked the beginning of the dominance of deep convolutional architectures in computer vision tasks. While most of the attention from the mathematical community has since been directed toward the underlying statistical learning and optimization problems, several heuristics emerging from the field have not yet been framed within what is arguably their natural analytic environment, namely modern harmonic analysis. Indeed, a convolutional architecture ultimately combines frequency-selective filters, pointwise nonlinearities and averaging/pooling mechanisms into a multilayer hierarchy that may be viewed overall as a structured representation map $\Phi\colon L^2(\R^2)\to\calH$ sending the input signal into a suitable Hilbert space $\calH$ where target tasks become more accessible. In this context, $\calH$ is commonly referred to as \textit{feature space}.

Although this clearly resonates with the tradition of Fourier analysis, designing effective feature extractors requires balancing the tension between data compression and information preservation \cite{Bengio2013,CuckerSmale2002,Mallat2016}. For instance, keeping image classification in mind for concreteness, an effective feature map is expected to collapse variations of the input that are irrelevant to the task, such as small translations, rotations or deformations modeling changes of pose and local registration errors, without losing discriminative power. A spectral perspective helps make the challenge at stake more transparent, since the high-frequency components of a signal are precisely those encoding distinctive details such as edges, contours and patterns, which make the signal identifiable, and as such they are also most sensitive to spatial variations. 

A drastic way to achieve stability is to simply discard such unstable spectral components by means of low-pass averaging: given $\varphi\in\calS(\R^2)$ with $\varphi\geq0$ and $\|\varphi\|_{L^1}=1$ (a suitably normalized Gaussian function suffices), for $s>0$ one can consider the convolution operator
\begin{equation*}
    \calA_s f(x)=\int_{\R^2}f(x-y)\varphi_s(y)\dd{y}=f*\varphi_s(x),\qquad \text{where} \quad \varphi_s(y)=s^{-2}\varphi(s^{-1}y).
\end{equation*}
A simple Schur-type argument shows that $\calA_s f$ is stable under small spatially varying translations at scale $s$. More precisely, setting $L_\tau f(x)=f(x-\tau(x))$ for $\tau\in C_b^1(\R^2;\R^2)$, if $\norm{D\tau}_{L^\infty}\leq\kappa<1$, then
\begin{equation*}
    \norm{L_\tau(\calA_s f)-\calA_s f}_{L^2}\leq C\frac{\|\tau\|_{L^\infty}}{s}\|f\|_{L^2},
\end{equation*}
where $C>0$ depends on $\varphi$ and $\kappa$, and deteriorates as $\kappa\uparrow1$.

It is also not difficult to realize that the same strategy fails on high-frequency components without additional care in the scale separation step. Indeed, the linearized action of $L_\tau$ on the spectral side is given by $\xi\mapsto(I-D\tau(x))^\top\xi$; as a result, a distortion of gradient size $\norm{D\tau}_{L^\infty}\sim\varepsilon$ may induce an absolute frequency shift of order $\sim\eps R$ on signal components localized near $\abs{\xi}\sim R$, which can displace their energy across widely separated frequency channels when $R$ is large. To mitigate this rigidity, which ultimately stems from the fixed-width channel decomposition, one should instead group frequencies into channels whose width is comparable to the magnitude of their center frequency. Such a geometry becomes increasingly tolerant, at high frequencies, to the spectral displacement predicted by the first-order deformation law.

This is precisely the paradigm of \textit{multiresolution analysis} \cite{Mallat_book2009,Meyer_book1992}, which can be concretely implemented by means of a dyadic Littlewood--Paley tiling of the frequency plane associated with a low-pass filter $\chi$ and high-pass wavelet filters $\Psi=(\psi_\lambda)_{\lambda\in\Lambda}$, where one may think of $\lambda=(j,\theta)$ as recording dyadic scale $2^j$ and orientation $\theta$, thus satisfying
\begin{equation*}
    \abs{\wh\chi(\xi)}^2+\sum_{\lambda\in\Lambda}\abs{\wh{\psi_\lambda}(\xi)}^2=1\qquad\text{for a.e. }\xi\in\R^2.
\end{equation*}
Equivalently, by Plancherel's theorem, the associated filter-bank transform is an isometry:
\begin{equation*}
    \norm{f*\chi}_{L^2}^2+\sum_{\lambda\in\Lambda}\norm{f*\psi_\lambda}_{L^2}^2=\norm{f}_{L^2}^2,\qquad f\in L^2(\R^2).
\end{equation*}
Although the coefficient $f*\psi_\lambda$ isolates an oscillatory component, averaging it directly by the same low-pass filter $\chi$ to enforce stability would lead to trivial features since their spectral supports are separated. The missing step is demodulation: interspersing a pointwise nonlinearity like the modulus transfers part of the oscillatory information to a lower-frequency envelope while preserving the $L^2$ energy and commuting with deformations, producing the stable feature $f\mapsto \abs{f*\psi_\lambda}*\chi$. Again, some information is lost by the final low-pass filter, but the whole procedure can be iterated along the whole family of filters, resulting in a layered cascade that shares intriguing similarities with that of convolutional neural networks (CNNs).
These are, in essence, the main conceptual steps guiding the construction of the wavelet scattering transform introduced by Mallat in \cite{Mallat2012} in the two-dimensional case. To be more precise, let us define paths of length $n\geq1$ over the alphabet $\Lambda$ as tuples $p=(\lambda_1,\ldots,\lambda_n)\in\Lambda^n$; we set $\Lambda^0\coloneqq\{e\}$, where $e$ denotes the empty path, and write
\[\Lambda^\ast\coloneqq\bigcup_{n\geq0}\Lambda^n.\]
For $f\in L^2(\R^2)$, we define $U[e]f\coloneqq f$; if $p=(\lambda_1,\ldots,\lambda_n)\in\Lambda^n$ with $n\geq1$, we set
\begin{equation*}
    U[p]f\coloneqq U[\lambda_n]\cdots U[\lambda_1]f,\qquad U[\lambda]g\coloneqq \abs{g*\psi_\lambda}.
\end{equation*}
The \textit{scattering transform} is then the feature map $\Phi\colon L^2(\R^2)\to\calH$ constructed above, namely
\[\Phi(f)=\Scat f=(U[p]f*\chi)_{p\in\Lambda^\ast},\]
where the feature space $\calH=\ell^2(\Lambda^\ast;L^2(\R^2))$ is equipped with the norm
\begin{equation*}
    \|\Scat f\|_{\lp}=\biggl(\sum_{p\in\Lambda^\ast}\|U[p]f*\chi\|_{L^2}^2\biggr)^{1/2}.
\end{equation*}
The value of this ingeniously crafted feature extractor relies both on its provable analytic properties and on its remarkable performance in processing tasks, despite filters being appropriately fixed rather than trained. In image analysis, scattering coefficients were applied to handwritten-digit classification and texture discrimination in \cite{BrunaMallat2013}, while rotation- and scale-invariant extensions were developed in \cite{OyallonMallat2015,SifreMallat2013}. Scattering-based multiresolution representations were employed for audio source separation in \cite{SprechmannBrunaLeCun2015}. Beyond classical signal processing, wavelet scattering models have found successful applications in quantum chemistry \cite{EickenbergExarchakisHirnMallatThiry2018,HirnMallatPoilvert2017}, as well as in physics and cosmology \cite{ChengMorelAllysMenardMallat2024,LicciardiCarboneRondoniNagar2025,ValogiannisDvorkin2022}. More recently, related multiscale harmonic principles have entered hybrid learning models, from covariance texture synthesis \cite{BrochardZhangMallat2022} to wavelet score-based generative modeling \cite{GuthCosteDeBortoliMallat2022}.

We are more interested here in the theoretical side, and we hope that our brief account of the scattering transform convinced the reader that this peculiar CNN-like architecture is ultimately forced by designing around a small set of desirable properties, which in this setting can be proved rigorously rather than treated as mere heuristics. The starting point is again Mallat's pioneering paper \cite{Mallat2012}, where the following results are first proved. 

\begin{enumerate}[label=\textup{(\roman*)}]
    \item \textit{Nonexpansiveness.} The scattering transform is nonexpansive:
    \begin{equation*}
        \|\Scat f-\Scat g\|_{\lp}\leq\norm{f-g}_{L^2}, \qquad f,g\in L^2(\R^2).
    \end{equation*}
    This follows from the Littlewood--Paley identity for the filter bank and the pointwise nonexpansiveness of the modulus. As such, this result still holds when other semidiscrete filter frames beyond wavelets and more general nonlinearities are taken into account; see \cite{CzajaLi2019,WiatowskiBolcskei2018}. 

    \item \textit{Asymptotic translation invariance.} If $\chi_{_{R}}$ denotes a low-pass filter at spatial scale $R>0$, then in the limit $R\to\infty$ the windowed scattering transform converges in a suitable sense to a translation-invariant representation. Equivalently, for each fixed translation $L_y f=f(\cdot-y)$ one has
    \begin{equation*}
        \lim_{R\to\infty}\norm{\mathsf{S}_{\chi_{_{R}}}(L_y f)-\mathsf{S}_{\chi_{_{R}}} f}_{\lp}=0.
    \end{equation*}
    A novel proof of this result, under weaker assumptions and with broader scope, was recently obtained in \cite{CzajaKolstoeKoralov2024}.

    \item \textit{Energy conservation and propagation.} The Littlewood--Paley identity yields an exact balance between the energy captured by the low-pass outputs up to a given depth and the residual energy propagated to deeper layers. For wavelet scattering, under the additional technical scattering admissibility condition of \cite[Theorem~2.6]{Mallat2012}, Mallat proved that this residual energy vanishes asymptotically. Equivalently, the full scattering transform is norm preserving:
    \[\norm{\Scat f}_{\lp}=\norm{f}_{L^2}.\]
    Waldspurger later established the same conclusion under substantially weaker assumptions on the wavelet family \cite{Waldspurger2017}. More recently, it was shown in \cite{FuhrGetter2025} that the rate of energy propagation is a delicate geometric issue: a conjecture concerning energy decay over depth was settled by proving that this decay can be arbitrarily slow for wavelet scattering on generic $L^2$ inputs. This is in sharp contrast with other non-wavelet scattering transforms like those in \cite{CzajaLi2019}, where exponential energy decay is proved in the uniform time-frequency setting.

    \item \textit{Deformation stability.} Being the central thrust behind the entire architecture, it comes as no surprise that wavelet scattering is Lipschitz stable under small smooth diffeomorphisms: if $\tau\in C_b^2(\R^2;\R^2)$ satisfies $\norm{D\tau}_{L^\infty}<\frac12$, the scattering transform with output low-pass filter $\chi_{_{R}}$ at spatial averaging scale $R$ satisfies
    \begin{equation*}
        \norm{\mathsf{S}_{\chi_{_{R}}}(L_\tau f)-\mathsf{S}_{\chi_{_{R}}} f}_{\lp}\leq C K_R(\tau)\sum_{n=0}^{\infty}\|U_n f\|_{\lp},
    \end{equation*}
    where $U_n f = (U[p]f)_{p \in \Lambda^n}$ and the deformation cost has the form
    \begin{equation*}
        K_R(\tau)=\frac{\|\tau\|_{L^\infty}}{R}+\max\{\abs{\log R},1\}\|D\tau\|_{L^\infty}+\norm{D^2\tau}_{L^\infty}.
    \end{equation*}
    The first term is the expected sensitivity to translations at output scale $R$, while the remaining ones account for the nonrigid distortion $\tau$ transverse to the translation orbit. The latter arise from sophisticated commutator bounds between the wavelet transform and the deformation operator, which make the proof of this result particularly challenging. Note also that stability holds for sufficiently regular input signals, as subtly measured by the interplay with the scattering architecture via finiteness of the mixed $\ell^1(\ell^2(L^2))$ series on the right-hand side. The regularity of the deformation plays a key role as well: it was proved recently in \cite{NicolaTrapasso2023} that stability still occurs for wavelet scattering transforms with respect to input deformations of class $C^{\gamma}$ with $\gamma>1$, while instabilities arise at lower regularity $C^\gamma$ with $0 \le \gamma < 1$.
\end{enumerate}

\subsection{The geometric stability problem}

While our discussion of the scattering transform has deliberately remained at a broad, conceptual level and has not attempted to exhaust its many facets, it will have served its purpose if it has raised a few natural questions by this stage. In particular, it seems that some properties, such as nonexpansiveness, follow from the template scattering architecture and do not depend on the choice of the filters; others, like energy propagation, are tied much more delicately to the special geometry of the frequency covering. This naturally raises the question of how deformation stability fits into this broader perspective. The issue is particularly compelling because deformation stability was a guiding principle in the design of the scattering architecture, while the existing results in the literature establish it by strikingly different methods.

More precisely, Mallat's stability theorem in the wavelet setting gives deformation stability the status of an inherent \textit{structural} property of the representation, since its proof quantifies the interaction between the deformation and the frequency geometry of the filter bank. To the best of our knowledge, it remains to date the only full-depth theorem of this kind for scattering transforms under spatial diffeomorphisms; related results on manifolds, graphs and measure spaces concern finite-width constructions or different perturbation models for the underlying domain or operator, and are thus not direct analogues of the problem considered here \cite{ChewEtAl2024,GamaRibeiroBruna2019,KokeKutyniok2022,PerlmutterGaoWolfHirn2020}. By contrast, several deformation estimates for non-wavelet Euclidean scattering representations, including \cite[Theorem~2]{WiatowskiBolcskei2018} and \cite[Theorem~3.6(d)]{CzajaLi2019}, proceed through a \emph{decoupling} argument: one first restricts the input to a class that is already stable under deformations, such as bandlimited or Lipschitz signals \cite{WiatowskiBolcskei2018}, cartoon functions \cite{Donoho2001}, or multiresolution approximation spaces \cite{NicolaTrapasso2025}, and then transfers the input estimate to scattering by nonexpansiveness. For example, \cite{Koller2018} proves that for $f\in H^1(\R^2)$ and sufficiently small $\tau\in C_b^1(\R^2;\R^2)$,
\begin{equation*}
    \norm{L_\tau f-f}_{L^2}\leq 2\norm{\tau}_{L^\infty}\norm{\nabla f}_{L^2}.
\end{equation*}
It then follows that \textit{every} nonexpansive feature map $\Phi\colon L^2(\R^2)\to\calH$, of scattering type or not, satisfies
\begin{equation*}
    \norm{\Phi(L_\tau f)-\Phi(f)}_{\calH}\leq\norm{L_\tau f-f}_{L^2}\leq 4\pi \norm{\tau}_{L^\infty}\norm{f}_{H^1}.
\end{equation*}
Despite its broad scope and practical advantages, this approach does not exploit any structure of $\Phi$ beyond nonexpansiveness, since the deformation estimate is established at the input level.

The sharper problem for scattering-type architectures is instead how deformations interact with the frequency channelization. Equivalently, one asks whether the following diagram is approximately commutative:
\begin{equation*}
    \begin{tikzcd}[column sep=large,row sep=large] f \arrow[r,mapsto,"\calW"] \arrow[d,|->,"L_\tau"'] & \calW f \arrow[d,|->,"L_\tau"] \\ L_\tau f \arrow[r,mapsto,"\calW"'] & \calW(L_\tau f), \end{tikzcd}
\end{equation*}
where $\calW$ denotes the associated linear filter-bank transform (in particular, the wavelet transform for a wavelet filter bank)
\begin{equation*}
    \calW \colon L^2(\R^2)\to \ell^2(\{0\}\cup\Lambda;L^2(\R^2)), \quad\calW f= (f*\chi,(f*\psi_\lambda)_{\lambda\in\Lambda}).
\end{equation*}
On a more quantitative level, this reduces to controlling the commutator defect
\[[\calW,L_\tau]f=\calW(L_\tau f)-L_\tau(\calW f).\]
Measuring how frequency localization reacts to a change of variables or a transport vector field is a commutator problem in the classical harmonic-analytic sense: results of this kind are ubiquitous in Littlewood--Paley theory and paradifferential calculus \cite{Lannes2006,Taylor-PDC}, in Kato--Ponce inequalities \cite{KatoPonce} and in the analysis of transport-diffusion equations \cite{Danchin2005}. 

These considerations lead to the central question of the present work:
\begin{center}
    \emph{Which frequency geometries support deformation stability, and which are inherently unstable?}
\end{center}
We may already anticipate that the answer is governed by the relation between the magnitude of a channel's center frequency and its directional extents. Heuristically, deformation stability can persist only when these quantities remain asymptotically comparable throughout the filter bank; when this balance fails, arbitrarily small deformations may shift the active frequency region of a high-frequency channel by an amount comparable to one of its narrow scales. Such a displacement disrupts the cancellations required for the relevant diagram to approximately commute and thereby produces instability of the associated scattering architecture. In the present work, this mechanism is quantified by a single geometric parameter, that is, the \emph{transverse thickness} of the frequency channels. 

This issue is particularly pertinent for directional and anisotropic systems that refine the frequency resolution of isotropic wavelets. Gabor systems provide localized time-frequency representations \cite{Grochenig2001}, while curvelets and shearlets efficiently resolve oriented singularities \cite{CandesDonoho2004,GuoLabate2007}, and wave atoms sparsely represent oscillatory patterns \cite{DemanetYing2007}. Therefore, one may expect these advantages to persist in the corresponding scattering transforms. On the other hand, the enhanced resolution is achieved through increasingly thin frequency channels, raising the question of whether small deformations can transport high-frequency content across channel boundaries.

Before stating our main results more precisely, we introduce a unified framework in which the relevant phenomena can be identified uniformly rather than on a case-by-case basis. To this end, our setting is provided by the two-dimensional wave-packet coverings $\calQ^{(\alpha,\beta)}$ introduced by Bytchenkoff and Voigtlaender \cite{BytchenkoffVoigtlaender2020}. A thorough discussion of this construction is technical and can be found in Section \ref{sec:filters} below, while here we restrict ourselves to conveying the main heuristics. We regard the two-dimensional setting as the natural one in which to present our arguments: it is highly relevant for applications involving image-like input signals on the $\R^2$ plane, and it allows the underlying geometric framework to be visualized most transparently.

The wave-packet covering $\calQ^{(\alpha,\beta)}$ of the frequency plane comes with structural parameters $0\leq\beta\leq\alpha\leq1$ and consists of frequency tiles with the following property: a generation-$j$ tile is located at frequency radius comparable to $2^j$, has radial length $\sim 2^{\alpha j}$ and transverse width $\sim 2^{\beta j}$; in other words, $\alpha$ controls radial resolution of the tiles, while $\beta$ controls transverse/angular resolution. Figure~\ref{fig:bv-tile-intro} shows a model tile with the precise indexing of the covering introduced in Section~\ref{sec:filters}; in particular, the index $m$ records the radial position within the dyadic annulus, while $\theta_{j,\ell}$ selects the angular sector.

\begin{figure}[h]
    \centering
    \includegraphics[width=0.8\textwidth]{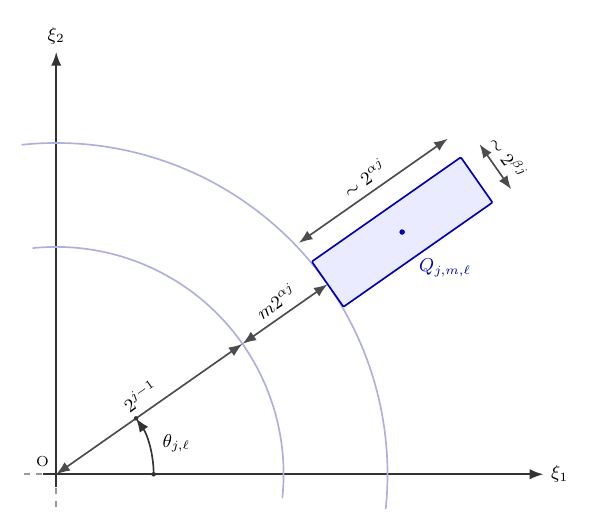}
    \caption{A generation-$j$ wave-packet tile (cf.\ Section \ref{sec:filters}). Its distance from the origin is of order $2^j$, its radial length is of order $2^{\alpha j}$, and its transverse width is of order $2^{\beta j}$. }
    \label{fig:bv-tile-intro}
\end{figure}

The $(\alpha,\beta)$ plane provides a unified picture of several interesting spectral geometries with transverse width smaller than the ambient frequency radius. For instance, the point $(\alpha,\beta)=(0,0)$ describes uniform coverings like those of Gabor type, whereas the vertical edge $\alpha=1$ hosts several well-known families of anisotropic coverings like ridgelets, curvelets and shearlets. The endpoint corner $(\alpha,\beta)=(1,1)$ corresponds to wavelets, while wave atoms correspond to $(\alpha,\beta)=(\frac12,\frac12)$; see Figure~\ref{fig:alpha-beta-plane} below for a summary phase diagram in this spirit. 

In order to discuss scattering transforms associated with a wave-packet covering it is necessary to complement the geometric view discussed so far with some analytic considerations. A first issue concerns the design of a suitable Littlewood--Paley filter bank subordinated to a given $\calQ^{(\alpha,\beta)}$, meaning that the corresponding Fourier multipliers are supported within the covering tiles. This transition is not canonical and its relevance should not be underestimated, since filters are the key objects in scattering-related arguments and the waveform design carries a significantly larger amount of information than the mere geometric one associated with the tiling. In particular, our main results rely on several mild but technical assumptions on the filters. For clarity, these are introduced progressively as needed; in Section \ref{sec:filters} we provide an explicit construction of a natural smooth Parseval filter bank satisfying all requirements imposed throughout the paper.

Another analytic aspect concerns measuring function regularity through the weighted summability of spectral coefficients associated with the frequency tiling induced by the covering $\calQ^{(\alpha,\beta)}$. This is precisely how \textit{wave-packet smoothness spaces} are defined in \cite{BytchenkoffVoigtlaender2020}: one takes a regular partition of unity $\Phi=(\varphi_i)_{i\in I}$ subordinate to $\calQ^{(\alpha,\beta)}$ and measures the size of the pieces $\calF^{-1}(\varphi_i\widehat f)$ after inserting suitable scale weights. Their properties and embeddings into more familiar regularity spaces are largely investigated in that work. For our purposes, since we are inherently working in an $L^2$-based setting, it is enough to record that the Hilbertian wave-packet smoothness spaces associated with $\calQ^{(\alpha,\beta)}$ identify with the standard Sobolev spaces $H^s(\R^2)$, with equivalent norms, which thus become the natural family of $L^2$-based regularity inputs for the purposes of stability. 

\subsection{Geometric rigidity forces instability}

A possible instability mechanism can be sensed from Figure \ref{fig:bv-tile-intro}. Since a deformation gradient of size $\sim\varepsilon$ moves a carrier at frequency $2^j$ by $\sim \varepsilon 2^j$, to cross a transverse channel of width $2^{\beta j}$ one needs these quantities to be comparable, that is, $\varepsilon \sim 2^{-(1-\beta)j}$. This remark suggests that for $\beta<1$ one can potentially design arbitrarily small deformations that swap wave-packet channels at high frequency; at the same time, it seems clear that the wavelet geometry $\beta=1$ is exceptional, being structurally immune to such disruptions. 

While it is not difficult to craft affine deformation models enforcing this channel-swap phenomenon (see Lemma \ref{lem:BV-separation} below), the main challenge here is to lift these findings into quantitative lower bounds for filter-deformation commutators and scattering outputs. Among several questions, it is interesting to investigate how the stability break depends on scattering depth and, in view of the problems raised in \cite{NicolaTrapasso2023}, whether the regularity of the adversarial deformation and the input signal class play a hidden subtle role in its occurrence. The answers to these questions are given in Section \ref{sec:L2-negative}, under mild nondegeneracy and regularity assumptions on the scattering filters. Without entering into the technical details, we distill our main findings in the following claim.

\begin{theorem}\label{thm:intro-neg}
    Assume $0\leq\beta<1$ and suppose that the subordinate Parseval scattering bank satisfies the hypotheses of Section \ref{sec:L2-negative}.
    There exist functions $f_j\in\calS(\R^2)$ with $\norm{f_j}_{L^2}=1$ and compactly supported smooth deformations $\tau_j\in C_c^\infty(\R^2;\R^2)$, as well as a constant $c>0$, independent of $j$, such that, for all sufficiently large $j$,
    \begin{equation}
        \norm{\tau_j}_{C^m}\lesssim_m 2^{-(1-\beta)j}, \qquad \text{for all } m\in\N_0,
    \end{equation}
    and, setting $\ScatOne f=(\abs{f*\psi_i}*\chi)_{i\in I_\ast}$ for the first-layer scattering coefficients, 
    \begin{equation}
        \norm{\ScatOne(L_{\tau_j}f_j)-\ScatOne f_j}_{\lp}\geq c.
    \end{equation}
    Moreover, for the associated linear filter transform $\calW$ one also has, for all sufficiently large $j$,
    \begin{equation}
        \norm{[\calW, L_{\tau_j}]f_j}_{\lp}\geq c.
    \end{equation}
\end{theorem}
Theorem~\ref{thm:intro-neg} is consequential already for the low-order scattering representations typically used in applications. Indeed, practical implementations of the scattering cascade are usually truncated after only two layers while still yielding sufficiently informative features in diverse contexts, ranging from audio recognition \cite{AndenMallat2014} to the detection of retinal abnormalities from optical coherence tomography images \cite{BaharloueiRabbaniPlonka2023}; recent work also considers stochastic path sampling as a way to reduce the computational cost of multivariable scattering transforms \cite{MitcheltreeEtAl2026}. More generally, if $\ScatLe{N}$ denotes any finite depth truncation containing the first-layer output, for every $N\geq1$ one has
\begin{equation*}
    \norm{\ScatOne(L_{\tau_j}f_j)-\ScatOne f_j}_{\lp}\leq\norm{\ScatLe{N}(L_{\tau_j}f_j)-\ScatLe{N}f_j}_{\lp}\leq\norm{\Scat(L_{\tau_j}f_j)-\Scat f_j}_{\lp}.
\end{equation*}
As a result, the same lower bound applies simultaneously to every such finite truncation and to the full infinite-depth scattering transform.

Moreover, we stress that the linear lower bound in Theorem \ref{thm:intro-neg} shows that the stability failure is caused neither by the modulus nonlinearity nor by any lack of regularity of the inputs or deformations, confirming that the underlying mismatch between frequency channelization and spatial distortion has a deeper root. On the other hand, the previous result rules out any stability estimate based solely on $L^2$ control, leaving open the possibility that the instability is merely an artifact of this level of generality, where pathological examples and ad hoc constructions are plentiful. Restricting the input class to suitable subspaces of $L^2$ is therefore unavoidable to circumvent this obstruction. A natural, if not essentially forced, choice is represented by the wave-packet decomposition spaces inherently associated with the geometry behind the network architecture; as already mentioned, in the $L^2$-based scenario these spaces collapse to the standard Sobolev regularity scale. The following result shows that the geometric instability mechanism places a precise regularity barrier to scattering stability in this setting: 

\begin{corollary} \label{cor:intro-modulus}
    Assume the hypotheses of Theorem \ref{thm:intro-neg}, and fix $m\in\Nzero$, $s\geq0$, and $\kappa\in(0,1)$. Let $\omega\colon[0,1]\to[0,\infty)$ be nondecreasing with $\omega(0)=0$ and $\omega(t)\to0$ as $t\downarrow0$. Suppose that there exists a constant $C<\infty$ such that
    \begin{equation*}
        \norm{\ScatOne(L_\tau f)-\ScatOne f}_{\lp}\leq C\,\omega(\norm{\tau}_{C^m})\norm{f}_{H^s}
    \end{equation*}
    for every $f\in H^s(\R^2)$ and every $\tau\in C_c^\infty(\R^2;\R^2)$ with $\norm{\tau}_{C^m}\leq1$ and $\norm{D\tau}_{L^\infty}\leq\kappa$. Then necessarily
    \[\limsup_{t\downarrow0}\frac{\omega(t)}{t^{s/(1-\beta)}}>0.\]
    In particular, for every $0\leq s<1-\beta$, no Lipschitz estimate of the form
    \begin{equation*}
        \norm{\ScatOne(L_\tau f)-\ScatOne f}_{\lp}\leq C\norm{\tau}_{C^m}\norm{f}_{H^s}
    \end{equation*}
    can hold uniformly over all $f\in H^s(\R^2)$ and all compactly supported smooth deformations $\tau$ satisfying $\norm{\tau}_{C^m}\leq1$ and $\norm{D\tau}_{L^\infty}\leq\kappa$.
\end{corollary}

The exponent $1-\beta$ is therefore the analytic manifestation of the scale disparity, already present in the first layer, between the carrier frequency $2^j$ and the transverse channel width $2^{\beta j}$. Indeed, the adversarial packets in the instability construction are concentrated at frequencies of order $2^j$, and hence, after $L^2$ normalization, their $H^s$ norms grow like $2^{sj}$. At the same time, the size of the corresponding deformations is of order $2^{-(1-\beta)j}$, which is precisely the scale required to displace such packets across a transverse channel of width $2^{\beta j}$. If a uniform estimate held with modulus $o(t^{s/(1-\beta)})$, its right-hand side would tend to zero along this sequence, while the first-layer scattering outputs remain separated by a fixed positive amount.

\subsection{Sharp stability at the Sobolev scale}

The negative results just discussed naturally raise some complementary questions: if $L^2$ stability fails for every genuinely packet-like transverse geometry $\beta<1$, and if Lipschitz stability cannot hold below $H^{1-\beta}$, does the threshold Sobolev scale $H^{1-\beta}$ actually restore stability? In other terms, is this regularity level enough to control the corresponding deformation commutator bound? 

We answer these questions in the affirmative, providing in particular what seem to be the first examples of structural scattering stability beyond the original wavelet geometry of \cite{Mallat2012}. The main ingredient of our stability theory is a Mallat-type commutator result, obtained by adapting paradifferential methods to the anisotropic geometry of the filters under Assumption~\ref{ass:comm-filter-ass}. This is precisely the point where the smoothness and anisotropic scale behavior of the multipliers critically interact: the commutator argument rests on controlling scale-regularity interactions in terms of derivatives of filters conjugated by the deformation flow, which give rise to commutators between wave-packet multipliers and suitable transport vector fields. In our setting, at generation $j$ the multiplier changes at transverse scale $\sim2^{\beta j}$, while the vector field acts on an oscillation of size $\sim2^j$. The overall cost is therefore $\sim2^{(1-\beta)j}$, which is exactly the Sobolev weight corresponding to $H^{1-\beta}$ already appearing on the negative edge of the story.

To state the main result, let $U_nf$ denote the collection of propagated signals at depth $n$ and set
\[\mathfrak E_s(f)=\sum_{n=0}^{\infty}\norm{U_nf}_{\ell^2(H^s)}.\]
Intuitively, this energy-like quantity measures the Sobolev regularity that remains distributed along the whole scattering tree. Indeed, the propagated family $U_nf$ is the input of the $n$-th scattering layer, so any stability principle that is meant to be iterated through the network should measure not only the regularity of $f$ at the entrance but the regularity still present at each subsequent depth. The sum defining $\mathfrak E_s(f)$ records exactly this cumulative Sobolev content. The first positive theorem is therefore formulated in terms of $\mathfrak E_{1-\beta}(f)$:
\begin{theorem}\label{thm:intro-positive-abstract}
    Assume $0\leq\beta<1$ and suppose that the subordinate Parseval scattering bank satisfies the hypotheses of Section \ref{sec:positive-full-depth}. For every $\kappa\in(0,1)$ there exists $C>0$ such that, whenever $\tau\in W^{1,\infty}(\R^2;\R^2)$ satisfies $\norm{D\tau}_{L^\infty}\leq\kappa$, one has
    \begin{equation*}
        \norm{\Scat(L_\tau f)-\Scat f}_{\lp}\leq C\bigl(\norm{\tau}_{L^\infty}+\norm{D\tau}_{L^\infty}\bigr)\mathfrak E_{1-\beta}(f).
    \end{equation*}
\end{theorem}

This Lipschitz stability estimate should be viewed as the positive counterpart of the instability mechanism discussed above. The threshold $1-\beta$ matches exactly both sides of the theory: negatively, as the inverse scale of the deformation needed to push a packet across a transverse channel; positively, as the Sobolev cost required to commute the deformation through such channels. In this sense, Theorem~\ref{thm:intro-positive-abstract} is already sharp at the level of the propagated quantity $\mathfrak E_{1-\beta}(f)$.

A comparison with Corollary~\ref{cor:intro-modulus} requires replacing $\mathfrak E_{1-\beta}(f)$ by a standard Sobolev norm of the input. This amounts to controlling the critical Sobolev content along the infinite cascade, a question related to the energy-propagation problem studied in \cite{FuhrGetter2025}. In the present setting, removing the carrier from each filtered signal leaves its modulus unchanged while reducing the Sobolev weight on its Fourier support. The high-pass gap makes this reduction uniform, and the Sobolev nonexpansiveness of the modulus then gives a strict contraction at each layer. More precisely, Lemma~\ref{lem:summability-above-critical} gives $\norm{U_nf}_{\ell^2(H^s)}\leq\varrho^{sn}\norm{f}_{H^s}$ for every $0<s\leq1$, with a fixed $\varrho\in(0,1)$. The required geometric conditions hold throughout the parameter triangle, including $\alpha=1$. Taking $s=1-\beta$ and summing over depth therefore gives the following endpoint estimate under the commutator assumption alone.

\begin{corollary}\label{cor:intro-positive-endpoint}
    Assume $0\leq\beta\leq\alpha\leq1$, $\beta<1$, and suppose that the subordinate Parseval scattering bank satisfies Assumption~\ref{ass:comm-filter-ass}. For every $\kappa\in(0,1)$ there exists $C>0$ such that, whenever $\tau\in W^{1,\infty}(\R^2;\R^2)$ satisfies $\norm{D\tau}_{L^\infty}\leq\kappa$, for every $f\in H^{1-\beta}(\R^2)$ one has
    \begin{equation*}
        \norm{\Scat(L_\tau f)-\Scat f}_{\lp}\leq C\bigl(\norm{\tau}_{L^\infty}+\norm{D\tau}_{L^\infty}\bigr)\norm{f}_{H^{1-\beta}}.
    \end{equation*}
    The same constant applies to every finite-depth truncation $\ScatLe{N}$, $N\in\Nzero$.
\end{corollary}

The estimate also holds with $\norm{f}_{H^s}$ on the right for every $s\geq1-\beta$, by Sobolev embedding. For banks satisfying the negative hypotheses of Corollary~\ref{cor:intro-modulus}, this threshold is sharp. In addition to the value of its structural nature, the quality of the endpoint estimate is best read against the decoupled route discussed earlier. If the latter is taken as baseline, in Corollary~\ref{cor:sobolev-decoupling-bound} below we prove that, at the critical regularity $s=1-\beta$ with $0<\beta<1$, nonexpansiveness of the scattering transform gives
\begin{equation*}
    \norm{\Scat(L_\tau f)-\Scat f}_{\lp}\lesssim\norm{\tau}_{L^\infty}^{1-\beta}\norm{f}_{H^{1-\beta}}.
\end{equation*}
The exponent $1-\beta$ is optimal for the underlying input-space estimate, already for constant translations and the identity feature map. To compare the resulting deformation moduli, write $\tau=\varepsilon\sigma$, where $\sigma\in W^{1,\infty}(\R^2;\R^2)$ is fixed and $\varepsilon\downarrow0$. The best architecture-blind estimate is only H\"older in the deformation size of order $\calO(\varepsilon^{1-\beta})$, while our structural estimate instead gives $\calO(\varepsilon)$ at the same critical scale, thus improving the deformation modulus without raising the Sobolev regularity since $\varepsilon^{1-\beta}\gg\varepsilon$ as $\varepsilon\downarrow0$. Equivalently, when $\beta>0$, the decoupled route can recover Lipschitz dependence only by going up to $H^1$ level, whereas the structural machinery achieves it already at the sharp scale $H^{1-\beta}$. 

The endpoint $\beta=0$ is qualitatively different: in this case the critical space is already $H^1$, and both the architecture-blind and architecture-intrinsic routes yield Lipschitz dependence on $\varepsilon$. Uniform time-frequency coverings, including those of Gabor type, therefore occupy the least favorable position in the present stability scale, since exploiting the scattering architecture cannot reduce the Sobolev regularity required for deformation stability. A surprising reversal emerges when the same geometries are compared from the perspective of energy propagation, which favors the endpoint geometries in precisely the opposite order: uniform time-frequency scattering is guaranteed to satisfy exponential decay of the energy propagated to deeper layers \cite{CzajaLi2019}, whereas arbitrarily slow decay may occur in the wavelet scenario \cite{FuhrGetter2025}.

The endpoint estimate includes the boundary $\alpha=1$, $\beta<1$, which contains the curvelet and shearlet scaling regimes represented in Figure~\ref{fig:alpha-beta-plane}. These geometric correspondences do not assert stability for every filter family bearing those names: the theorem applies to the separately indexed oriented banks subordinate to the specified covering and satisfying Assumption~\ref{ass:comm-filter-ass}.

\begin{figure}[h]
    \centering
    \includegraphics[width=0.8\textwidth]{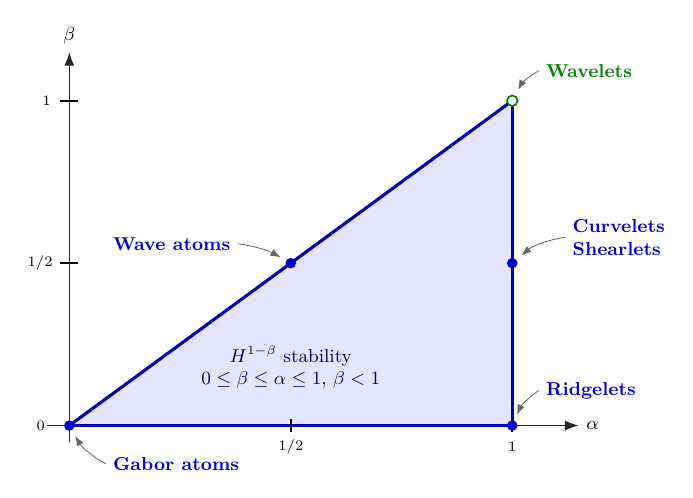}
    \caption{The parameter triangle $0\leq\beta\leq\alpha\leq1$. For the subordinate banks satisfying Assumption~\ref{ass:comm-filter-ass}, endpoint $H^{1-\beta}$ stability holds throughout $\beta<1$. The threshold matches the first-layer lower bound for banks also satisfying the negative hypotheses of Corollary~\ref{cor:intro-modulus}. The wavelet corner $(1,1)$ is excluded from the endpoint theorem and is immune to the channel-swap obstruction.}
    \label{fig:alpha-beta-plane}
\end{figure}

We conclude by emphasizing the exceptional role of the wavelet corner as the singular endpoint at which the geometric mechanism underlying the instability is no longer available. When $\beta=1$, the transverse width of a generation-$j$ channel is comparable to the carrier frequency scale $2^j$; a deformation capable of moving a wave packet across such a channel must therefore have order-one size, and the small-deformation channel-swap mechanism disappears. It is striking that Mallat's choice of wavelets as building blocks for the original scattering architecture lies precisely at this exceptional endpoint and is thus immune to the geometric obstruction identified here. This observation also clarifies the relation between Mallat's theorem and the present results: as $\beta\uparrow1$, the critical Sobolev exponent $1-\beta$ tends to zero, in agreement with the propagated $L^2$-type quantity entering Mallat's wavelet deformation bound. The latter should not, however, be interpreted as the limiting case of our estimates, since neither the constants nor the analytic mechanisms underlying our arguments are claimed to remain uniform as $\beta\uparrow1$.

\subsection{Organization of the paper}

In Section~\ref{sec:preliminaries} we fix notation and collect preliminary results. The goal of Section~\ref{sec:architecture-blind-stability} is to isolate the consequences of input regularity and nonexpansiveness, including the optimal Sobolev modulus available to architecture-blind arguments.  Section~\ref{sec:filters} introduces the wave-packet coverings $\calQ^{(\alpha,\beta)}$ and recalls their basic geometric properties; we also explicitly construct there a smooth Parseval class of scattering filters subordinate to the covering. Section~\ref{sec:L2-negative} is devoted to the exploration of the instability mechanism and its consequences, while the stability counterpart is investigated in Section~\ref{sec:positive-full-depth} using the commutator argument that is thoroughly developed in Section \ref{sec:appendix-commutator}. 

\section{Preliminaries}
\label{sec:preliminaries}

\subsection*{Notation and conventions}
We take $\N=\{1,2,\ldots\}$ and $\Nzero=\{0,1,2,\ldots\}$.

We work on $\R^2$ with the following normalization for the Fourier transform:
\[\wh f(\xi)=\int_{\R^2} e^{-2\pi \iu x\cdot\xi} f(x)\dd{x}.\]
For a countable family $q=(q_i)_i$ of functions, we write
\begin{equation*}
    \norm q_{\lp}=\left(\sum_i\norm{q_i}_{L^2}^2\right)^{1/2},\qquad \norm q_{\ell^2 (H^s)}=\left(\sum_i\norm{q_i}_{H^s}^2\right)^{1/2},
\end{equation*}
where $\norm f_{H^s}=\norm{\la\xi\ra^s\wh f(\xi)}_{L^2_\xi}$ denotes the standard Sobolev norm with $s\in \R$ and $\la\xi\ra=(1+\abs{\xi}^2)^{1/2}$. 

Throughout this paper, for non-negative quantities $X$ and $Y$, we write $X \lesssim Y$ to denote $X \leq CY$ for some generic constant $C > 0$ that may vary from line to line, and $X\asymp Y$ when both inequalities $X\lesssim Y$ and $Y\lesssim X$ hold. When we wish to emphasize that a constant depends on a specific parameter such as $\gamma$, we write $C_\gamma$ or $X \lesssim_\gamma Y$ and $X\asymp_\gamma Y$. 

The open Euclidean ball of radius $r>0$ centered at $x_0 \in \R^2$ is
\[B_r(x_0) = \{\xi\in\R^2:|\xi-x_0|<r\}.\]
The Minkowski sum of $E,F\subset\R^2$ is $E+F=\{e+f:e\in E, f\in F\}$, while if $A \in \R^{2\times 2}$ is a matrix, we set $AE=\{Ax:x\in E\}$. The symbol $\one_E$ denotes the indicator function of a set $E$. 

For matrices, $\norm{A}$ denotes the operator norm induced by the Euclidean norm, and $\norm{D\tau}_{L^\infty}$ is its essential supremum. A deformation field is a map $\tau \colon \R^2\to\R^2$, the corresponding deformation operator acting on $f \colon \R^2 \to \C$ being
\[L_\tau f(x)=f(x-\tau(x)).\]
For $m \in \N_0$ we set
\begin{equation*}
    \norm{\tau}_{C^m}=\sum_{\abs{\nu}\leq m}\norm{\partial^\nu\tau}_{L^\infty},
\end{equation*}
with the standard multi-index notation. The space $C_b^k(\R^2;\R^2)$ consists of $C^k$ vector fields whose derivatives up to order $k$ are bounded. 

We first record a basic fact about Mallat-type deformations viewed as composition operators on Sobolev spaces. Besides serving as a convenient reference, it ensures that the deformation operators considered throughout the paper are well defined on the relevant function spaces. Its proof relies on elementary Sobolev theory; see, for instance, \cite{Leoni}.

\begin{lemma} \label{lem:composition-small-deformation}
    Let $\kappa\in(0,1)$, let $0\leq s\leq1$, and suppose that $\tau\in W^{1,\infty}(\R^2;\R^2)$ satisfies $\norm{D\tau}_{L^\infty}\leq\kappa$. For $0\leq t\leq1$, set $F_t\coloneqq \Id-t\tau$. Then $F_t$ is a bi-Lipschitz homeomorphism of $\R^2$ onto itself. Moreover, the composition operators $V_tf\coloneqq f\circ F_t$ and $V_t^{-1}f=f\circ F_t^{-1}$ are bounded on $H^s(\R^2)$, with
    \[\norm{V_t}_{H^s\to H^s}\leq(1-t\kappa)^{-1}\leq (1-\kappa)^{-1}\]
    and
    \begin{equation*}
        \norm{V_t^{-1}}_{H^s\to H^s}\leq(1+t\kappa)^{1-s/2}(1-t\kappa)^{-s/2}\leq(1+\kappa)^{1-s/2}(1-\kappa)^{-s/2},
    \end{equation*}
    \begin{equation*}
        \norm{V_1}_{H^s\to H^s}\leq(1-\kappa)^{-1},\qquad \norm{V_1^{-1}}_{H^s\to H^s}\leq(1+\kappa)^{1-s/2}(1-\kappa)^{-s/2}.
    \end{equation*}
\end{lemma}

\begin{proof} 
    We work with the Lipschitz representative of $\tau$. Fix $0\leq t\leq1$ and write $\delta\coloneqq t\kappa$, $a\coloneqq 1-\delta$, and $A\coloneqq 1+\delta$. For all $x,y\in\R^2$,
    \[a\abs{x-y}\leq\abs{F_t(x)-F_t(y)}\leq A\abs{x-y}.\]
    Thus $F_t$ is injective, Lipschitz, and has closed range. To see that it is onto, fix $y\in\R^2$ and consider the map $x\mapsto y+t\tau(x)$. This is a contraction with constant $\delta<1$, hence has a unique fixed point $x$, which is precisely the solution of $F_t(x)=y$. Therefore $F_t$ is a bi-Lipschitz homeomorphism of $\R^2$ onto itself, with $\operatorname{Lip}(F_t)\leq A$ and $\operatorname{Lip}(F_t^{-1})\leq a^{-1}$.

    For a.e. $x\in\R^2$, all singular values of $DF_t(x)=\Id-tD\tau(x)$ lie in $[a,A]$. In particular, we have
    \begin{equation*}
        a^2\leq \abs{\det DF_t(x)}\leq A^2 \quad \text{for a.e. } x \in \R^2.
    \end{equation*}
    The change-of-variables formula gives, for $f\in L^2(\R^2)$,
    \begin{equation*}
        \norm{V_tf}_{L^2}^2=\int_{\R^2}\abs{f(y)}^2\abs{\det DF_t(F_t^{-1}y)}^{-1}\dy\leq a^{-2}\norm{f}_{L^2}^2,
    \end{equation*}
    and hence
    \[\norm{V_t}_{L^2\to L^2}\leq a^{-1}.\]
    Similarly,
    \begin{equation*}
        \norm{V_t^{-1}f}_{L^2}^2=\int_{\R^2}\abs{f(y)}^2 \abs{\det DF_t(y)} \dy\leq A^2\norm{f}_{L^2}^2,
    \end{equation*}
    so that
    \[\norm{V_t^{-1}}_{L^2\to L^2}\leq A.\]
    We next estimate the $H^1$ norms. First, note that if $s_1(x)\leq s_2(x)$ are the singular values of $DF_t(x)$, then $a\leq s_1(x)\leq s_2(x)\leq A$ and
    \begin{equation*}
        \frac{\abs{DF_t(x)^\top v}^2}{\abs{\det DF_t(x)}}\leq\frac{s_2(x)}{s_1(x)}\abs{v}^2\leq\frac{A}{a}\abs{v}^2, \qquad v \in \R^2.
    \end{equation*}
    For $f\in\calS(\R^2)$, the chain rule and the change of variables $y=F_t(x)$ then yield
    \begin{equation*}
        \norm{\nabla(f\circ F_t)}_{L^2}^2=\int_{\R^2}\frac{\abs{DF_t(F_t^{-1}y)^\top \nabla f(y)}^2}{\abs{\det DF_t(F_t^{-1}y)}}\dy\leq\frac{A}{a}\norm{\nabla f}_{L^2}^2.
    \end{equation*}
    Since $A/a\leq a^{-2}$, together with the $L^2$ estimate this implies
    \[\norm{V_tf}_{H^1}\leq a^{-1}\norm{f}_{H^1}.\]
    For the inverse composition operator, again by the chain rule and change of variables,
    \begin{equation*}
        \norm{\nabla(f\circ F_t^{-1})}_{L^2}^2=\int_{\R^2}\abs{DF_t(y)^{-\top} \nabla f(y)}^2\abs{\det DF_t(y)}\dy\leq\frac{A}{a}\norm{\nabla f}_{L^2}^2.
    \end{equation*}
    Since $A^2\leq A/a$, the $L^2$ estimate gives
    \begin{equation*}
        \norm{V_t^{-1}f}_{H^1}\leq\left(\frac{A}{a}\right)^{1/2}\norm{f}_{H^1}.
    \end{equation*}
    By density, both $H^1$ bounds extend from $\calS(\R^2)$ to all of $H^1(\R^2)$.

    The bounds for intermediate exponents $0 \le s \le 1$ then follow by complex interpolation between $L^2(\R^2)$ and $H^1(\R^2)$: for $0\leq s\leq1$ we thus obtain
    \[\norm{V_t}_{H^s\to H^s}\leq a^{-1}=(1-t\kappa)^{-1},\]
    while
    \begin{equation*}
        \norm{V_t^{-1}}_{H^s\to H^s}\leq A^{1-s}\left(\frac{A}{a}\right)^{s/2}=A^{1-s/2}a^{-s/2}=(1+t\kappa)^{1-s/2}(1-t\kappa)^{-s/2}.
    \end{equation*}
\end{proof}

\section{Architecture-blind stability from input regularity}\label{sec:architecture-blind-stability}

Before exploiting the geometry of the wave-packet transform, we record a general stability mechanism that is entirely independent of the feature extractor architecture. In accordance with the decoupling method \cite{WiatowskiBolcskei2018}, it separates the problem into continuity of the deformation action on the input class and nonexpansiveness of the feature map.

\begin{proposition}\label{prop:architecture-blind-deformation-stability}
    Let $\kappa\in(0,1)$, let $m \colon \R^2\to[1,\infty)$ be measurable, and set
    \begin{equation*}
        H_m(\R^2)\coloneqq\set*{f\in L^2(\R^2)\given m\wh f\in L^2(\R^2)},\qquad \norm{f}_{H_m}\coloneqq\norm{m\wh f}_{L^2}.
    \end{equation*}
    Suppose that
    \begin{equation*}
        m_\ast(R)\coloneqq\operatorname*{ess\,inf}_{\abs{\xi}>R}m(\xi)\longrightarrow\infty\quad\text{as }R\to\infty.
    \end{equation*}
    Let $\calH$ be a Hilbert space, and let $\calT \colon L^2(\R^2)\to\calH$ be nonexpansive, namely
    \begin{equation*}
        \norm{\calT f-\calT g}_{\calH}\leq\norm{f-g}_{L^2} \qquad\text{for all }f,g\in L^2(\R^2).
    \end{equation*}
    Then there exists an increasing, concave, continuous function $h\colon [0,\infty)\to[0,\infty)$ with $h(0)=0$ such that
    \begin{equation*}
        \norm{\calT(L_\tau f)-\calT f}_{\calH}\leq h\bigl(\norm{\tau}_{L^\infty}\bigr)\norm{f}_{H_m}
    \end{equation*}
    for every $f\in H_m$ and every $\tau\in W^{1,\infty}(\R^2;\R^2)$ satisfying $\norm{D\tau}_{L^\infty}\leq\kappa$. More precisely, one may take
    \begin{equation*}
        h(w)\coloneqq\frac{2\pi}{1-\kappa}\inf_{R\geq1}\left(wR+\frac{1}{m_\ast(R)}\right).
    \end{equation*}
\end{proposition}

\begin{proof}
    Fix $R\geq1$ and write
    \begin{equation*}
        f_{\leq R}\coloneqq\calF^{-1}\bigl(\one_{B_R(0)}\wh f\bigr),\qquad f_{>R}\coloneqq f-f_{\leq R}.
    \end{equation*}
    For $0\leq t\leq1$ we set $F_t\coloneqq\Id-t\tau$. In view of Lemma~\ref{lem:composition-small-deformation} each $F_t$ is bi-Lipschitz and the corresponding composition operator $V_tg=g\circ F_t$ satisfies
    \[\norm{V_t}_{L^2\to L^2}\leq(1-t\kappa)^{-1}\leq(1-\kappa)^{-1}.\]
    Since $f_{\leq R}\in H^1(\R^2)$, the Sobolev chain rule gives
    \begin{equation*}
        f_{\leq R}(F_1(x))-f_{\leq R}(x)=-\int_0^1\tau(x)\cdot\nabla f_{\leq R}(F_t(x))\dt
    \end{equation*}
    for almost every $x\in\R^2$. Therefore, by Minkowski's inequality, Lemma~\ref{lem:composition-small-deformation} and the Fourier support of $f_{\leq R}$, we obtain
    \begin{equation*}
        \norm{L_\tau f_{\leq R}-f_{\leq R}}_{L^2} \leq\frac{\norm{\tau}_{L^\infty}}{1-\kappa}\norm{\nabla f_{\leq R}}_{L^2} \leq\frac{2\pi R}{1-\kappa}\norm{\tau}_{L^\infty}\norm{f}_{H_m}.
    \end{equation*}
    On the other hand, since
    \begin{equation*}
        \norm{f_{>R}}_{L^2} \leq\frac{1}{m_\ast(R)}\norm{f}_{H_m},\qquad \norm{L_\tau f_{>R}}_{L^2} \leq\frac{1}{(1-\kappa)m_\ast(R)}\norm{f}_{H_m},
    \end{equation*}
    it follows that
    \begin{equation*}
        \norm{L_\tau f-f}_{L^2} \leq\frac{2\pi}{1-\kappa}\left(R\norm{\tau}_{L^\infty}+\frac{1}{m_\ast(R)}\right)\norm{f}_{H_m}.
    \end{equation*}
    Nonexpansiveness of $\calT$ and optimization over $R\geq1$ prove the claim. The function $h$ is finite and increasing, and it is concave as the pointwise infimum of affine functions of $w$; it is therefore continuous on $(0,\infty)$. Moreover, $h(0)=0$ by the assumption $m_\ast(R)\to\infty$, and the same assumption shows that $h(w)\to0$ as $w\downarrow0$.
\end{proof}

\begin{remark}[Arbitrarily weak high-frequency regularity]\label{rem:arbitrarily-weak-input-regularity}
    The preceding proposition applies in particular to weights with arbitrarily slow growth at infinity. For example, if
    \[m(\xi)=\bigl(\log(e+\abs{\xi})\bigr)^\sigma,\qquad \sigma>0,\]
    then choosing $R=w^{-1/2}$ for $0<w\leq\frac12$ gives
    \begin{equation*}
        h(w)\leq C_{\sigma}\left(\log\left(e+\frac1w\right)\right)^{-\sigma}.
    \end{equation*}
    It is then established that even a logarithmic high-frequency penalty suffices to obtain a uniform deformation modulus for every nonexpansive feature extractor.
\end{remark}

\begin{corollary}\label{cor:sobolev-decoupling-bound}
    Let $s>0$, and let $\calT \colon L^2(\R^2)\to\calH$ be nonexpansive. There exists $C>0$ such that
    \begin{equation*}
        \norm{\calT(L_\tau f)-\calT f}_{\calH}\leq C \norm{\tau}_{L^\infty}^{\min\{s,1\}}\norm{f}_{H^s}
    \end{equation*}
    whenever $f\in H^s(\R^2)$, $\tau\in C_b^1(\R^2;\R^2)$, $\norm{D\tau}_{L^\infty}\leq\frac12$, and $\norm{\tau}_{L^\infty}\leq1$. The exponent $\min\{s,1\}$ is optimal for the corresponding input-space estimate: for the constant deformation $\tau\equiv y$,
    \begin{equation*}
        \norm{L_y-\Id}_{H^s\to L^2}\asymp\abs{y}^{\min\{s,1\}}\qquad\text{as }\abs{y}\downarrow0.
    \end{equation*}
\end{corollary}
\begin{proof}
    The change-of-variables estimate used above and the fundamental theorem of calculus give, for some $C>0$,
    \begin{equation*}
        \norm{L_\tau-\Id}_{L^2\to L^2}\leq C,\qquad \norm{L_\tau-\Id}_{H^1\to L^2}\leq C \norm{\tau}_{L^\infty}.
    \end{equation*}
    Complex interpolation yields $\norm{L_\tau-\Id}_{H^s\to L^2}\leq C\norm{\tau}_{L^\infty}^s$ for $0<s<1$, while the case $s\geq1$ follows from the $H^1$ estimate and the embedding $H^s(\R^2)\hookrightarrow H^1(\R^2)$. Nonexpansiveness of $\calT$ gives then the first assertion. 

    For a constant translation $y\in\R^2$, Plancherel's theorem yields
    \begin{equation*}
        \norm{L_y-\Id}_{H^s\to L^2}=\sup_{\xi\in\R^2}\frac{\abs{e^{-2\pi\iu y\cdot\xi}-1}}{(1+\abs{\xi}^2)^{s/2}}.
    \end{equation*}
    The claimed upper estimate then follows in the case where $0<s\leq1$ from $\abs{e^{-\iu t}-1}\lesssim\min\{1,\abs{t}\}$, and in the case $s>1$ again from the embedding $H^s(\R^2)\hookrightarrow H^1(\R^2)$. For $0<s\leq1$, the matching lower estimate is obtained by taking $\xi=(2\abs{y})^{-1}y/\abs{y}$; for $s\geq1$, it follows by taking $y/\abs{y}$.
\end{proof}

\begin{remark}[Extension to higher dimensions]
    Although the discussion in this section has been formulated on $\R^2$ for consistency with the remainder of the paper, none of the arguments is dimension-specific. The results extend verbatim to $\R^d$, $d\geq1$, with constants that may additionally depend on the dimension.
\end{remark}

\section{Wave-packet scattering filters}
\label{sec:filters}

Let us introduce the setting from which relevant scattering filters are drawn. In short, we start by outlining a frequency covering structure supplying the geometry, and then place a Parseval scattering filter bank on it.

\subsection{The covering \texorpdfstring{$\calQ^{(\alpha,\beta)}$}{Q alpha beta}}

We follow the wave-packet frequency covering introduced by Bytchenkoff and Voigtlaender in \cite{BytchenkoffVoigtlaender2020}, with minor modifications as needed below. First, fix once and for all parameters
\[0\leq \beta\leq \alpha\leq 1, \qquad \varepsilon_0\in(0,1/32),\]
and set
\begin{equation*}
    Q=(-\varepsilon_0,1+\varepsilon_0)\times(-1-\varepsilon_0,1+\varepsilon_0), \qquad P=[0,1]\times[-1,1].
\end{equation*}
The set $Q$ is the fixed open model rectangle used to generate the covering, while $P$ is the companion closed core rectangle. The small parameter $\varepsilon_0$ is not optimized here, but it is chosen to guarantee that the slightly enlarged normalized rectangles have the intended overlap pattern.

For $j\in\N$ we define
\begin{equation*}
    m_j^{\max} \coloneqq \left\lceil 2^{(1-\alpha)j-1}\right\rceil, \qquad \ell_j^{\max} \coloneqq \left\lceil 10\cdot 2^{(1-\beta)j}\right\rceil.
\end{equation*}
Denote the high-frequency index set in this setting by
\begin{equation*}
    I_\ast=I_\ast^{(\alpha,\beta)} \coloneqq \set*{(j,m,\ell)\in \N\times \N_0\times \N_0 \given m\leq m_j^{\max},\ \ell\leq \ell_j^{\max}}.
\end{equation*}
In particular, at generation $j$ we also set
\begin{equation*}
    I_{\ast,j}\coloneqq \set*{(j,m,\ell)\in I_\ast \given m \le  m_j^{\max},\ \ell \le  \ell_j^{\max}}.
\end{equation*}
For every $m \in \N_0$ such that $m\leq m_j^{\max}$ we introduce
\begin{equation*}
    A_j= \begin{pmatrix} 2^{\alpha j}&0\\ 0&2^{\beta j} \end{pmatrix}, \qquad c_{j,m}=\binom{2^{j-1}+m2^{\alpha j}}{0}.
\end{equation*}
Moreover, for every $\ell\in \N_0$ such that $\ell\leq \ell_j^{\max}$, set
\begin{equation*}
    \phi_j \coloneqq \frac{\pi}{10}\,2^{(\beta-1)j}, \qquad \theta_{j,\ell} \coloneqq 2\ell\phi_j,
\end{equation*}
and let
\begin{equation*}
    R_{j,\ell}=\begin{pmatrix} \cos\theta_{j,\ell}&-\sin\theta_{j,\ell}\\ \sin\theta_{j,\ell}&\cos\theta_{j,\ell} \end{pmatrix}
\end{equation*}
denote the rotation by $\theta_{j,\ell}$. For each $i=(j,m,\ell)\in I_\ast$ we define
\begin{equation*}
    Q_i=Q_{j,m,\ell} \coloneqq R_{j,\ell}(A_jQ+c_{j,m}), \qquad P_i=P_{j,m,\ell} \coloneqq R_{j,\ell}(A_jP+c_{j,m}).
\end{equation*}
Finally, set
\[Q_0 \coloneqq B_4(0), \qquad P_0 \coloneqq B_3(0),\]
and denote the full index set of the covering by
\[I=I^{(\alpha,\beta)}\coloneqq \{0\} \cup I_\ast.\]
The family
\[\bvQ\coloneqq (Q_i)_{i\in I}\]
is called an \textit{$(\alpha,\beta)$-wave-packet covering} of $\R^2$. Let us record some properties of the covering $\bvQ$ proved in \cite[Lemmas 3.2, 4.1, and 4.2]{BytchenkoffVoigtlaender2020} useful for our purposes. 
\begin{lemma}\label{lem:BV-covering-basic-properties}
    Let $0\leq\beta\leq\alpha\leq1$, and let $\bvQ= (Q_i)_{i\in I}$ and $\calP^{(\alpha,\beta)}=(P_i)_{i\in I}$ be the BV covering families. Then:
    \begin{enumerate}
        \item[(a)] The families $\calP^{(\alpha,\beta)}$ and $\calQ^{(\alpha,\beta)}$ cover the whole frequency plane:
        \begin{equation*}
            \R^2=P_0\cup\bigcup_{j=1}^{\infty}\bigcup_{m=0}^{m_j^{\max}}\bigcup_{\ell=0}^{\ell_j^{\max}}P_{j,m,\ell}=Q_0\cup\bigcup_{j=1}^{\infty}\bigcup_{m=0}^{m_j^{\max}}\bigcup_{\ell=0}^{\ell_j^{\max}}Q_{j,m,\ell}.
        \end{equation*}
        \item[(b)] The covering $\calQ^{(\alpha,\beta)}$ has uniformly finite overlap. More precisely,
        \begin{equation*}
            \calN (\calQ^{(\alpha,\beta)}) \coloneqq\sup_{i\in I^{(\alpha,\beta)}}\#\set*{i'\in I^{(\alpha,\beta)}\given Q_i\cap Q_{i'}\neq\varnothing}<\infty.
        \end{equation*}
        \item[(c)] If $(j,m,\ell),(j',m',\ell')\in I_\ast$, then
        \begin{equation*}
            Q_{j,m,\ell}\cap Q_{j',m',\ell'}=\varnothing\qquad\text{unless}\qquad \abs{j-j'}\leq3.
        \end{equation*}
        \item[(d)] Let $(j,m,\ell)\in I_\ast$ and $\xi\in Q_{j,m,\ell}$. Then
        \begin{equation*}
            2^{j-2}<2^{j-1}+2^{\alpha j}(m-\varepsilon_0)\leq\abs{\xi}\leq2^{j-1}+2^{\alpha j}(m+2+2\varepsilon_0)\leq2^j+2^{\alpha j}(3+2\varepsilon_0)<2^{j+3}.
        \end{equation*}
        In particular, the high-frequency part of the covering has a fixed gap from the origin:
        \[Q_i\cap B_{1/2}(0)=\varnothing,\qquad i\in I_\ast.\]
    \end{enumerate}
\end{lemma}

Informally, at generation $j$ the tiles $Q_{j,m,\ell}$ are located at distance $\sim 2^j$ from the origin and have radial length comparable to $2^{\alpha j}$, transverse width comparable to $2^{\beta j}$, and angular step comparable to $2^{(\beta-1)j}$. Therefore, $\alpha$ controls the radial resolution, while $\beta$ controls the directional resolution, with smaller values of $\beta$ corresponding to finer angular localization. The wavelet corner is $(\alpha,\beta)=(1,1)$, whereas the genuinely directional wave-packet regimes are characterized by $\beta<1$.

The covering $\bvQ$ is almost structured \cite[Section 5]{BytchenkoffVoigtlaender2020}; specifically, this means that each tile is represented through an affine chart that satisfies certain additional properties \cite[Definition 2.1]{BytchenkoffVoigtlaender2020}. The most important for our purposes is the fact that, for every $i\in I_\ast$,
\begin{equation*}
    Q_i=\aff_i(Q), \qquad P_i=\aff_i(P),\qquad \aff_i \coloneqq  T_i \cdot + b_i,
\end{equation*}
where the affine map $\aff_i$ is described by
\begin{equation*}
    T_i=T_{j,m,\ell}=R_{j,\ell}A_j\in\mathrm{GL}(2,\R), \qquad b_i=b_{j,m,\ell}=R_{j,\ell}c_{j,m} \in\R^2.
\end{equation*}
Equivalently, the correspondence $\xi\mapsto T_i\xi+b_i$ maps the fixed model coordinates to the physical frequency coordinates of the tile $Q_i$.

\subsection{Wave-packet smoothness spaces}

In this section, we briefly recall how the wave-packet covering gives rise to the associated family of wave-packet smoothness spaces. 

Let $\Phi=(\varphi_i)_{i\in I}$ be a regular partition of unity subordinate to the covering $\bvQ=(Q_i)_{i\in I}$ in the sense of \cite[Definition 2.2]{BytchenkoffVoigtlaender2020}. This means that each $\varphi_i$ belongs to $C_c^\infty(\R^2)$ and
\begin{equation*}
    \supp\varphi_i\subset Q_i,\qquad \sum_{i\in I}\varphi_i\equiv 1\quad\text{on }\R^2.
\end{equation*}
Note that the sum is locally finite because $\bvQ$ is admissible in the sense of Lemma~\ref{lem:BV-covering-basic-properties}(b). With the same notation from above for the representation of the tiles through the specified affine chart, we write $\varphi_0^\natural\coloneqq \varphi_0$ as well as
\begin{equation*}
    \varphi_i^\natural(\xi) \coloneqq \varphi_i(T_i\xi+b_i),\qquad \xi\in\R^2, i\in I_\ast.
\end{equation*}
The regularity condition is
\begin{equation*}
    \sup_{i\in I}\|\partial^\gamma\varphi_i^\natural\|_{L^\infty}<\infty\qquad\text{for every }\gamma\in\mathbb N_0^2.
\end{equation*}
In particular, the smoothness degree of the partition is measured after pulling each cutoff back to the fixed model scale of its tile. As a result, although derivatives may grow in frequency variables as the tiles become thinner, they remain uniformly controlled in the normalized affine coordinates. 

In light of \cite[Lemma~6.1]{BytchenkoffVoigtlaender2020}, fix $s\in \R$ and $0<p,q\leq\infty$ and consider the scale weight
\begin{equation*}
    w_i^s=\begin{cases}2^{js}& \text{ if } i=(j,m,\ell)\in I_\ast,\\1 & \text{ if } i=0.\end{cases}
\end{equation*}
For a sequence $a=(a_i)_{i\in I}$ and $0<  q \le \infty$, we set
\begin{equation*}
    \norm{a}_{\ell^q_{w^s}} \coloneqq \left(\sum_{i\in I}\abs{w_i^s a_i}^q\right)^{1/q}, \qquad \|a\|_{\ell^\infty_{w^s}} \coloneqq \sup_{i\in I}\abs{w_i^s a_i}.
\end{equation*}
Following \cite[Definition~6.2]{BytchenkoffVoigtlaender2020}, the decomposition space $W_s^{p,q}(\alpha,\beta) \coloneqq \calD (\bvQ,L^p,\ell^q_{w^s})$ associated with the covering $\bvQ$ contains all tempered distributions $g\in\calS'(\R^2)$ such that
\begin{equation*}
    \left(\|\calF ^{-1}(\varphi_i\widehat g)\|_{L^p}\right)_{i\in I}\in\ell^q_{w^s}.
\end{equation*}
Equivalently,
\begin{equation*}
    \norm{g}_{\calD(\bvQ,L^p,\ell^q_{w^s})} \coloneqq \norm{\left(\norm{\calF ^{-1}(\varphi_i\widehat g)}_{L^p}\right)_{i\in I}}_{\ell^q_{w^s}}<\infty.
\end{equation*}
Different regular partitions of unity subordinate to the same almost structured covering induce equivalent norms, so the notation does not record the particular choice of $\Phi$. For $0<p,q\leq\infty$, this construction defines a quasi-Banach space in general, in fact a Banach space whenever $1\leq p,q\leq\infty$. Our main interest lies in the Hilbertian case $p=q=2$, where we have the following collapse to the Sobolev scale: 

\begin{lemma}[{\cite[Theorem~6.6]{BytchenkoffVoigtlaender2020}}]\label{lem:hilbertian-identification}
    For every $s\in\R$ we have $W_s^{2,2}(\alpha,\beta)=H^s(\R^2)$ with equivalent norms.
\end{lemma}

\subsection{Scattering filter banks}

Let us now introduce an additional analytic layer on top of the decomposition-space framework discussed so far: the construction in \cite{BytchenkoffVoigtlaender2020} fixes the geometry, but in order to build a scattering structure we need to design Fourier multipliers supported in the tiles. The underlying philosophy of decomposition spaces is that they are determined, up to equivalent norms, by the geometry of the covering rather than by a particular choice of subordinate filters. The picture is generally quite different when it comes to scattering transforms, in light of the sensitivity of extracted scattering features to filter phases and exact profiles. To resolve this dichotomy, let us first clarify what we mean by a scattering filter bank subordinate to $\bvQ$.

\begin{definition} \label{def:fixed-channel-bank}
    Let $\bvQ=(Q_i)_{i\in I}$ be the wave-packet covering introduced above. A \textit{wave-packet scattering bank} subordinate to $\bvQ$ consists of a low-pass filter $\chi\in L^1(\R^2)\cap L^2(\R^2)$ and high-pass filters $\Psi=(\psi_i)_{i\in I_\ast}$ in $L^1(\R^2)\cap L^2(\R^2)$ with Fourier multipliers $m_i\coloneqq \widehat\psi_i$ such that
    \begin{equation}\label{eq:LPcondition} 
        \abs{\widehat\chi(\xi)}^2+\sum_{i\in I_\ast}\abs{m_i(\xi)}^2=1\qquad\text{for a.e. }\xi\in\R^2.
    \end{equation}
    Subordination to the covering here means that the filters are supported in the covering tiles:
    \begin{equation}\label{eq:supp-sub-transfer}
        \supp \wh{\chi} \subset Q_0, \qquad \supp m_i\subset Q_i \quad \text{for all } i\in I_\ast.
    \end{equation}
\end{definition}

Let us now recall the main aspects of the scattering architecture. We regard paths of length $n\geq1$ over the alphabet $I_\ast$ as tuples $p=(i_1,\ldots,i_n)\in I_\ast^n$. Further, set $I_\ast^0\coloneqq\{e\}$, where $e$ denotes the empty path, and write
\[I_\ast^\ast\coloneqq\bigcup_{n\geq0}\PathDepth{n}.\]
For $f\in L^2(\R^2)$, we define $U[e]f\coloneqq f$; if $p=(i_1,\ldots,i_n)\in I_\ast^n$ with $n\geq1$, we set
\begin{equation*}
    U[p]f\coloneqq U[i_n]\cdots U[i_1]f,\qquad U[i]g\coloneqq \abs{g*\psi_i}.
\end{equation*}
Young's convolution inequality ensures that $U[p]f\in L^2(\R^2)$. For a path set $P\subset I_\ast^\ast$, we write
\begin{equation*}
    U[P]f\coloneqq(U[p]f)_{p\in P},\qquad\ScatSet{P}f\coloneqq(U[p]f*\chi)_{p\in P}, \qquad U_nf\coloneqq U[\PathDepth{n}]f.
\end{equation*}
Finally, setting
\begin{equation*}
    \PathsLe{N}\coloneqq\bigcup_{0\leq n\leq N}\PathDepth{n},
\end{equation*}
we can introduce the full, truncated, and first-layer scattering transforms respectively:
\begin{equation*}
    \Scat f\coloneqq\ScatSet{\PathsAll}f,\qquad\ScatLe{N}f\coloneqq\ScatSet{\PathsLe{N}}f,\qquad\ScatOne f\coloneqq\ScatSet{\PathDepth{1}}f.
\end{equation*}
Equivalently,
\begin{equation*}
    \Scat f=(U[p]f*\chi)_{p\in\PathsAll},\qquad \ScatOne f=(\abs{f*\psi_i}*\chi)_{i\in I_\ast}.
\end{equation*}
All scattering output norms are $\lp$ norms over the displayed path set.

The Littlewood--Paley condition \eqref{eq:LPcondition} on the Fourier multipliers of the filters guarantees that the filters form a semi-discrete Parseval frame. This means that every $h\in L^2(\R^2)$ satisfies the energy balance 
\begin{equation*}
    \norm{h*\chi}_{L^2}^2+\sum_{i\in I_\ast}\norm{h*\psi_i}_{L^2}^2=\norm h_{L^2}^2.
\end{equation*}
In particular, the one-layer map
\[\calW \colon h\longmapsto (h*\chi,(h*\psi_i)_{i\in I_\ast})\]
is an isometry $L^2(\R^2) \to \ell^2(I;L^2(\R^2))$. Consequently, replacing each high-pass output by its modulus gives the estimate
\begin{equation*}
    \norm{\abs{f*\psi_i}-\abs{g*\psi_i}}_{L^2}\leq\norm{(f-g)*\psi_i}_{L^2},\qquad i\in I_\ast,
\end{equation*}
showing that the scattering cascade is nonexpansive in $L^2$ along descendant subtrees. In particular, the first layer is nonexpansive:
\begin{equation}\label{eq:first-layer-nonexpansive-transfer}
    \norm{\ScatOne f-\ScatOne g}_{\lp}\leq \norm{f-g}_{L^2}.
\end{equation}

\subsection{A smooth Parseval model class} \label{subsec:designed-model}

For the sake of completeness, let us show that the setting built above is not trivial by providing an explicit construction of a wave-packet scattering bank subordinate to the covering $\bvQ$. The construction is deliberately not canonical and allows for an arbitrary phase in the generator to emphasize that, for scattering
transforms, the filter bank contains more information than just the underlying covering.

We define the high-pass filter bank starting from a single generator in Fourier domain: fix an arbitrary function $\gamma\in C_c^{\infty}(\R^2)$ with $0\leq \abs{\gamma}\leq 1$ and
\[|\gamma|\equiv 1 \text{ on } P, \qquad \supp\gamma\subset Q.\]
For every $i=(j,m,\ell)\in I_\ast$ and every $\xi\in \R^2$ define
\begin{equation*}
    \gamma_i(\xi)=\gamma_{j,m,\ell}(\xi)\coloneqq \gamma\!\left(A_j^{-1}\bigl(R_{j,\ell}^{-1}\xi-c_{j,m}\bigr)\right).
\end{equation*}
Thus,
\begin{equation*}
    \abs{\gamma_{i}}\equiv 1 \text{ on } P_{i},\qquad\supp\gamma_{i}\subset Q_{i}.
\end{equation*}
Next, fix any low-pass $b_0\in C_c^\infty(\R^2)$ such that $0\leq \abs{b_0}\leq 1$ and
\begin{equation*}
    \abs{b_0}\equiv 1 \text{ on } P_0=B_3(0),\qquad\supp b_0\subset Q_0=B_4(0).
\end{equation*}
At this stage, we have already defined a system $(b_0,(\gamma_i)_{i\in I_\ast})$ of compactly supported smooth functions in Fourier domain, which is subordinate to $\bvQ$. It remains to normalize the system to guarantee that the Littlewood--Paley condition \eqref{eq:LPcondition} is satisfied. 
To this end, define
\begin{equation*}
    S(\xi)\coloneqq \abs{b_0(\xi)}^2+\sum_{i\in I_\ast}|\gamma_{i}(\xi)|^2 \qquad \text{for all } \xi\in\R^2.
\end{equation*}
Since the covering is locally finite, this sum is locally finite. As a result of the construction and by Lemma~\ref{lem:BV-covering-basic-properties}, we have the pointwise bounds
\begin{equation*}
    1\leq S(\xi)\leq \calN\left(\bvQ\right)\qquad\text{for all } \xi\in\R^2.
\end{equation*}
This justifies the following normalization. We define the output-generating low-pass filter for the scattering transform on the Fourier side by
\begin{equation*}
    \widehat{\chi}(\xi)\coloneqq \frac{b_0(\xi)}{\sqrt{S(\xi)}} \qquad\text{for all } \xi\in\R^2.
\end{equation*}
Likewise, we obtain our family of high-pass filters on the Fourier side by letting
\begin{equation*}
    m_i(\xi)=\widehat{\psi}_{i}(\xi)\coloneqq \frac{\gamma_{i}(\xi)}{\sqrt{S(\xi)}} \qquad\text{for all } \xi\in\R^2, i\in I_\ast.
\end{equation*}
After normalization it still holds
\begin{equation*}
    \supp\widehat{\chi}\subset Q_0=B_4(0),\qquad\supp\widehat{\psi}_{i}\subset Q_{i}, \quad i\in I_\ast,
\end{equation*}
and by construction
\begin{equation*}
    \abs{\widehat{\chi}(\xi)}^2+\sum_{i\in I_\ast}\abs{m_{i}(\xi)}^2=1\qquad\text{for all } \xi\in\R^2.
\end{equation*}
Writing $\Psi\coloneqq (\psi_{i})_{i\in I_\ast}$ for the high-pass filters,
the tuple $(\chi,\Psi)$ forms a wave-packet scattering bank in the sense of Definition~\ref{def:fixed-channel-bank}. Moreover, since $b_0$ and all $\gamma_{i}$, $i\in I_\ast$ are smooth and compactly supported, and
since $S$ is smooth and bounded away from zero, all Fourier-side filters are again smooth and compactly supported.
Thus,
\[\chi,\psi_{i}\in\calS(\R^2)\subset L^1(\R^2)\cap L^2(\R^2).\]
\begin{remark}[Loss of exact covariance under Parseval normalization]
    Before normalization, the filters $\gamma_i$, $i=(j,m,\ell)\in I_\ast$ are generated from one
    single prototype by rotation, anisotropic dilation, and translation. In
    particular, for fixed $j$ and $m$, the angular copies are exactly related by
    rotations. After the Parseval normalization by $S^{-1/2}$, this exact covariance
    is generally lost: one only has
    \begin{equation*}
        \widehat\psi_{j,m,\ell}(U_{\ell,\ell'}\xi) = \left( \frac{S(\xi)}{S(U_{\ell,\ell'}\xi)} \right)^{1/2} \widehat\psi_{j,m,\ell'}(\xi),
    \end{equation*}
    where
    \[U_{\ell,\ell'}\coloneqq R_{j,\ell}R_{j,\ell'}^{-1}.\]
    The normalizing factor $S$ is uniformly bounded above and below, but it is not
    in general invariant under the angular rotations generating the filter bank.
    Thus the construction is Parseval and geometrically adapted, but not exactly
    generated by a single covariant high-pass prototype after normalization.
\end{remark}

\section{Geometric instability mechanisms}
\label{sec:L2-negative}

Consider from now on the setting introduced in Section \ref{sec:filters}. 

\subsection{Transverse channel swaps}
We now record the geometric displacement argument, which is the key technical ingredient in the instability mechanism. The argument is carried out in the narrower transverse direction of the frequency tiles: at frequency scale $2^j$, the transverse width of a tile is comparable to $2^{\beta j}$. Thus, one may separate the tile by applying a displacement of this order in a direction transverse to its radial direction. Since the carrier frequency itself has size comparable to $2^j$, such a transverse displacement is produced by a deformation of size $\asymp2^{-(1-\beta)j}$. 

\begin{lemma}\label{lem:BV-separation}
    Under the standing assumption $0\leq\beta\leq\alpha\leq1$, assume further that $\beta<1$. 

    There exist $j_0\in\N$, $r_0>0$, indices $i_j\in I_\ast$, and points $\xi_j\in P_{i_j}$, $j\geq j_0$, such that $\abs{\xi_j}\asymp2^j$ and
    \begin{equation}\label{eq:active-core-ball}
        B_{r_0}(\xi_j)\subset P_{i_j}\subset Q_{i_j}.
    \end{equation}
    Moreover, there are matrices $M_j\in\R^{2\times2}$ such that, for all $j\geq j_0$,
    \begin{equation}\label{eq:transverse-matrix-size}
        \norm{M_j}\asymp2^{-(1-\beta)j},\qquad \Id-M_j\in\mathrm{GL}(2,\R),
    \end{equation}
    and for every $0<\rho\leq r_0$ and all $i\in I$,
    \begin{equation}\label{eq:BV-separation}
        Q_i\cap B_\rho(\xi_j)\neq\varnothing\quad\Longrightarrow\quad Q_i\cap (\Id-M_j) B_\rho(\xi_j)=\varnothing.
    \end{equation}
\end{lemma}

\begin{proof}
    Throughout the proof, the symbols $\lesssim$, $\gtrsim$, and $\asymp$ refer to constants which are independent of the scale parameters, the indices of the covering, and the specific choice of radius $\rho\in (0,r_0]$.

    Let us first prove \eqref{eq:active-core-ball}.

    Fix $\xi_\ast \coloneqq (1/2,0)^\top \in \R^2$. We construct a sequence of radial interior cores in the direction $e_1$: if $m_j^{\max}$ denotes the largest radial index at scale $j$, choose $j_0\in \N$ so that
    \begin{equation*}
        0\leq \left\lfloor 2^{(1-\alpha)j-2}\right\rfloor\leq m_j^{\max} \quad \text{ for all } j\geq j_0.
    \end{equation*}
    For such $j$ then set
    \begin{equation*}
        m_j^{\prime} \coloneqq \left\lfloor 2^{(1-\alpha)j-2}\right\rfloor,\qquad i_j \coloneqq (j,m_j^{\prime},0),\qquad \xi_j \coloneqq \aff_{i_j}(\xi_\ast)=A_j\xi_\ast+c_{j,m_j^\prime}.
    \end{equation*}
    In particular, $i_j\in I_\ast$. Since $\ell=0$, the point $\xi_j$ lies on the positive $e_1$-axis, more precisely
    \begin{equation*}
        \xi_j=\left(2^{j-1}+\left(m_j^{\prime}+\frac12\right)2^{\alpha j}\right)e_1.
    \end{equation*}
    Let us write for convenience $\xi_j=\mu_je_1$, so that the choice of $m_j^{\prime}$ gives $\abs{\xi_j}=\mu_j\asymp2^j$. 
    Moreover, since $\xi_\ast\in P^\circ$, we can choose $r_0>0$ such that $B_{r_0}(\xi_\ast)\subset P$. Note that the linear part of $\aff_{i}$ is $R_{j,\ell}A_j$, whose least singular value is $2^{\beta j}\geq1$. Therefore,
    \begin{equation*}
        B_{r_0}(\aff_{i} (\xi_\ast))\subset \aff_{i}(B_{r_0}(\xi_\ast))\subset \aff_{i}(P)=P_i,
    \end{equation*}
    and choosing $i=i_j$ yields the claimed inclusions
    \begin{equation*}
        B_{r_0}(\xi_j)\subset P_{i_j}\subset Q_{i_j} \quad \text{ for all } j \ge j_0.
    \end{equation*}
    Let us now turn our attention to \eqref{eq:BV-separation}. After increasing $j_0$, if necessary, we may assume that $Q_0\cap B_{r_0}(\xi_j)=\varnothing$ for all $j\geq j_0$. Hence it remains to consider high-frequency indices $i=(k,m,\ell)\in I_\ast$.
    Recall that there are constants $0<c<C<\infty$ such that every high-frequency tile $Q_{k,m,\ell}$ is contained in the dyadic annulus
    \[\{\eta:c2^k\leq |\eta|\leq C2^k\}.\]
    As a result of this and since $\abs{\xi_j}\asymp2^j$, if a tile $Q_{k,m,\ell}$ intersects $B_\rho(\xi_j)$ (hence also the larger ball $B_{r_0}(\xi_j)$), then $\abs{k-j}\leq C_0$ for some constant $C_0>0$ independent of the specific choice $0<\rho\leq r_0$.
    Set $\Theta \coloneqq \theta_{k,\ell}$ and write points of $Q_{k,m,\ell}$ via the respective rotation matrix $R_\Theta$ as 
    $R_\Theta (x,t)^\top $
    with $x\gtrsim 2^k$ and $\abs{t}\lesssim 2^{\beta k}$. Choose then
    \[\eta=R_\Theta \binom{x}{t}\in Q_{k,m,\ell}\cap B_\rho(\xi_j),\]
    and note that we have $|\langle\eta,e_{2}\rangle|\leq \rho \leq r_0$, since $\xi_j$ lies on the positive $e_1$-axis. Therefore, we obtain
    \begin{equation*}
        \abs{x\sin\Theta}\leq |\langle\eta,e_{2}\rangle|+\abs{t\cos\Theta}\lesssim 2^{\beta k}.
    \end{equation*}
    Combining the previous estimate with $x\gtrsim 2^k$ and $\abs{k-j}\leq C_0$ gives
    $ \abs{\sin\Theta}\lesssim 2^{(\beta-1)j}$. 

    Let $\eta',\zeta'\in Q_{k,m,\ell}$. Before rotation their difference is a point in the plane with coordinates $(u,v)^\top \in \R^2$, where $\abs{u}\lesssim 2^{\alpha k}$ and $\abs{v}\lesssim 2^{\beta k}$. Hence
    \begin{equation*}
        |\langle\eta'-\zeta',e_{2}\rangle| \leq |u||\sin\Theta|+|v||\cos\Theta| \lesssim 2^{\alpha k}2^{(\beta-1)j}+2^{\beta k} \lesssim 2^{\beta j},
    \end{equation*}
    where in the last step we used $\abs{k-j}\leq C_0$ and $\alpha\leq1$. Consequently, there exists $C_1>0$ such that every tile $Q_{k,m,\ell}$ intersecting $B_\rho(\xi_j)$ satisfies
    \begin{equation*}
        \sup_{\eta,\zeta\in Q_{k,m,\ell}}|\langle\eta-\zeta,e_{2}\rangle|\leq C_1 2^{\beta j}.
    \end{equation*}
    We now design the deformation: set
    \begin{equation*}
        J \coloneqq \begin{pmatrix}0&-1\\1&0\end{pmatrix},\qquad M_j \coloneqq \pi_jJ,\qquad \pi_j \coloneqq L2^{-(1-\beta)j},
    \end{equation*}
    where $L>0$ will be fixed in a moment. We have $\norm{M_j}=\pi_j\asymp2^{-(1-\beta)j}$ and $\Id-M_j$ is invertible for every $j$, since $\det(\Id-\pi_jJ)=1+\pi_j^2$.

    Choose now $c_0>0$ such that $\abs{\xi_j}\geq c_02^j$ for all $j\geq j_0$, and choose $L$ so large that
    \[c_0L>C_1+3r_0+1.\]
    Increasing $j_0$ once more, if necessary, we may also assume that
    $\pi_j=L2^{-(1-\beta)j}\leq 1$ for all $j\geq j_0$.
    This does not affect any of the preceding conclusions.

    For the sake of contradiction, suppose that some tile $Q_i=Q_{k,m,\ell}$ meets both $B_\rho(\xi_j)$ and $(\Id-M_j) B_\rho(\xi_j)$. Choose
    \begin{equation*}
        p\in Q_i\cap B_\rho(\xi_j),\qquad q\in Q_i\cap (\Id-M_j) B_\rho(\xi_j).
    \end{equation*}
    Then $q=(\Id-M_j)\eta$ for some $\eta\in B_\rho(\xi_j)$. We have $J\xi_j=|\xi_j|e_{2}$ and thus, for all $j\geq j_0$,
    \begin{align}
        |\langle q-p,e_{2}\rangle|
        &=|\langle \eta-p,e_{2}\rangle-\pi_j\langle J\eta,e_{2}\rangle| \notag \\
        &\geq \pi_j|\xi_j|-\abs{\eta-\xi_j}-\abs{p-\xi_j}-\pi_j|\eta-\xi_j| \notag \\
        &\geq c_0L2^{\beta j}-3r_0.
    \end{align}
    On the other hand, since $p,q\in Q_i$, the transverse projection estimate above yields 
    \[|\langle q-p,e_2\rangle|\leq C_1 2^{\beta j}.\]
    Combining the two estimates, we obtain
    \[(c_0L-C_1)2^{\beta j}\leq 3r_0.\]
    This is impossible by the choice of $L$, since $2^{\beta j}\geq 1$. Consequently, we derived the intended contradiction, and the separation property follows.
\end{proof}

\begin{remark}[Transverse versus radial mechanisms] \label{rm:transverseVSradialDisplacement}
    The proof above exploits a \textit{transverse} displacement of the wave-packet geometry. Likewise, a similar geometric displacement argument in the \textit{radial} direction applies whenever $\alpha<1$: at frequency radius $2^j$ a tile has radial length comparable to $2^{\alpha j}$. Therefore, a deformation whose frequency action displaces a packet radially by $2^{\alpha j}$ can move a fixed packet core outside every channel that intersects the original one. 

    Although the displacement argument in the radial direction is technically somewhat more straightforward, the transverse formulation used above has two advantages. First, it covers the full admissible parameter range
    \[0\leq \beta\leq \alpha\leq 1,\qquad \beta<1,\]
    in which the instability mechanism applies; the excluded endpoint $(\alpha,\beta)=(1,1)$ is precisely the wavelet geometry, where stability is known. Second, the deformation required in the transverse construction has size only of order $2^{-(1-\beta)j}$. Whenever $\beta<\alpha$, this is strictly smaller than the corresponding radial deformation scale $2^{-(1-\alpha)j}$. Thus the transverse argument detects the sharper instability scale, in agreement with the critical Sobolev threshold $H^{1-\beta}$, which is necessary and sufficient for Lipschitz stability under the respective hypotheses of Corollaries~\ref{cor:modulus-obstruction} and~\ref{cor:full-depth-positive-endpoint}.
\end{remark}

\subsection{Transfer to scattering stability and linear commutator estimates}

Let us show how to turn the geometric channel swap phenomenon into a genuine scattering separation result. We will focus on the first layer output 
\begin{equation}\label{eq:first-layer}
    \ScatOne f=(\abs{f*\psi_i}*\chi)_{i\in I_\ast},
\end{equation}
with output norm denoted by $\norm{\cdot}_{\lp}$. Note that any lower bound obstructing stability at the level of the first layer automatically propagates to the full scattering transform, whenever the output norm dominates the contribution of the first-layer coefficients. In other words, if stability fails for the first-order coefficients, then it also fails for the scattering transform as a whole.

Throughout the remainder of this section, $(i_j,\xi_j)_{j\geq j_0}$, $r_0$, and $(M_j)_{j\geq j_0}$ denote, respectively, the active-core sequence, the active-core radius, and the transverse matrices furnished by Lemma~\ref{lem:BV-separation} and its proof. The basic hypothesis needed to convert the geometric channel swap into a quantitative lower bound for the scattering transform is the following.

\begin{assumption}[Active core condition (AC)]\label{ass:active-core}
    There exists
    $c_{\rm act}>0$ such that, for all $j\geq j_0$,
    \begin{equation*}
        \abs{m_{i_j}(\xi)}\geq c_{\rm act} \qquad \text{for all } \xi\in B_{r_0}(\xi_j).
    \end{equation*}
\end{assumption}

We record that the smooth Parseval model class from Section~\ref{subsec:designed-model} satisfies the active-core condition.

\begin{proposition}\label{prop:model-ACass}
    The smooth Parseval model class from Section~\ref{subsec:designed-model} satisfies \textup{(AC)}.  
\end{proposition}
\begin{proof}
    By construction, $\abs{m_i}\geq \calN(\bvQ)^{-1/2}$ on $P_i$, $i\in I_\ast$. Thus, the active core condition \textup{(AC)} follows directly from Lemma \ref{lem:BV-separation} with $\cact=\calN(\bvQ)^{-1/2}$.
\end{proof}

Let us briefly outline how this condition enables conversion of the geometric transposition principle from the previous section into a channel-swap mechanism for the scattering transform. Starting from a fixed, well-localized test packet in the Fourier domain, we construct a sequence $(f_j)_{j\geq j_0}$ by frequency shifts, so that each $f_j$ is localized in a small region where a suitable multiplier $m_{i_j}$ is active. By the \textup{(AC)} assumption, this multiplier can be inverted locally and stably on the active region. The deformation constructed in Lemma~\ref{lem:BV-separation} then moves this region into a different channel, thereby destroying the cancellation effects responsible for deformation stability and yielding a genuine instability.

The following auxiliary lemma records a basic consequence of the fact that $\chi$ is genuinely low-pass. 

\begin{lemma}\label{lem:TestPacket}
    Let $0\neq\varphi\in C_c^\infty(B_{r_0/4}(0))$. Then, with $h \coloneqq \calF^{-1}\varphi$,
    \begin{equation*}
        \norm{\abs{h}*\chi}_{L^2}>0.
    \end{equation*}
\end{lemma}

\begin{proof}
    Lemma~\ref{lem:BV-covering-basic-properties}, support subordination, and the Littlewood--Paley condition~\eqref{eq:LPcondition} imply that $\chi$ is genuinely low-pass. In particular, we have $\abs{\wh{\chi}(\xi)}\equiv1$ on $B_{1/2}(0)$. Since $\wh{\abs{h}}(0)=\norm{h}_{L^1}>0$ and $\wh{\abs{h}}$ is continuous, there exists a neighborhood of the origin on which $\wh{\abs{h}}$ is bounded away from zero. Combining this with Parseval's theorem concludes the proof.
\end{proof}

Under the active-core assumption \textup{(AC)} we can first establish the instability mechanism for linear deformation fields.

\begin{proposition}\label{prop:first-layer-separation}
    Under the standing assumption $0\leq\beta\leq\alpha\leq1$, assume further that $\beta<1$, that the support subordination condition \eqref{eq:supp-sub-transfer} holds, and that the active-core condition \textup{(AC)} is satisfied. Let $(M_j)_{j\geq j_0}$ be the transverse matrices supplied by Lemma \ref{lem:BV-separation}, and define 
    \begin{equation}\label{eq:linear-field}
        \ell_j(x) \coloneqq M_j^\top x,\qquad x\in\R^2.
    \end{equation}

    Then 
    \begin{equation}\label{eq:small-linear-def}
        \norm{D \ell_j}_{L^\infty}=\norm{M_j}\asymp 2^{-(1-\beta)j}\to0 \qquad\text{as }j\to\infty.
    \end{equation}
    
    Nevertheless, there exist functions $f_j\in L^2(\R^2)$ with compact Fourier support and $\norm{f_j}_{L^2}=1$, and a constant $c>0$, independent of $j$, such that
    \begin{equation}\label{eq:linear-first-layer-sep}
        \norm{\ScatOne(L_{\ell_j}f_j)-\ScatOne f_j}_{\lp}\geq c
    \end{equation}
    for all sufficiently large $j$.
\end{proposition}

\begin{proof}
    Fix any $0\neq\varphi\in C_c^\infty(B_{r_0/4}(0))$ and set $h \coloneqq \calF^{-1}\varphi$ as well as $a_\chi \coloneqq \norm{\abs{h}*\chi}_{L^2}$. Then define
    \begin{equation}
        q_j(\zeta)\coloneqq
        \begin{cases}
            \dfrac{\varphi(\zeta)}{m_{i_j}(\xi_j+\zeta)}, & \zeta\in B_{r_0}(0)\\[4pt]
            0, & \zeta\notin B_{r_0}(0),
        \end{cases}
        \qquad g_j\coloneqq\calF^{-1}q_j,\qquad f_j^0(x)\coloneqq e^{2\pi\iu\xi_j\cdot x}g_j(x).
    \end{equation}
    We have $\supp\widehat f_j^0\subset B_{r_0/4}(\xi_j)$.

    By Assumption~\ref{ass:active-core} and the Littlewood--Paley condition, the multiplier satisfies the pointwise bounds 
    \begin{equation*}
        \cact\leq \abs{m_{i_j}(\xi_j+\zeta)}\leq 1 \qquad \text{for all } \zeta \in \supp(\varphi)\subseteq B_{r_0/4}(0).
    \end{equation*}
    We thus obtain bounds that are uniform over all $j\geq j_0$,
    \begin{equation}\label{eq:uniform-norm-control-g_j}
        0<\norm{\varphi}_{L^2}\leq \norm{f_j^0}_{L^2}=\norm{g_j}_{L^2}\leq \cact^{-1} \, \norm{\varphi}_{L^2}.
    \end{equation}
    Set then, for all $j\geq j_0$,
    \begin{equation}\label{eq:def_f_j}
        f_j \coloneqq \dfrac{f_j^0}{\norm{f_j^0}_{L^2}}.
    \end{equation}
    By construction we have $\widehat{f_j^0*\psi_{i_j}}(\xi)=\varphi(\xi-\xi_j)$, hence
    \begin{equation*}
        f_j^0*\psi_{i_j}=e^{2 \pi \iu \xi_j\cdot x}h(x),\qquad \abs{f_j^0*\psi_{i_j}}=\abs{h},
    \end{equation*}
    and thus
    \[\norm{|f_j^0*\psi_{i_j}|*\chi}_{L^2}=a_\chi>0,\]
    where the positivity of the last expression is guaranteed by Lemma~\ref{lem:TestPacket}. 

    Let us now consider the linear deformation. Since $L_{\ell_j}F(x)=F((\Id-M_j^\top)x)$, we have
    \begin{equation}
        \widehat{L_{\ell_j}F}(\xi)=\abs{\det(\Id-M_j)}^{-1}\widehat F((\Id-M_j)^{-1}\xi),
    \end{equation}
    which implies
    \begin{equation}
        \supp\widehat{L_{\ell_j}f_j^0}\subset (\Id-M_j)B_{r_0/4}(\xi_j).
    \end{equation}
    Recall that $B_{r_0/4}(\xi_j)\subset B_{r_0}(\xi_j) \subset Q_{i_j}$, so Lemma \ref{lem:BV-separation} entails $Q_{i_j}\cap (\Id-M_j)B_{r_0/4}(\xi_j)=\varnothing$. Support subordination \eqref{eq:supp-sub-transfer} then implies $(L_{\ell_j}f_j^0)*\psi_{i_j}=0$, and the single active coordinate $i_j$ finally supplies the lower bound
    \begin{equation*}
        \norm{\ScatOne(L_{\ell_j}f_j^0)-\ScatOne f_j^0}_{\lp} \geq \norm{\abs{f_j^0*\psi_{i_j}}*\chi}_{L^2} =a_\chi,
    \end{equation*}
    which after normalization proves \eqref{eq:linear-first-layer-sep} with $c=\dfrac{a_\chi \cact}{\norm{\varphi}_{L^2}}>0$.  
\end{proof}

The preceding results isolate the instability mechanism for linear deformation fields. These fields have the advantage that their action in frequency is completely explicit, but they are not compactly supported and their displacement is not uniformly bounded. In this sense, they still differ from the deformation model usually considered in the stability theory of scattering transforms, where smallness is measured in a global $C^m$-norm.

We now pass from this affine model to compactly supported smooth deformations. The idea is to cut off the linear fields on a spatial region containing the relevant wave packets, so that the deformation agrees with the affine one where the channel-swap mechanism takes place, while becoming compactly supported and globally small in the required $C^m$-topology. To carry out this localization without losing the lower bound, we impose the following additional mild regularity assumption on the active filters.

\begin{assumption}[Uniform active-core regularity \textup{(R)}]\label{ass:active-coreRegularity}
    There is a constant $\Creg<\infty$ such that
    \begin{equation*}
        \sup_{j\geq j_0}\norm{\nabla m_{i_j}}_{L^\infty(B_{r_0}(\xi_j))}\leq \Creg.
    \end{equation*}
\end{assumption}

\begin{remark}[Locality of the active-core assumptions]
    The active-core condition \textup{(AC)} ensures that the selected multiplier is uniformly nondegenerate on a fixed interior core. By contrast, the uniform active-core regularity condition \textup{(R)} provides the local first-order control needed to show that the associated demodulated profiles have uniformly small $L^2$-tails. Both conditions are local and are satisfied by the smooth adapted Parseval model class introduced in Section~\ref{subsec:designed-model}: condition \textup{(AC)} was verified above, while condition \textup{(R)} will be established in Proposition~\ref{prop:model-Rass} below. More precisely, Lemma~\ref{lem:BV-separation} remains valid for every sufficiently small core radius. We may therefore decrease the initial choice of $r_0$, if necessary, and assume that \textup{(AC)} and \textup{(R)} hold simultaneously on the same active balls $B_{r_0}(\xi_j)$.
\end{remark}

It remains to verify the regularity condition \textup{(R)} for the smooth adapted Parseval model class.

\begin{proposition}\label{prop:model-Rass}
    The smooth Parseval model class from Section~\ref{subsec:designed-model} satisfies \textup{(R)}.  
\end{proposition}
\begin{proof}
    First note that, for $i=(j,m,\ell)\in I_\ast$, we have $\gamma_i=\gamma\circ \aff_i^{-1}$, where $\aff_i^{-1}$ is the inverse of the affine map associated with the almost structured covering, i.e., 
    $\aff_i^{-1}= A_j^{-1}(R_{j,\ell}^{-1} \cdot -c_{j,m})$. Since the maximum singular value of the Jacobian $D(\aff_i^{-1})=A_j^{-1}R_{j,\ell}^{-1}$ is $2^{-\beta j}\leq 1$, we get the uniform bound
    \begin{equation*}
        \norm{\nabla \gamma_{i}}_{L^\infty}\leq \norm{\nabla\gamma}_{L^\infty}<\infty.
    \end{equation*}
    Recall from Section~\ref{subsec:designed-model} the definition of the auxiliary function $S=\abs{b_0}^2+\sum_{i\in I_\ast} |\gamma_i|^2$. Since all atoms are smooth and the sum is everywhere locally finite by the admissibility of the covering, so is $S$ smooth with
    \begin{equation*}
        \norm{\nabla S}_{L^\infty}\leq 2 (\norm{\nabla b_0}_{L^\infty}+\calN(\bvQ) \norm{\nabla\gamma}_{L^\infty}).
    \end{equation*}
    Since $S$ is bounded from below by $1$, the previous estimates applied to $m_{i}=S^{-1/2} \gamma_{i}$ yield
    \begin{equation*}
        \norm{\nabla m_i}_{L^\infty} \leq \norm{\nabla \gamma_i}_{L^\infty} + \norm{\gamma_i}_{L^\infty} \norm{\nabla S}_{L^\infty}\leq \norm{\nabla\gamma}_{L^\infty} + 2 (\norm{\nabla b_0}_{L^\infty}+\calN(\bvQ) \norm{\nabla\gamma}_{L^\infty})=:\Creg.
    \end{equation*}
    In particular, this proves the uniform active-core regularity \textup{(R)} of the system.
\end{proof}

The next lemma illustrates how we get from linear to compactly supported smooth deformation instability.
\begin{lemma}\label{lem:compactification-transfer}
    Let $j_0$, $(\xi_j)_{j\ge j_0}$ and $(M_j)_{j\ge j_0}$ be as in Lemma \ref{lem:BV-separation}, and define
    \begin{equation}\label{eq:linear-field-transfer}
        \ell_j(x)\coloneqq M_j^\top x,\qquad x\in\R^2.
    \end{equation}
    Let $(f_j)_{j\ge j_0}$ be a sequence in $L^2(\R^2)$ such that
    \begin{equation}\label{eq:uniform-tightness-transfer}
        \lim_{R\to\infty}\sup_{j\geq j_0}\norm{f_j}_{L^2(\{\abs{x}>R\})}=0.
    \end{equation}
    Then, for every $\varepsilon>0$ there exist deformations $\tau_j\in C_c^\infty(\R^2;\R^2)$, $j \ge j_0$, such that, for every $m\in\N_0$, there is a constant $C_m<\infty$, independent of $j$, with
    \[\norm{\tau_j}_{C^m}\leq C_m2^{-(1-\beta)j},\qquad j\geq j_0,\]
    and
    \begin{equation*}
        \limsup_{j\to\infty}\norm{L_{\tau_j}f_j-L_{\ell_j}f_j}_{L^2}\leq\eps.
    \end{equation*}
\end{lemma}
\begin{proof}
    Fix a cutoff $\eta\in C_c^\infty(\R^2)$ such that $\eta=1$ on $B_1(0)$. For $R\geq1$ set
    \[\tau_{j,R}(x)\coloneqq M_j^\top x\,\eta(x/R),\qquad x\in\R^2.\]
    Then $\tau_{j,R}\in C_c^\infty(\R^2;\R^2)$ and $\tau_{j,R}=\ell_j$ on $B_R(0)$. Moreover, for every $m\in\Nzero$ and fixed $R$,
    \begin{equation*}
        \norm{\tau_{j,R}}_{C^m}\lesssim_{m,R}\norm{M_j}\lesssim_{m,R}2^{-(1-\beta)j},
    \end{equation*}
    by Lemma~\ref{lem:BV-separation}.
    
    We next compare the two deformation operators. Let
    \begin{equation*}
        \Phi_{j,R}(x)\coloneqq x-\tau_{j,R}(x),\qquad \Phi_j^{\rm lin}(x)\coloneqq x-\ell_j(x).
    \end{equation*}
    Since $M_j\to0$, there is $J_R\geq j_0$ such that, for all $j\geq J_R$, both $\Phi_{j,R}$ and $\Phi_j^{\rm lin}$ have Jacobians bounded from above and below by constants independent of $j$. Moreover, after increasing $J_R$ if necessary,
    \begin{equation*}
        \abs{\Phi_{j,R}(x)}\geq \frac12\abs{x},\qquad \abs{\Phi_j^{\rm lin}(x)}\geq \frac12\abs{x},\qquad j\geq J_R.
    \end{equation*}
    Hence both maps send $\set*{\abs{x}>R}$ into $\set*{\abs{y}>R/2}$. Since $\tau_{j,R}=\ell_j$ on $B_R(0)$, the two outputs agree there, and a change of variables gives, for all $j\geq J_R$,
    \begin{equation*}
        \norm{L_{\tau_{j,R}}f_j-L_{\ell_j}f_j}_{L^2}\lesssim \norm{f_j}_{L^2(\set*{\abs{y}>R/2})}.
    \end{equation*}
    Here the implicit constant is independent of $j$ and $R$. The uniform tightness assumption \eqref{eq:uniform-tightness-transfer} implies
    \begin{equation*}
        \lim_{R\to\infty}\sup_{j\geq j_0}\norm{f_j}_{L^2(\set*{\abs{y}>R/2})}=0.
    \end{equation*}
    We may therefore choose $R\geq1$ so large that the corresponding limsup is at most $\eps$. With this $R$ fixed, define $\tau_j=\tau_{j,R}$ for all sufficiently large $j$, and define the remaining finitely many $\tau_j$ arbitrarily, for instance by $\tau_j=0$. Enlarging the constants $C_m$ to account for these finitely many indices, we obtain, for every $m\in\Nzero$,
    \[\norm{\tau_j}_{C^m}\leq C_m2^{-(1-\beta)j},\qquad j\geq j_0,\]
    with $C_m$ independent of $j$, and
    \begin{equation*}
        \limsup_{j\to\infty}\norm{L_{\tau_j}f_j-L_{\ell_j}f_j}_{L^2}\leq\eps.
    \end{equation*}
\end{proof}
Let us finally present one last auxiliary lemma, in which we establish the role of assumption \textup{(R)} in the proceeding results. In a nutshell, the imposed regularity on the multipliers ensures uniform $L^2$-concentration of the adversarial packets in the space domain.

\begin{lemma}\label{lem:uniform-tightness}
    Assume both \textup{(AC)} and \textup{(R)}. Then, the functions $(f_j)_{j\ge j_0}$ defined in \eqref{eq:def_f_j} in Proposition~\ref{prop:first-layer-separation} satisfy
    \[\lim_{R\to\infty}\sup_{j\geq j_0}\norm{f_j}_{L^2(\{\abs{x}>R\})}=0.\]
\end{lemma}

\begin{proof}
    Recall from the proof of Proposition~\ref{prop:first-layer-separation} that the $(f_j)_{j\ge j_0}$ are obtained from auxiliary functions $g_j$ by means of modulation and normalization, where
    \begin{equation}
        q_j(\zeta)\coloneqq
        \begin{cases}
            \dfrac{\varphi(\zeta)}{m_{i_j}(\xi_j+\zeta)}, & \zeta\in B_{r_0}(0)\\[4pt]
            0, & \zeta\notin B_{r_0}(0),
        \end{cases}
        \qquad g_j\coloneqq\calF^{-1}q_j,
    \end{equation}
    for some fixed $0\neq\varphi\in C_c^\infty(B_{r_0/4}(0))$. By \eqref{eq:uniform-norm-control-g_j}, it thus suffices to prove that 
    \[\lim_{R\to\infty}\sup_{j\geq j_0}\norm{g_j}_{L^2(\{\abs{x}>R\})}=0.\]
    First, using assumption \textup{(AC)}, we estimate
    \begin{align*}
        \norm{\nabla q_j}_{L^2}
        &\leq \norm{\frac{\nabla \varphi}{m_{i_j}(\xi_j+\cdot)}}_{L^2} + \norm{\frac{\varphi \nabla m_{i_j}(\xi_j+\cdot)}{m_{i_j}^2(\xi_j+\cdot)}}_{L^2} \\
        &\leq \frac{\norm{\nabla \varphi}_{L^2}}{\cact} + \frac{\norm{\varphi}_{L^2}\norm{\nabla m_{i_j}}_{L^\infty(B_{r_0}(\xi_j))}}{\cact^2}.
    \end{align*}
    In particular, from assumption \textup{(R)} we can now infer that
    \begin{align*}
        C\coloneqq \sup_{j\geq j_0}\norm{\nabla q_j}_{L^2}<\infty.
    \end{align*}
    It finally follows from Parseval's theorem that
    \begin{equation*}
        \int_{\R^2}\abs{x}^2 \abs{g_j(x)}^2 \dx = \frac{1}{4\pi^2} \norm{\nabla q_j}_{L^2}^2\leq C^2,
    \end{equation*}
    which itself entails the uniform bound 
    \begin{equation}
        \norm{g_j}_{L^2(\{\abs{x}>R\})}^2\leq R^{-2} \int_{\R^2}\abs{x}^2 \abs{g_j(x)}^2 \dx \leq \left(\frac{C}{R}\right)^2 \to 0 \qquad \text{as } R\to\infty. \qedhere
    \end{equation}
\end{proof}

With these preparations in place, we can now prove the main result of this section.
\begin{theorem} \label{thm:first-layer-separation}
    Under the standing assumption $0\leq\beta\leq\alpha\leq1$, assume further that $\beta<1$, that the support subordination condition \eqref{eq:supp-sub-transfer} holds, and that both the active-core condition \textup{(AC)} and the uniform active-core regularity \textup{(R)} are satisfied. 
    
    There exist functions $f_j\in\calS(\R^2)$ with compact Fourier support and $\norm{f_j}_{L^2}=1$, and compactly supported smooth deformations $\tau_j\in C_c^\infty(\R^2;\R^2)$, as well as a constant $c>0$, independent of $j$, such that, for all sufficiently large $j$,
    \begin{equation}\label{eq:compact-first-layer-data}
        \norm{\tau_j}_{C^m}\lesssim_m 2^{-(1-\beta)j}, \qquad \text{for all } m\in\N_0,
    \end{equation}
    and
    \begin{equation}\label{eq:compact-first-layer-sep}
        \norm{\ScatOne(L_{\tau_j}f_j)-\ScatOne f_j}_{\lp}\geq c.
    \end{equation}

    Moreover, for the associated canonical linear filter map $\calW$ one also has
    \begin{equation}\label{eq:first-layer-sep-linear-comm}
        \norm{[\calW, L_{\tau_j}]f_j}_{\lp}\geq c
    \end{equation}
    for all sufficiently large $j$.
\end{theorem}

\begin{proof}
    We start out by providing the proof of \eqref{eq:compact-first-layer-sep}.
    Let the functions $(f_j)_{j\ge j_0}$ be again as specified in Proposition \ref{prop:first-layer-separation}. Recall that the same result guarantees the existence of $c_\ast>0$ such that, for the linear deformation $\ell_j(x)=M_j^\top x$, we have the lower bound
    \[\norm{\ScatOne(L_{\ell_j}f_j)-\ScatOne f_j}_{\lp}\geq c_\ast.\]
    It remains to replace the linear deformation by a compactly supported one, given the additional hypothesis \textup{(R)}. By Lemma~\ref{lem:uniform-tightness}, we have
    \begin{equation*}
        \lim_{R\to\infty}\sup_{j\geq j_0}\norm{f_j}_{L^2(\{\abs{x}>R\})}=0,
    \end{equation*}
    so we can resort to Lemma \ref{lem:compactification-transfer} with $\varepsilon=c_*/4$: possibly after increasing $j_0$, we obtain $\tau_j\in C_c^\infty(\R^2;\R^2)$ such that
    \begin{equation}
        \norm{\tau_j}_{C^m}\lesssim_m \norm{M_j}\asymp2^{-(1-\beta)j}, \qquad 
        \norm{L_{\tau_j}f_j-L_{\ell_j}f_j}_{L^2}\leq \frac{c_*}{2}
    \end{equation}
    for all $j\geq j_0$. By the first-layer nonexpansiveness \eqref{eq:first-layer-nonexpansive-transfer}, we obtain
    \begin{equation}
        \norm{\ScatOne(L_{\tau_j}f_j)-\ScatOne(L_{\ell_j}f_j)}_{\lp} \leq \norm{L_{\tau_j}f_j-L_{\ell_j}f_j}_{L^2}
        \leq \frac{c_*}{2}.
    \end{equation}
    Combining this estimate with the lower bound for the linear deformation, we finally get
    \begin{align}
        \norm{\ScatOne(L_{\tau_j}f_j)-\ScatOne f_j}_{\lp}
        &\geq \norm{\ScatOne(L_{\ell_j}f_j)-\ScatOne f_j}_{\lp} - \norm{\ScatOne(L_{\tau_j}f_j)-\ScatOne(L_{\ell_j}f_j)}_{\lp} \\
        &\geq c_*-\frac{c_*}{2} = \frac{c_*}{2},
    \end{align}
    which proves the first part of the theorem with $c=c_*/2$.

    We next prove the linear commutator lower bound \eqref{eq:first-layer-sep-linear-comm}. Recall that
    \[\ScatOne=\calC_\chi\circ U[I_\ast],\]
    where $\calC_\chi$ denotes componentwise convolution with the low-pass filter $\chi$. The commutator identity gives
    \begin{equation*}
        [\ScatOne,L_{\tau_j}]=\calC_\chi[U[I_\ast],L_{\tau_j}]+[\calC_\chi,L_{\tau_j}]U[I_\ast].
    \end{equation*}
    Since $\calC_\chi$ is nonexpansive on $\lp$, the reverse triangle inequality implies
    \begin{align}
        \norm{[\ScatOne,L_{\tau_j}]f_j}_{\lp}
        &\leq \norm{[U[I_\ast],L_{\tau_j}]f_j}_{\lp} +\norm{[\calC_\chi,L_{\tau_j}]U[I_\ast]f_j}_{\lp} \notag \\
        &\leq \norm{[\calW,L_{\tau_j}]f_j}_{\lp} +C(\|\tau_j\|_{L^\infty}+\norm{D\tau_j}_{L^\infty}) \norm{U[I_\ast]f_j}_{\lp}. \label{eq:nonlinear-linear-comm-transfer}
    \end{align}
    Here, the last inequality is due to Lemma~\ref{lem:appendix-low-pass-deformation-estimate} for a constant $C>0$ depending only on the low-pass filter $\chi$; the lemma applies since, for all sufficiently large $j$, we have $\norm{D\tau_j}_{L^\infty}\leq \frac12$.
    Moreover, the Parseval property and the invariance of the $L^2$-norm under the modulus give
    \begin{equation*}
        \norm{U[I_\ast]f_j}_{\lp}=\norm{(f_j*\psi_i)_{i\in I_\ast}}_{\lp}\leq \norm{f_j}_{L^2}=1.
    \end{equation*}
    Consequently,
    \begin{equation}
        \norm{[\ScatOne,L_{\tau_j}]f_j}_{\lp}
        \leq \norm{[\calW,L_{\tau_j}]f_j}_{\lp}+C(\|\tau_j\|_{L^\infty}+\norm{D\tau_j}_{L^\infty}).
        \label{eq:first-layer-comm-upper-linear}
    \end{equation}
    
    It remains to compare the nonlinear commutator with the difference appearing in \eqref{eq:compact-first-layer-sep}. Again by Lemma~\ref{lem:appendix-low-pass-deformation-estimate},
    \begin{equation}\label{eq:formed-coefficient-deformation}
        \norm{\ScatOne f_j-L_{\tau_j}\ScatOne f_j}_{\lp} \leq C \norm{\tau_j}_{L^\infty} \norm{(f_j*\psi_i)_{i\in I_\ast}}_{\lp} \leq C \norm{\tau_j}_{L^\infty}.
    \end{equation}
    On the other hand,
    \begin{equation*}
        \ScatOne(L_{\tau_j}f_j)-\ScatOne f_j=[\ScatOne,L_{\tau_j}]f_j+L_{\tau_j}\ScatOne f_j-\ScatOne f_j.
    \end{equation*}
    Combining the lower bound already established for the compactly supported deformations with \eqref{eq:first-layer-comm-upper-linear} and \eqref{eq:formed-coefficient-deformation}, we obtain
    \begin{equation*}
        \frac{c_\ast}{2} \leq \norm{\ScatOne(L_{\tau_j}f_j)-\ScatOne f_j}_{\lp} \leq \norm{[\calW,L_{\tau_j}]f_j}_{\lp} +C (2\norm{\tau_j}_{L^\infty}+\norm{D\tau_j}_{L^\infty}).
    \end{equation*}
    By \eqref{eq:compact-first-layer-data},
    \begin{equation*}
        2\norm{\tau_j}_{L^\infty}+\norm{D\tau_j}_{L^\infty}\lesssim 2^{-(1-\beta)j}\longrightarrow 0.
    \end{equation*}
    After increasing $j_0$ once more, we may therefore arrange that
    \[\norm{[\calW,L_{\tau_j}]f_j}_{\lp}\geq \frac{c_\ast}{4},\]
    which proves \eqref{eq:first-layer-sep-linear-comm}. Thus, upon decreasing the constant in the statement to $c=c_\ast/4$, if necessary, the same constant applies to both conclusions.

    It remains to obtain Schwartz inputs. For all sufficiently large $j$, we have $\norm{D\tau_j}_{L^\infty}\leq1/2$, so Lemma~\ref{lem:composition-small-deformation} gives $\norm{L_{\tau_j}}_{L^2\to L^2}\leq2$. Since $\supp\widehat f_j\subset B_{r_0/4}(\xi_j)$, density on the Fourier side and normalization provide functions $h_j\in\calS(\R^2)$ such that
    \begin{equation*}
        \norm{h_j}_{L^2}=1,\qquad\supp\widehat h_j\subset B_{r_0/2}(\xi_j),\qquad\norm{h_j-f_j}_{L^2}\leq \frac{c}{8}.
    \end{equation*}
    By \eqref{eq:first-layer-nonexpansive-transfer} and the isometry of $\calW$,
    \begin{align*}
        \norm{(\ScatOne(L_{\tau_j}h_j)-\ScatOne h_j)-(\ScatOne(L_{\tau_j}f_j)-\ScatOne f_j)}_{\lp}&\leq3\norm{h_j-f_j}_{L^2},\\
        \norm{[\calW,L_{\tau_j}](h_j-f_j)}_{\lp}&\leq4\norm{h_j-f_j}_{L^2}.
    \end{align*}
    Thus both separation bounds remain valid for $h_j$ with lower bound $c/2$.
\end{proof}

\begin{corollary} \label{cor:modulus-obstruction}
    Assume the hypotheses of Theorem \ref{thm:first-layer-separation} and fix $m\in\Nzero$, $s\geq0$, and $\kappa\in(0,1)$. Let $\omega\colon[0,1]\to[0,\infty)$ be nondecreasing with $\omega(0)=0$ and $\omega(t)\to0$ as $t\downarrow0$. Suppose that there exists a constant $C<\infty$ such that
    \begin{equation}\label{eq:modulus-stability}
        \norm{\ScatOne(L_\tau f)-\ScatOne f}_{\lp} \leq C\,\omega(\norm{\tau}_{C^m})\,\norm{f}_{H^s}
    \end{equation}
    for all deformations $\tau\in C_c^\infty(\R^2;\R^2)$ satisfying $\norm{\tau}_{C^m}\leq1$ and $\norm{D\tau}_{L^\infty}\leq\kappa$ and all $f\in H^s(\R^2)$. Then necessarily
    \begin{equation}\label{eq:necessary-modulus-growth}
        \limsup_{t\downarrow0}\frac{\omega(t)}{t^{s/(1-\beta)}}>0.
    \end{equation}
    Equivalently, no estimate of the form \eqref{eq:modulus-stability} can hold with $\omega(t)=o(t^{s/(1-\beta)})$ as $t\downarrow0$.

    In particular, for every $0\leq s<1-\beta$, no Lipschitz estimate of the form
    \begin{equation*}
        \norm{\ScatOne(L_\tau f)-\ScatOne f}_{\lp}\leq C\norm{\tau}_{C^m}\norm{f}_{H^s}
    \end{equation*}
    can hold uniformly for compactly supported smooth deformations $\tau$ satisfying $\norm{\tau}_{C^m}\leq1$ and $\norm{D\tau}_{L^\infty}\leq\kappa$ and inputs $f\in H^s(\R^2)$. 

    The counterexamples can be chosen in $\calS(\R^2)$ and the same stability obstruction applies to the full scattering transform whenever the output norm contains the first layer as a subcollection of features.
\end{corollary}

\begin{proof}
    On the one hand, the functions $f_j$ from Theorem \ref{thm:first-layer-separation} satisfy $\supp\widehat f_j\subset B_{r_0}(\xi_j)$ and $\abs{\xi_j}\asymp2^j$; therefore $\norm{f_j}_{H^s} \le B_s 2^{sj}$ for some $B_s>0$. On the other hand, the compactly supported deformations $\tau_j$ from Theorem \ref{thm:first-layer-separation} satisfy $\norm{\tau_j}_{C^m} \le A_m  2^{-(1-\beta)j}$ for some $A_m>0$. Set then $t_j=A_m2^{-(1-\beta)j}$, so that $t_j\downarrow0$ and in particular $t_j\leq1$ for all sufficiently large $j$. Moreover, $\norm{D\tau_j}_{L^\infty}\to0$, so $\norm{D\tau_j}_{L^\infty}\leq\kappa$ for all sufficiently large $j$. Since $\omega$ is nondecreasing, if \eqref{eq:modulus-stability} held, then
    \begin{equation}
        0< c \leq C\, \omega(\norm{\tau_j}_{C^m})\norm{f_j}_{H^s} \le CB_s\omega(t_j)2^{sj}.
    \end{equation}
    Furthermore, $2^{sj}=A_m^{s/(1-\beta)}t_j^{-s/(1-\beta)}$ and thus
    \begin{equation*}
        \frac{\omega(t_j)}{t_j^{s/(1-\beta)}}\geq \frac{c}{CB_sA_m^{s/(1-\beta)}}>0,
    \end{equation*}
    which in particular implies
    \[\limsup_{j\to\infty}\frac{\omega(t_j)}{t_j^{s/(1-\beta)}}>0.\]
    Since $t_j\downarrow0$, this rules out $\omega(t)=o(t^{s/(1-\beta)})$ as $t\downarrow0$.

    The remaining conclusions follow directly by combining the preceding results.
\end{proof}

\begin{remark}[Regularity of the active-core multipliers]\label{rm:active-core-regularity}
    The active-core condition \textup{(AC)} alone suffices for the construction of $L^2$ counterexamples with compact Fourier support. The uniform active-core regularity \textup{(R)} is used in the compactification argument to obtain uniformly small spatial tails. The support-preserving approximation in the proof of Theorem~\ref{thm:first-layer-separation} then yields Schwartz inputs without any higher-order multiplier regularity. Their Fourier support remains in $B_{r_0}(\xi_j)$, so this approximation preserves the Sobolev growth used in Corollary~\ref{cor:modulus-obstruction}.
\end{remark}

\begin{remark}[Sharpness of the parameter regime]
    The observed instability mechanism is sharp over the full parameter regime
    \[0\leq \beta\leq \alpha\leq 1\]
    for the studied wave-packet coverings $\bvQ$. In fact, the results apply to all admissible pairs $(\alpha,\beta)$ with $\beta<1$, while $\beta=1$ corresponds to the wavelet endpoint case $(\alpha,\beta)=(1,1)$, for which stability of the scattering transform was proved in Mallat's seminal work \cite{Mallat2012}.
\end{remark}

We conclude this section by explaining to what extent the preceding results extend to more general $d$-dimensional settings.

\begin{remark}[Higher-dimensional extensions]
    The systematic theory of wave-packet coverings used in this work is presently available only in dimension $d=2$. To the best of our knowledge, there is no directly comparable higher-dimensional theory, at least not with the same geometric precision and classification. This is one reason why the two-dimensional framework is particularly well suited to our purposes: thanks to the detailed geometric analysis of wave-packet coverings developed in \cite{BytchenkoffVoigtlaender2020}, the instability mechanism can ultimately be traced back to the geometry of the covering, rather than to a special choice of subordinate filters.

    In dimensions where such a geometric covering theory is not available, the affine separation argument can still be carried out once the relevant geometric features are imposed directly at the level of the filters. Fix a Parseval scattering bank $(\chi,\Psi)$ on $\R^d$ with $\chi\in L^1(\R^d)\cap L^2(\R^d)$, $\wh{\chi}(0)\neq0$, and $\Psi\subset L^1(\R^d)\cap L^2(\R^d)$. For $U\subset\R^d$, define
    \begin{equation*}
        \Psi(U)\coloneqq \set*{\psi\in \Psi\given \supp(\wh{\psi})\cap U \neq \varnothing}.
    \end{equation*}
    Assume that there exist $r_0>0$, points $\xi_j\in \R^d$,  auxiliary functions $g_j\in L^1(\R^d)\cap L^2(\R^d)$, high-pass filters $\psi_j\in \Psi$, and real matrices $M_j$ such that the following hold:
    \begin{enumerate}
        \item[(a)] $\norm{M_j}\to 0$ as $j\to\infty$;
        \item[(b)] $\sup_{j\in \N} \norm{g_j}_{L^1}<\infty$;
        \item[(c)] $\wh{g_j*\psi_j}\equiv 1$ on $B_{r_0}(\xi_j)$;
        \item[(d)] $\Psi(B_{r_0}(\xi_j))\cap \Psi((\Id-M_j)B_{r_0}(\xi_j))=\varnothing$.
    \end{enumerate}
    Under these assumptions, the affine first-layer separation bound \eqref{eq:linear-first-layer-sep} holds with $\ell_j(x)\coloneqq M_j^\top x$ and $\norm{D\ell_j}_{L^\infty}\to0$. Indeed, fix $0\neq\varphi\in C_c^\infty(B_{r_0/4}(0))$, set $h\coloneqq\calF^{-1}\varphi$, and define
    \begin{equation*}
        F_j\coloneqq g_j*(e^{2\pi\iu\xi_j\cdot(\cdot)}h),\qquad f_j\coloneqq\frac{F_j}{\norm{F_j}_{L^2}}.
    \end{equation*}
    Condition~(c) gives $F_j*\psi_j=e^{2\pi\iu\xi_j\cdot(\cdot)}h$, so $F_j\neq0$, while Young's inequality and~(b) bound $\norm{F_j}_{L^2}$ uniformly. For sufficiently large $j$, condition~(a) makes $\Id-M_j$ invertible, and condition~(d) gives $(L_{\ell_j}F_j)*\psi_j=0$. Since $\norm{\abs{h}*\chi}_{L^2}>0$ by the same argument as in Lemma~\ref{lem:TestPacket}, the $\psi_j$ channel yields a uniform positive separation after normalization.

    The compactification argument additionally requires uniform spatial localization of the normalized packets. Quantitative Sobolev obstructions also require a relation between the carrier frequencies $\xi_j$ and the deformation sizes $\norm{M_j}$; neither conclusion is asserted under conditions~(a)--(d) alone.
\end{remark}

\section{Stability at the critical Sobolev regularity scale} \label{sec:positive-full-depth}

We now give a positive counterpart to the geometric instability mechanism. In particular, note that the negative result shows that the transverse tile width forces the Sobolev exponent $1-\beta$. Our purpose in this section is to prove a matching stability estimate at this level of regularity, in the spirit of the classical one in \cite{Mallat2012} for the wavelet parameter corner $(\alpha,\beta)=(1,1)$. 

Throughout the arguments, the covering $\bvQ$ and the subordinate Parseval scattering bank $(\chi,\Psi)$ are fixed as in Definition~\ref{def:fixed-channel-bank}, and we assume $0 \le \beta<1$.

For $i\in I_\ast$, write $m_i=\widehat\psi_i$ and introduce the high-pass convolution operator $\calC_i h=h*\psi_i$. For $j\in\N$ we set
\begin{equation*}
    I_{\ast,j}\coloneqq \set*{(j,m,\ell)\in I_\ast \given m \le  m_j^{\max},\ \ell \le  \ell_j^{\max}}.
\end{equation*}
Recall that, for a family $q=(q_p)_p$ of functions, we define the high-pass modulus layer and the low-pass output layer by
\begin{equation*}
    (\calU q)_{(p,i)}=\abs{q_p*\psi_i},\qquad \calA q=(q_p*\chi)_p.
\end{equation*}
Thus $U_nf=\calU^nf$, $\ScatLe{N}f=(\calA U_nf)_{0 \le  n \le  N}$, and $\Scat f=(\calA U_nf)_{n\geq0}$.

For $0 \le  s \le 1$ and $f\in H^s(\R^2)$ we introduce the quantities
\begin{equation*}
    E_n^s(f) \coloneqq \|U_nf\|_{\ell^2(H^s)},\qquad \mathfrak E_s(f) \coloneqq \sum_{n=0}^\infty E_n^s(f)\in[0,\infty].
\end{equation*}
Note that $U_0f=f$, hence $E_0^s(f)=\norm{f}_{H^s}$. 

Let us record a simple technical result for later use. 

\begin{lemma} \label{lem:Sobolev-nonexpansiveness}
    For every $0 \le  s \le 1$ and countable family $q=(q_p)_p$ with $q\in\ell^2(H^s)$, we have
    \[\|\calU q\|_{\ell^2(H^s)} \le \|q\|_{\ell^2(H^s)}.\]
    As a result,
    \begin{equation*}
        E_{n+1}^s(f) \le  E_n^s(f) \le \norm{f}_{H^s},\qquad 0 \le  s \le 1,\quad n\geq0.
    \end{equation*}
\end{lemma}

\begin{proof}
    We shall use the contraction property
    \begin{equation}\label{eq:Hs-norm-of-|g|-vs-g} 
        \norm{|g|}_{H^s}\leq \norm{g}_{H^s},\qquad 0\leq s\leq1.
    \end{equation}
    For $s=0$ this is immediate. Let $0<s<1$. By the identity
    \begin{equation*}
        (1+\lambda)^s=1+c_s\int_0^\infty (1-e^{-t\lambda})e^{-t}\frac{\dd t}{t^{1+s}},\qquad \lambda\geq0,
    \end{equation*}
    with $c_s>0$, Plancherel gives
    \begin{equation*}
        \norm{g}_{H^s}^2=\norm{g}_{L^2} ^2+c_s\int_0^\infty e^{-t}t^{-1-s}\bigl(\norm{g}_{L^2}^2-\langle e^{t\Delta/(4\pi^2)}g,g\rangle\bigr)\dd t.
    \end{equation*}
    Since $e^{t\Delta/(4\pi^2)}$ is convolution with the positive heat kernel $p_t$, we have
    \begin{equation*}
        \norm{g}_{L^2}^2-\langle e^{t\Delta/(4\pi^2)}g,g\rangle=\frac12\int_{\R^2}\int_{\R^2}\abs{g(x)-g(y)}^2p_t(x-y)\dd{x}\dd{y}.
    \end{equation*}
    Applying this identity to $\abs{g}$ and using $\||g|\|_{L^2}=\|g\|_{L^2}$ together with $\bigl||g(x)|-|g(y)|\bigr|\leq |g(x)-g(y)|$ yields $\||g|\|_{H^s}\leq \|g\|_{H^s}$ for $0<s<1$. Finally, the case $s=1$ follows from the weak chain rule, which gives $|\nabla |g||\leq |\nabla g|$.

    Combining the contraction \eqref{eq:Hs-norm-of-|g|-vs-g} with the Parseval--Plancherel theorem, since $\sum_{i\in I_\ast}\abs{m_i(\xi)}^2 \le 1$, we obtain
    \begin{align}
        \|\calU q\|_{\ell^2(H^s)}^2
        &=\sum_p\sum_{i\in I_\ast}\||q_p*\psi_i|\|_{H^s}^2 \leq \sum_p\sum_{i\in I_\ast}\|q_p*\psi_i\|_{H^s}^2 \\
        &=\sum_p\int_{\R^2}\langle\xi\rangle^{2s}\left(\sum_{i\in I_\ast}|m_i(\xi)|^2\right)|\widehat q_p(\xi)|^2 \dd{\xi} \leq \sum_p\|q_p\|_{H^s}^2.
    \end{align}
    This proves the first assertion. The second follows by iteration, applied to $q=U_nf$.
\end{proof}

\subsection{A commutator-based approach}

As in \cite{Mallat2012}, we are going to exploit a commutator-bound argument. To this end, some additional regularity assumptions on the filters are needed. 
\begin{assumption} \label{ass:comm-filter-ass}
    We assume that $\beta<1$ and that the fixed filter bank satisfies the following conditions. All multiplier derivatives below are understood in the weak sense.
    \begin{enumerate}[label={(C\arabic*)}]
        \item There exists $j_{\rm reg}\in\N$ such that all the multipliers $m_i$ with
        \[i\in I_{\rm fin}\coloneqq \bigcup_{1\le j<j_{\rm reg}}I_{\ast,j}\]
        have compact support and belong to $W^{4,\infty}(\R^2)$.
        \item For every multi-index $\nu$ with $1\le \abs{\nu}\le 4$,
        \begin{equation*}
            \sup_{j\geq j_{\rm reg}}\operatorname*{ess\,sup}_{\xi\in\R^2}2^{\beta j|\nu|}\left(\sum_{i\in I_{\ast,j}}|\partial^\nu m_i(\xi)|^2\right)^{1/2}<\infty.
        \end{equation*}
    \end{enumerate}
\end{assumption}

\begin{remark}[The fourth-order smoothness requirement]
    The numerical order four in \textup{(C2)} is imposed only for the commutator estimate in Section~\ref{sec:appendix-commutator}. In dimension two, the scale-local Coifman--Meyer estimate requires three derivatives in the localized frequency variable, reflecting the summability of $\sum_{n\in\Z^2}(1+\abs{n})^{-3}$. The commutator symbol involves one further derivative of the filter multiplier, hence the need for multiplier bounds up to order four. This is the only point where this amount of smoothness is used.
\end{remark}

\begin{proposition} 
    The smooth Parseval model class constructed in Section~\ref{subsec:designed-model} satisfies Assumption~\ref{ass:comm-filter-ass}.
\end{proposition}

\begin{proof}
    Recall the notation of the smooth Parseval construction: the generators are
    \begin{equation*}
        \gamma_i(\xi)=\gamma\big(A_j^{-1}(R_{j,\ell}^{-1}\xi-c_{j,m})\big),\qquad i=(j,m,\ell),
    \end{equation*}
    and
    \begin{equation*}
        m_i(\xi)=\gamma_i(\xi)S(\xi)^{-1/2},\qquad S(\xi)=\abs{b_0(\xi)}^2+\sum_{k\in I_\ast}|\gamma_k(\xi)|^2.
    \end{equation*}
    Recall that $\supp\gamma_i\subset Q_i$, so $\supp m_i\subset Q_i$ and $1 \le  S \le \calN(\bvQ)$.

    Choose now $j_{\rm reg}$ so large that $Q_0$ does not meet $Q_i$ for $i\in I_{\ast,j}$ and $j\ge j_{\rm reg}$; this implies that the set $I_{\rm fin}=\bigcup_{1 \le  j<j_{\rm reg}}I_{\ast,j}$ is finite, and all the corresponding multipliers are smooth and compactly supported by design, proving (C1).

    It remains to prove (C2). Given a multi-index $\nu \in \Nzero^2$ with $1 \le \abs{\nu} \le  4$, the chain rule gives
    \begin{equation*}
        \abs{\partial^\nu\gamma_i(\xi)} \lesssim_\nu \norm{A_j^{-1}}^{\abs{\nu}} \lesssim_\nu 2^{-\beta j|\nu|},\qquad i\in I_{\ast,j},
    \end{equation*}
    where we used that the smallest singular value of $A_j$ is $2^{\beta j}$. Since the base regularity is $C_c^\infty$, these estimates are available for every derivative order, hence in particular for all orders up to four. 

    We next prove the corresponding estimates for the normalization factor. Fix $i\in I_{\ast,j}$ with $j\ge j_{\rm reg}$ and $\xi\in\supp\gamma_i$; by the way $j_{\rm reg}$ is chosen, the low-pass term $b_0$ vanishes in a neighborhood of $\xi$. Moreover, if $\gamma_k(\xi)\neq0$ for some $k\in I_{\ast,j'}$, then $\xi\in\supp\gamma_i\cap\supp\gamma_k\subset Q_i\cap Q_k$. Now, by the finite-band property of the covering and uniformly finite overlap, this implies $\abs{j-j'} \le  C$ with only $\calO(1)$ such indices $k$ contributing at the point $\xi$. 
    As a consequence, for every $\nu\in\Nzero^2$ with $1 \le \abs{\nu} \le  4$,
    \begin{equation*}
        \abs{\partial^\nu S(\xi)}\lesssim_\nu 2^{-\beta j|\nu|},\qquad \xi\in\supp\gamma_i,\quad i\in I_{\ast,j}, \qquad j\ge j_{\rm reg}.
    \end{equation*}
    Since $S\geq1$, Faà di Bruno's formula gives the same estimate for the derivatives of $S^{-1/2}$:
    \begin{equation*}
        \abs{\partial^\nu(S^{-1/2})(\xi)}\lesssim_\nu 2^{-\beta j|\nu|},\qquad \xi\in\supp\gamma_i,\quad i\in I_{\ast,j},\quad 1 \le \abs{\nu} \le  4.
    \end{equation*}
    Therefore, differentiating $m_i=\gamma_iS^{-1/2}$ and using Leibniz' rule yields
    \begin{equation*}
        \abs{\partial^\nu m_i(\xi)} \lesssim_\nu  2^{-\beta j|\nu|},\qquad i\in I_{\ast,j}.
    \end{equation*}
    Finally, at each fixed $\xi$, uniformly finite overlap implies that only $\calO(1)$ terms in the sum over $I_{\ast,j}$ are nonzero, hence
    \begin{equation*}
        \bigg(\sum_{i\in I_{\ast,j}}|\partial^\nu m_i(\xi)|^2\bigg)^{1/2} \lesssim_\nu 2^{-\beta j|\nu|}, \qquad j\ge j_{\rm reg},
    \end{equation*}
    which is precisely (C2).
\end{proof}

The following lemma is a key technical ingredient for reducing the stability of the scattering transform under small diffeomorphisms to a commutator estimate for the deformation operator $L_\tau$ and the one-step scattering propagator $\calU$.

\begin{lemma}\label{lem:finite-depth-defect-propagation}
    For every $N\in\N_0$ and arbitrary families $y_n\in\ell^2(I_\ast^n;L^2)$, $0\leq n\leq N$, one has
    \begin{equation*}
        \left(\sum_{n=0}^N \|\calA \calU^n y_0-\calA  y_n\|_{\lp}^2\right)^{1/2}\leq \sum_{\ell=0}^{N-1}\|\calU y_\ell-y_{\ell+1}\|_{\lp},
    \end{equation*}
    with the convention that the empty sum is zero.
\end{lemma}

\begin{proof}
    The cases $N=0$ and $N=1$ are immediate. For $2\leq n\leq N$, we insert the intermediate terms
    \[\calA \calU^{n-\ell}y_\ell,\qquad 1\leq \ell\leq n-1,\]
    and obtain
    \begin{equation*}
        \|\calA \calU^n y_0-\calA  y_n\|_{\lp}\leq\sum_{\ell=0}^{n-1}\|\calA \calU^{n-\ell-1}(\calU y_\ell)-\calA \calU^{n-\ell-1}y_{\ell+1}\|_{\lp}.
    \end{equation*}
    Taking the $\ell^2$-norm in $n$ and using Minkowski's inequality gives
    \begin{align*}
        \left(\sum_{n=0}^N \|\calA \calU^n y_0-\calA y_n\|_{\lp}^2\right)^{1/2}
        &\leq \sum_{\ell=0}^{N-1} \left(\sum_{n=\ell+1}^N \|\calA \calU^{n-\ell-1}(\calU y_\ell) -\calA \calU^{n-\ell-1}y_{\ell+1}\|_{\lp}^2 \right)^{1/2} \\
        &= \sum_{\ell=0}^{N-1} \left(\sum_{r=0}^{N-\ell-1} \|\calA \calU^r(\calU y_\ell) -\calA \calU^r y_{\ell+1}\|_{\lp}^2 \right)^{1/2}.
    \end{align*}
    The Littlewood--Paley identity and the nonexpansiveness of the modulus give, for compatible arrays $q$ and $\widetilde{q}$,
    \begin{equation*}
        \norm{\calA q-\calA \widetilde{q}}_{\lp}^2+\norm{\calU q-\calU \widetilde{q}}_{\lp}^2\leq\norm{q-\widetilde{q}}_{\lp}^2.
    \end{equation*}
    Telescoping over this one-step inequality yields nonexpansiveness of the truncated scattering tail
    \begin{equation*}
        q\mapsto \bigl(\calA \calU^r q\bigr)_{0\leq r\leq N-\ell-1}.
    \end{equation*}
    This yields the claimed upper bound.
\end{proof}

Let us stress that conditions \textup{(C1)}--\textup{(C2)} are the minimal ingredients to obtain the following scattering commutator bounds, whose proof is deferred to Section~\ref{sec:appendix-commutator}. Intuitively, the statement says that a small deformation can be pushed through one high-pass modulus layer at the cost of exactly one Sobolev norm $H^{1-\beta}$. From a heuristic point of view, after applying a generation-$j$ wave-packet filter the deformation error is measured at the carrier scale $2^j$, while the channel resolution in the transverse direction is $2^{\beta j}$. The effective cost of moving the deformation through the layer is therefore $2^j/2^{\beta j}=2^{(1-\beta)j}$, which is precisely the dyadic weight of $H^{1-\beta}$ (equivalently, of the $L^2$-based wave-packet decomposition norm $W_{1-\beta}^{2,2}(\alpha,\beta)$). On the other hand, the low-pass part does not suffer from such critical cost and is controlled directly by the size of the deformation.

\begin{proposition} \label{prop:scattering-commutator-plugin}
    Suppose that Assumption~\ref{ass:comm-filter-ass} holds. Then, for every $\kappa\in(0,1)$ there exists a constant $C>0$ such that, whenever $\tau\in W^{1,\infty}(\R^2;\R^2)$ satisfies $\norm{D\tau}_{L^\infty}\leq \kappa$, we have, for every countable family $q=(q_p)_{p\in P} \in \ell^2(H^{1-\beta})$,
    \begin{equation*}
        \|\calU L_\tau q-L_\tau\calU q\|_{\lp} \le  C\|D\tau\|_{L^\infty}\|q\|_{\ell^2(H^{1-\beta})}.
    \end{equation*}
    If, in addition, $q=(q_p)_{p\in P} \in \lp$, then
    \begin{equation*}
        \|\calA L_\tau q-\calA q\|_{\lp} \le  C(\|\tau\|_{L^\infty}+\norm{D\tau}_{L^\infty})\|q\|_{\lp}.
    \end{equation*}
\end{proposition}

With this in hand, we are able to obtain the following endpoint stability result. 
\begin{theorem} \label{thm:abstract-full-depth-positive}
    Suppose that Assumption~\ref{ass:comm-filter-ass} holds. Then, for every $\kappa\in(0,1)$ there exists a constant $C>0$ such that, whenever $\tau\in W^{1,\infty}(\R^2;\R^2)$ satisfies $\norm{D\tau}_{L^\infty}\leq \kappa$, we have, for every $f\in H^{1-\beta}(\R^2)$,
    \begin{equation*}
        \norm{\Scat(L_\tau f)-\Scat f}_{\lp}\leq C\bigl(\norm{\tau}_{L^\infty}+\norm{D\tau}_{L^\infty}\bigr)\mathfrak E_{1-\beta}(f).
    \end{equation*}
    The same estimate holds with $\Scat$ replaced by $\ScatLe{N}$ for every $N\in\N_0$, with the same constant.
\end{theorem}

\begin{proof}
    It suffices to prove the estimate under the assumption $\mathfrak E_{1-\beta}(f)<\infty$, first for finite-depth outputs and then by passing to the limit. 
    
    Fix $N\in\N_0$. Applying Lemma~\ref{lem:finite-depth-defect-propagation} to $y_n\coloneqq L_\tau U_nf$, $0\leq n\leq N$, we obtain
    \begin{align*}
        \|\ScatLe{N}(L_\tau f)-\ScatLe{N}f\|_{\lp}
        &\leq \left(\sum_{n=0}^N \|\calA U_n(L_\tau f)-\calA L_\tau U_nf\|_{\lp}^2 \right)^{1/2} \\
        &\quad+ \left(\sum_{n=0}^N \|\calA L_\tau U_nf-\calA U_nf\|_{\lp}^2 \right)^{1/2} \\
        &\leq \sum_{\ell=0}^{N-1} \|\calU L_\tau U_\ell f-L_\tau\calU U_\ell f\|_{\lp} \\
        &\quad+ \left(\sum_{n=0}^N \|\calA L_\tau U_nf-\calA U_nf\|_{\lp}^2 \right)^{1/2}.
    \end{align*}

    Since $\mathfrak E_{1-\beta}(f)<\infty$, every $U_\ell f$ belongs to $\ell^2(H^{1-\beta})$, and Proposition~\ref{prop:scattering-commutator-plugin} applies. Thus
    \begin{equation*}
        \sum_{\ell=0}^{N-1}\|\calU L_\tau U_\ell f-L_\tau\calU U_\ell f\|_{\lp} \lesssim \|D\tau\|_{L^\infty}\sum_{\ell=0}^{N-1}E_\ell^{1-\beta}(f) \le \norm{D\tau}_{L^\infty}\mathfrak E_{1-\beta}(f),
    \end{equation*}
    and
    \begin{equation*}
        \|\calA L_\tau U_nf-\calA U_nf\|_{\lp} \lesssim (\|\tau\|_{L^\infty}+\norm{D\tau}_{L^\infty})\|U_nf\|_{\lp}.
    \end{equation*}
    Since $\norm{U_nf}_{\lp} \le  E_n^{1-\beta}(f)$, we infer
    \begin{align}
        \left(\sum_{n=0}^N\|\calA L_\tau U_nf-\calA U_nf\|_{\lp}^2\right)^{1/2}
        &\lesssim (\|\tau\|_{L^\infty}+\norm{D\tau}_{L^\infty})\left(\sum_{n=0}^NE_n^{1-\beta}(f)^2\right)^{1/2} \\
        &\lesssim (\|\tau\|_{L^\infty}+\norm{D\tau}_{L^\infty})\mathfrak E_{1-\beta}(f).
    \end{align}
    Combining the preceding estimates gives
    \begin{equation*}
        \|\ScatLe{N}(L_\tau f)-\ScatLe{N}f\|_{\lp} \lesssim (\|\tau\|_{L^\infty}+\norm{D\tau}_{L^\infty})\mathfrak E_{1-\beta}(f),
    \end{equation*}
    with a constant independent of $N$. Finally, the squared norms of the truncated output differences increase monotonically to the squared full output norm; letting $N\to\infty$ therefore proves the claim.
\end{proof}

\subsection{Energy propagation at the Sobolev regularity scale}

Just like the deformation stability result in \cite{Mallat2012} for the wavelet-based scattering transform, Theorem~\ref{thm:abstract-full-depth-positive} proves stability conditional on a regularity assumption for input signals that quantifies the interaction with the scattering architecture. In particular, we require here the summability of the propagated critical Sobolev norms, namely
\begin{equation*}
    \mathfrak E_{1-\beta}(f)=\sum_{n=0}^{\infty}\|U_n f\|_{\ell^2(H^{1-\beta})}<\infty.
\end{equation*}
This condition is a priori stronger than the endpoint assumption $f\in H^{1-\beta}(\R^2)$, since Sobolev nonexpansiveness only gives
\begin{equation*}
    \norm{U_n f}_{\ell^2(H^{1-\beta})}\leq\norm{f}_{H^{1-\beta}},\qquad n\geq0,
\end{equation*}
and thus does not enforce summability in $n$. On the other hand, the geometry of $\bvQ$ naturally suggests the wave-packet decomposition spaces $W_s^{p,q}(\alpha,\beta)$ as a stability space, and in the current $L^2$-based setting Lemma~\ref{lem:hilbertian-identification} gives $W_s^{2,2}(\alpha,\beta)=H^s(\R^2)$ with equivalent norms. Therefore, the Sobolev scale is precisely the Hilbertian realization of the associated decomposition space.

To replace $\mathfrak E_{1-\beta}(f)$ by the endpoint Sobolev norm, we use the frequency-localization mechanism underlying the energy-propagation results in \cite{FuhrGetter2025}. The following assumption ensures that removing a carrier contracts the frequency by a fixed factor. Together with the high-pass gap, this will give strict contraction of the propagated Sobolev norm at the same regularity order.

\begin{assumption}\label{ass:energy-localization}
    The subordinate Parseval scattering bank $(\chi,\Psi)$ in Definition~\ref{def:fixed-channel-bank} satisfies the following conditions.

    \begin{enumerate}[label={(E\arabic*)}]
        \item There is a high-pass gap: for some $\rgap>0$,
        \[\sum_{i\in I_\ast}|m_i(\xi)|^2=0\qquad\text{for }|\xi|<\rgap.\]
        \item There are $\theta\in(0,1)$ and points $\nu_i\in\R^2$, $i\in I_\ast$, such that
        \begin{equation*}
            \supp m_i\subset\set*{\xi\in\R^2\given|\xi-\nu_i|\leq\theta|\xi|}\qquad\text{for every }i\in I_\ast.
        \end{equation*}
    \end{enumerate}
\end{assumption}

For the wave-packet covering $\bvQ$, this condition is not an additional restriction on the particular choice of filter profiles. It is a consequence of the covering geometry itself, including the boundary case $\alpha=1$.

\begin{proposition}\label{prop:energy-localization-BV}
    Every filter bank subordinate to the wave-packet covering $\bvQ$ satisfies Assumption~\ref{ass:energy-localization}. In particular, this includes the smooth Parseval model class constructed in Section~\ref{subsec:designed-model}.
\end{proposition}

\begin{proof}
    The high-pass gap (E1) is automatic from the wave-packet design, since all high-pass supports are contained in generations $j\geq1$ and these supports stay a positive distance from the origin.

    Concerning (E2), recall that there are constants $0<c<C$ such that for every $i=(j,m,\ell)$ and every $\xi\in Q_i$, if $R_{j,\ell}^{-1}\xi=(x,y)$, then
    \[x\geq c2^j,\qquad|\xi|=|(x,y)|\leq C2^j.\]
    Choose now $a\in(0,c)$ and define $\nu_i=R_{j,\ell}(a2^j,0)$. Then, for $\xi\in Q_i$ we have
    \begin{align*}
        \abs{\xi-\nu_i}^2&=\abs{(x,y)-(a2^j,0)}^2\\
        &=\abs{\xi}^2-2a2^jx+a^2 2^{2j}\\
        &\leq\abs{\xi}^2-(2ac-a^2)2^{2j}\\
        &\leq\bigg(1-\frac{2ac-a^2}{C^2}\bigg)|\xi|^2.
    \end{align*}
    We have $2ac-a^2>0$ by construction, and (after possibly decreasing $a$) also
    \[\theta=\bigg(1-\dfrac{2ac-a^2}{C^2}\bigg)^{1/2}\in(0,1).\]
    We thus obtain $\abs{\xi-\nu_i}\leq\theta|\xi|$ on $Q_i$, and therefore on $\supp m_i$, proving (E2).
\end{proof}

Condition \textup{(E2)} says that each high-pass filter has a carrier frequency $\nu_i$ such that, after demodulation, the remaining envelope frequencies are bounded by a fixed fraction $\theta<1$ of the original frequency. The high-pass gap \textup{(E1)} makes this contraction strict also for the inhomogeneous Sobolev weight. Combining these observations with the modulus contraction in Lemma~\ref{lem:Sobolev-nonexpansiveness} gives the required summability directly.

\begin{lemma}\label{lem:summability-above-critical}
    Under Assumption~\ref{ass:energy-localization}, set
    \begin{equation*}
        \varrho\coloneqq\left(\frac{1+\theta^2\rgap^2}{1+\rgap^2}\right)^{1/2}\in(0,1).
    \end{equation*}
    For every $0<s\leq1$, every countable set $P$, and every $q\in\ell^2(P;H^s)$, we have
    \begin{equation*}
        \|\calU q\|_{\ell^2(P\times I_\ast;H^s)}\leq\varrho^s\|q\|_{\ell^2(P;H^s)}.
    \end{equation*}
    Consequently, for every $f\in H^s(\R^2)$ and $n\in\N_0$,
    \begin{equation*}
        E_n^s(f)\leq\varrho^{sn}\norm{f}_{H^s},\qquad\mathfrak E_s(f)\leq\frac{\norm{f}_{H^s}}{1-\varrho^s}.
    \end{equation*}
\end{lemma}

\begin{proof}
    If $m_i(\xi)\neq0$, conditions (E1)--(E2) imply $|\xi|\geq\rgap$ and $|\xi-\nu_i|\leq\theta|\xi|$. Since 
    \[t\mapsto\frac{1+\theta^2t}{1+t}\] 
    is decreasing, it follows that
    \begin{equation*}
        \frac{1+|\xi-\nu_i|^2}{1+|\xi|^2}\leq\frac{1+\theta^2|\xi|^2}{1+|\xi|^2}\leq\frac{1+\theta^2\rgap^2}{1+\rgap^2}=\varrho^2.
    \end{equation*}
    For $p\in P$ and $i\in I_\ast$, set $g_{p,i}(x)\coloneqq e^{-2\pi\iu\nu_i\cdot x}(q_p*\psi_i)(x)$. Then
    \begin{equation*}
        \norm{g_{p,i}}_{H^s}^2=\int_{\R^2}(1+|\xi-\nu_i|^2)^s|m_i(\xi)|^2|\widehat q_p(\xi)|^2\dxi\leq\varrho^{2s}\norm{q_p}_{H^s}^2,
    \end{equation*}
    where the last estimate also uses $|m_i|\leq1$. In particular, $g_{p,i}\in H^s(\R^2)$, so the scalar modulus contraction \eqref{eq:Hs-norm-of-|g|-vs-g} applies. Since $|g_{p,i}|=|q_p*\psi_i|$, Tonelli's theorem and $\sum_{i\in I_\ast}|m_i|^2\leq1$ yield
    \begin{align*}
        \|\calU q\|_{\ell^2(H^s)}^2&\leq\sum_{p\in P}\sum_{i\in I_\ast}\norm{g_{p,i}}_{H^s}^2\\
        &\leq\varrho^{2s}\sum_{p\in P}\int_{\R^2}(1+|\xi|^2)^s\left(\sum_{i\in I_\ast}|m_i(\xi)|^2\right)|\widehat q_p(\xi)|^2\dxi\\
        &\leq\varrho^{2s}\|q\|_{\ell^2(H^s)}^2.
    \end{align*}
    This finite bound proves both membership of $\calU q$ in $\ell^2(H^s)$ and the asserted contraction. Since the path index set at every finite depth is countable, iterating the estimate from $U_0f=f$ therefore gives $E_n^s(f)\leq\varrho^{sn}\norm{f}_{H^s}$, and summing the geometric series proves the bound for $\mathfrak E_s(f)$.
\end{proof}

The high-pass gap and carrier contraction in Assumption~\ref{ass:energy-localization} are the geometric ingredients underlying \cite[Theorem~3.5]{FuhrGetter2025}; under the conical and radial localization hypotheses of \cite[Assumption~3.1]{FuhrGetter2025}, the carrier contraction follows from \cite[Lemma~3.3]{FuhrGetter2025}. Our formulation does not require a conical scale partition: for $0<s\leq1$, Lemma~\ref{lem:summability-above-critical} uses these properties directly to obtain strict contraction of the propagated $H^s$ norm, yielding the summability needed at the critical exponent. The framework of \cite{FuhrGetter2025} additionally provides decay estimates for generalized Sobolev classes adapted to the filter geometry.

\begin{remark}[Instability with uniformly bounded mixed scattering norm]\label{rm:bounded-mixed-norm-instability}
    Under the hypotheses of Theorem~\ref{thm:first-layer-separation}, the counterexamples constructed in its proof also have uniformly bounded mixed $\ell^1(\ell^2(L^2))$-scattering norm. Recall that these functions satisfy $\norm{f_j}_{L^2}=1$ and $\supp\widehat f_j\subset B_{r_0}(\xi_j)$, where $r_0>0$ is independent of $j$ and $\xi_j$ denotes the carrier frequency of the corresponding packet.

    To control the mixed norm, we first obtain a bound for $U_1f_j$ in $\ell^2(H^1)$ that is uniform in $j$; Lemma~\ref{lem:summability-above-critical} will then control the subsequent layers. Since all convolutions $f_j*\psi_i$ have Fourier support in $B_{r_0}(\xi_j)$, we remove their common carrier by setting
    \begin{equation*}
        g_{j,i}(x)\coloneqq e^{-2\pi\iu\xi_j\cdot x}(f_j*\psi_i)(x),\qquad i\in I_\ast.
    \end{equation*}
    The Fourier support of $f_j$ implies $g_{j,i}\in H^1(\R^2)$. Since $|g_{j,i}|=|f_j*\psi_i|$, the scalar modulus contraction \eqref{eq:Hs-norm-of-|g|-vs-g}, Tonelli's theorem, and the Parseval bound give
    \begin{align*}
        \norm{U_1f_j}_{\ell^2(H^1)}^2&\leq\sum_{i\in I_\ast}\norm{g_{j,i}}_{H^1}^2\\
        &=\int_{\R^2}(1+|\eta-\xi_j|^2)\left(\sum_{i\in I_\ast}|m_i(\eta)|^2\right)|\widehat f_j(\eta)|^2\dd{\eta}\\
        &\leq1+r_0^2.
    \end{align*}
    By Proposition~\ref{prop:energy-localization-BV}, the filter bank satisfies Assumption~\ref{ass:energy-localization}, so Lemma~\ref{lem:summability-above-critical} applies to the family $U_1f_j$. Taking $s=1$ and iterating from the first layer, we obtain
    \begin{equation*}
        \norm{U_nf_j}_{\lp}\leq\norm{U_nf_j}_{\ell^2(H^1)}\leq\varrho^{n-1}\sqrt{1+r_0^2},\qquad n\geq1,
    \end{equation*}
    and hence
    \begin{equation*}
        \sum_{n=0}^\infty\norm{U_nf_j}_{\lp}\leq1+\frac{\sqrt{1+r_0^2}}{1-\varrho}.
    \end{equation*}
    The bound is independent of $j$ and requires no additional multiplier regularity. Consequently, the obstruction in Corollary~\ref{cor:modulus-obstruction} persists even after restricting the input class to signals with uniformly bounded mixed $\ell^1(\lp)$ scattering norm.
\end{remark}

\subsection{Endpoint stability}

For the wave-packet setting, Assumption~\ref{ass:energy-localization} is automatic by Proposition~\ref{prop:energy-localization-BV}. The preceding lemma therefore closes the critical Sobolev estimate throughout the parameter range $\beta<1$, including $\alpha=1$, under the commutator assumption alone.

\begin{corollary}\label{cor:full-depth-positive-endpoint}
    Suppose that $0\leq\beta\leq\alpha\leq1$, $\beta<1$, and that Assumption~\ref{ass:comm-filter-ass} holds. For every $\kappa\in(0,1)$ there exists a constant $C>0$ such that, whenever $\tau\in W^{1,\infty}(\R^2;\R^2)$ satisfies $\norm{D\tau}_{L^\infty}\leq\kappa$, we have, for every $f\in H^{1-\beta}(\R^2)$,
    \begin{equation*}
        \norm{\Scat(L_\tau f)-\Scat f}_{\lp}\leq C\bigl(\norm{\tau}_{L^\infty}+\norm{D\tau}_{L^\infty}\bigr)\norm{f}_{H^{1-\beta}}.
    \end{equation*}
    The same estimate holds with $\Scat$ replaced by $\ScatLe{N}$ for every $N\in\N_0$, with the same constant. The constant may depend on the fixed covering and bank, their regularity bounds, $\beta$, and $\kappa$, but not on $f$, $\tau$, or $N$.
\end{corollary}

\begin{proof}
    Since $0<1-\beta\leq1$, Lemma~\ref{lem:summability-above-critical} gives
    \begin{equation*}
        \mathfrak E_{1-\beta}(f)\leq\frac{\norm{f}_{H^{1-\beta}}}{1-\varrho^{1-\beta}}.
    \end{equation*}
    Combining this bound with Theorem~\ref{thm:abstract-full-depth-positive} and its finite-depth estimate proves the claim. More precisely, if $C_0$ denotes the constant in that theorem, we may take $C=C_0/(1-\varrho^{1-\beta})$.
\end{proof}

\begin{remark}
    The conclusion includes $\beta=0$, where the critical space is $H^1(\R^2)$. By the inclusion $H^s(\R^2)\subset H^{1-\beta}(\R^2)$, it also gives stability for every $s\geq1-\beta$ with $\norm{f}_{H^s}$ on the right-hand side. The proof only uses contraction at the same Sobolev order; it does not require the modulus layer to map a lower Sobolev order into $H^{1-\beta}$. In particular, radially elongated tiles at $\alpha=1$ do not obstruct this argument.

    At $\beta=1$, necessarily $\alpha=1$ and the critical exponent is zero. Then $\varrho^{1-\beta}=1$, so the geometric-series bound no longer applies; the commutator argument also requires $\beta<1$. No critical $L^2$ estimate at this corner, or uniformity as $\beta\uparrow1$, follows from the present proof.
\end{remark}

\section{Proof of the commutator bound} \label{sec:appendix-commutator}

The goal of this section is to prove Proposition~\ref{prop:scattering-commutator-plugin}. Throughout, we fix a covering $\bvQ$ and a subordinate filter bank, for instance the smooth model class from Section~\ref{subsec:designed-model}, and assume that Assumption~\ref{ass:comm-filter-ass} holds. We use the $C^{3,1}$ representatives of the high-pass multipliers $m_i$; identities, bounds, and support statements involving their fourth derivatives are understood almost everywhere.

\subsection{Preliminary estimates}

Let us start with a number of technical preliminaries. First, we need a scale-local vector-valued generalization of the Coifman--Meyer method \cite{Coifman1986} for boundedness of bilinear multipliers, cf.\ \cite[Theorem 7.5.3]{Grafakos}. 

\begin{lemma}\label{lem:scale-local-CM}
    Fix $0<R_0<R_1$ and $\lambda\geq 1$. Let $(\sigma_i)_{i\in J}$ be a countable family of measurable functions on $\R^2_\zeta\times\R^2_\eta$ such that, for almost every $\eta\in\R^2$, the section $\zeta\mapsto\sigma_i(\zeta,\eta)$ belongs to $W^{3,\infty}(B_{R_1\lambda}(0))$. Assume that
    \begin{equation*}
        A\coloneqq \max_{\abs{\gamma}\leq 3}\operatorname*{ess\,sup}_{\eta\in\R^2}\operatorname*{ess\,sup}_{\abs{\zeta}<R_1\lambda}\lambda^{\abs{\gamma}}\left(\sum_{i\in J}|\partial_\zeta^\gamma\sigma_i(\zeta,\eta)|^2\right)^{1/2}<\infty.
    \end{equation*}
    For Schwartz functions $g,h\in\calS(\R^2)$ with $\supp\wh g\subset B_{R_0\lambda}(0)$, define
    \begin{equation*}
        B_i(g,h)(x)=\iint_{\R^2\times\R^2}e^{2 \pi \iu x\cdot(\zeta+\eta)}\sigma_i(\zeta,\eta)\wh g(\zeta)\wh h(\eta)\dzeta\deta.
    \end{equation*}
    Then the operators $B_i$ admit a canonical extension to all $g\in L^\infty(\R^2)$ with $\supp\wh g\subset B_{R_0\lambda}(0)$, where $\widehat g$ is understood as a tempered distribution via the canonical embedding $L^\infty(\R^2)\hookrightarrow\calS'(\R^2)$, and to all $h\in L^2(\R^2)$. This extension satisfies
    \begin{equation*}
        \left(\sum_{i\in J}\norm{B_i(g,h)}_{L^2}^2\right)^{1/2}\leq C A\norm{g}_{L^\infty}\norm{h}_{L^2},
    \end{equation*}
    where $C>0$ depends only on $R_0$ and $R_1$.
\end{lemma}

\begin{proof}
    In this proof, statements concerning $\eta$ hold almost everywhere, and the $\zeta$-sections are represented by their continuous versions. Choose $L>2R_1$ and $\Theta\in C_c^\infty((-L/2,L/2)^2)$ such that $\supp\Theta\subset B_{R_1}(0)$ and $\Theta\equiv 1$ on $B_{R_0}(0)$. For $u\in(-L/2,L/2)^2$ and $\eta\in\R^2$, set
    \[\widetilde\sigma_i(u,\eta)=\Theta(u)\sigma_i(\lambda u,\eta).\]
    We extend $\widetilde\sigma_i(\cdot,\eta)$ $L$-periodically in $u$ and write its Fourier series in the form
    \begin{equation*}
        \widetilde\sigma_i(u,\eta)=\sum_{n\in\Z^2}a_{i,n}(\eta)e^{2 \pi \iu n\cdot u/L},\qquad a_{i,n}(\eta)=L^{-2}\int_{(-L/2,L/2)^2}\widetilde\sigma_i(u,\eta)e^{-2 \pi \iu n\cdot u/L}\du.
    \end{equation*}
    We first record the required estimate for the Fourier coefficients. For $0\leq q\leq 3$ and $r\in\set*{1,2}$, the weak product rule gives
    \begin{equation*}
        \partial_{u_r}^q\widetilde\sigma_i(u,\eta)=\sum_{\nu=0}^q\binom q\nu(\partial_{u_r}^{q-\nu}\Theta)(u)\lambda^\nu(\partial_{\zeta_r}^\nu\sigma_i)(\lambda u,\eta).
    \end{equation*}
    Since $\supp\Theta\subset B_{R_1}(0)$, the point $\lambda u$ belongs to $B_{R_1\lambda}(0)$ whenever $u\in\supp\Theta$. Hence the hypothesis implies
    \begin{equation*}
        \operatorname*{ess\,sup}_{\eta\in\R^2}\operatorname*{ess\,sup}_{u\in(-L/2,L/2)^2}\left(\sum_{i\in J}\abs{\partial_{u_r}^q\widetilde\sigma_i(u,\eta)}^2\right)^{1/2}\leq C_\Theta A,\qquad 0\leq q\leq 3,\quad r\in\set*{1,2}.
    \end{equation*}
    For $n=0$, Minkowski's integral inequality in $\ell^2(J)$ gives
    \begin{equation*}
        \left(\sum_{i\in J}\abs{a_{i,0}(\eta)}^2\right)^{1/2}\leq L^{-2}\int_{(-L/2,L/2)^2}\left(\sum_{i\in J}\abs{\widetilde\sigma_i(u,\eta)}^2\right)^{1/2}\du\leq C A.
    \end{equation*}
    If $n\neq 0$, choose $r\in\set*{1,2}$ such that $\abs{n_r}\geq \abs{n}/\sqrt 2$. Weak integration by parts three times in the variable $u_r$, with no boundary term because $\Theta$ is compactly supported, yields
    \begin{equation*}
        a_{i,n}(\eta)=\left(\frac{L}{2 \pi \iu  n_r}\right)^3 L^{-2}\int_{(-L/2,L/2)^2}\partial_{u_r}^3\widetilde\sigma_i(u,\eta)e^{-2 \pi \iu  n\cdot u/L}\du.
    \end{equation*}
    Using again Minkowski's integral inequality in $\ell^2(J)$, we obtain
    \begin{equation*}
        \left(\sum_{i\in J}\abs{a_{i,n}(\eta)}^2\right)^{1/2}\leq C A\abs{n}^{-3}.
    \end{equation*}
    Combining the cases $n=0$ and $n\neq 0$, we have proved
    \begin{equation*}
        \operatorname*{ess\,sup}_{\eta\in\R^2}\left(\sum_{i\in J}\abs{a_{i,n}(\eta)}^2\right)^{1/2}\leq C A(1+\abs{n})^{-3},\qquad n\in\Z^2.
    \end{equation*}
    Let $M_{i,n}$ denote the linear Fourier multiplier with symbol $a_{i,n}$, that is,
    \[\wh{M_{i,n}h}(\eta)=a_{i,n}(\eta)\wh h(\eta).\]
    By Plancherel's theorem and the preceding coefficient estimate,
    \begin{equation*}
        \left(\sum_{i\in J}\norm{M_{i,n}h}_{L^2}^2\right)^{1/2}=\left(\int_{\R^2}\sum_{i\in J}\abs{a_{i,n}(\eta)}^2\abs{\wh h(\eta)}^2\deta\right)^{1/2}\leq C A(1+\abs{n})^{-3}\norm{h}_{L^2}.
    \end{equation*}
    Now assume first that $g,h\in\calS(\R^2)$ and $\supp\wh g\subset B_{R_0\lambda}(0)$. Since $\Theta(\zeta/\lambda)=1$ on $\supp\wh g$, the Fourier expansion above gives, for every $\zeta\in\supp\wh g$,
    \begin{equation*}
        \sigma_i(\zeta,\eta)=\sum_{n\in\Z^2}a_{i,n}(\eta)e^{2 \pi \iu  n\cdot\zeta/(L\lambda)}.
    \end{equation*}
    The coefficient estimate implies absolute convergence of this series, uniformly outside a null set in $\eta$, with values in $\ell^2(J)$. Therefore we may insert the expansion into the definition of $B_i(g,h)$ and obtain
    \begin{equation*}
        B_i(g,h)=\sum_{n\in\Z^2}G_n\,M_{i,n}h,\qquad G_n(x)=\calF^{-1}\left(e^{2 \pi \iu  n\cdot\zeta/(L\lambda)}\wh g(\zeta)\right)(x)=g\left(x+\frac{n}{L\lambda}\right).
    \end{equation*}
    Consequently $\norm{G_n}_{L^\infty}=\norm{g}_{L^\infty}$ for all $n\in\Z^2$. Hence, by the triangle inequality in $\ell^2(J;L^2(\R^2))$,
    \begin{align*}
        \left(\sum_{i\in J}\norm{B_i(g,h)}_{L^2}^2\right)^{1/2}&\leq \sum_{n\in\Z^2}\norm{\left(G_nM_{i,n}h\right)_{i\in J}}_{L^2(\ell^2(J))} \\
        &\leq \norm{g}_{L^\infty}\sum_{n\in\Z^2}\left(\sum_{i\in J}\norm{M_{i,n}h}_{L^2}^2\right)^{1/2} \\
        &\leq C A\norm{g}_{L^\infty}\norm{h}_{L^2}\sum_{n\in\Z^2}(1+\abs{n})^{-3}.
    \end{align*}
    The final sum is finite.

    For general $g\in L^\infty(\R^2)$ with $\supp\wh g\subset B_{R_0\lambda}(0)$ and $h\in L^2(\R^2)$, the same formula
    \begin{equation*}
        B_i(g,h)=\sum_{n\in\Z^2}g\left(\cdot+\frac{n}{L\lambda}\right)M_{i,n}h
    \end{equation*}
    defines an unconditionally convergent series in $\ell^2(J;L^2(\R^2))$, by the estimate just proved for the summands. This gives the announced extension and the same bound. For Schwartz inputs the extension agrees with the original oscillatory integral, as shown above.
\end{proof}

The second ingredient is a Calder\'on-type result that will later allow us to handle the (finitely many) low-scale channels in the main commutator argument. Recall that, for $i\in I_\ast$, $m_i=\widehat\psi_i$ denotes the corresponding high-pass Fourier multiplier, and we set
\begin{equation*}
    \calC_i h \coloneqq h*\psi_i=\calF ^{-1}(m_i\widehat h),\qquad h\in L^2(\R^2).
\end{equation*}
\begin{lemma} \label{lem:fixed-low-scale-commutators}
    Given a finite subset $J\subset I_\ast$, assume that the multipliers $m_i$, $i \in J$, have compact support and belong to $W^{4,\infty}(\R^2)$. Then, there exists a constant $C>0$ such that, for every $b\in W^{1,\infty}(\R^2;\R^2)$ and $h\in L^2(\R^2)$, the (weakly defined) commutators $[\calC_i,b\cdot\nabla]h$ satisfy
    \begin{equation*}
        \left(\sum_{i\in J}\|[\calC_i,b\cdot\nabla]h\|_{L^2}^2\right)^{1/2} \leq C\|Db\|_{L^\infty}\|h\|_{L^2}.
    \end{equation*}
\end{lemma}

\begin{proof}
    It suffices to prove the claim for one index only, which will be omitted to lighten the notation. Note that the assumptions imply $m\in W^{4,1}(\R^2)$ and $\xi_s m(\xi) \in W^{4,1}(\R^2)$ for each coordinate $s$. Setting $K\coloneqq \calF ^{-1}m$, integration by parts in the Fourier inversion formula gives
    \begin{equation*}
        \abs{K(x)}\lesssim (1+\abs{x})^{-4}, \qquad |\nabla K(x)|\lesssim (1+\abs{x})^{-4},
    \end{equation*}
    ensuring $K\in L^1(\R^2)$ and $x\mapsto \abs{x}\nabla K(x)\in L^1(\R^2)$. 

    Write $\calC h=K*h$ for the fixed multiplier under consideration. Choosing $h\in\calS(\R^2)$ and $b\in C_b^\infty(\R^2;\R^2)$ initially, we have
    \begin{equation*}
        [\calC,b\cdot\nabla]h(x)=\sum_{s=1}^2 \int K(x-y)(b_s(y)-b_s(x)) \partial_s h(y)\dd{y}.
    \end{equation*}
    Integrating by parts in $y_s$ gives a sum of two terms: 
    The first has kernel $K(x-y)\partial_s b_s(y)$ and is bounded on $L^2$ by $\norm{K}_{L^1}\|Db\|_{L^\infty}\|h\|_{L^2}$. The second has kernel $\partial_sK(x-y)(b_s(y)-b_s(x))$, and we use
    \[\abs{b_s(y)-b_s(x)}\leq\norm{Db}_{L^\infty}\abs{x-y}.\]
    Thus the corresponding operator is dominated by convolution with $\norm{Db}_{L^\infty}\abs{\cdot}\abs{\partial_sK}$, and Young's inequality gives the bound
    \begin{equation*}
        \norm{\abs{\cdot}\nabla K}_{L^1}\norm{Db}_{L^\infty}\norm{h}_{L^2}.
    \end{equation*}
    Summing over $s$, we obtain
    \begin{equation*}
        \norm{[\calC,b\cdot\nabla]h}_{L^2}\lesssim \norm{Db}_{L^\infty}\norm{h}_{L^2}.
    \end{equation*}
    The same estimate in the case $b\in W^{1,\infty}(\R^2;\R^2)$ follows by standard mollification, and summing over the finite set $J$ proves the claim for $h\in\calS(\R^2)$; the claim extends to $h\in L^2(\R^2)$ by density.
\end{proof}

\subsection{The linear commutator bound}

We are now in a position to state the main commutator estimate at the infinitesimal deformation level. The result quantifies the following heuristic mechanism: differentiating a multiplier at generation $j$ costs the transverse packet scale $2^{-\beta j}$, while the derivative in the vector field contributes the carrier size $2^j$. Their combination produces the critical loss $2^{(1-\beta)j}$, which is captured by the Sobolev exponent $1-\beta$. We thus obtain the wave-packet analogue of Mallat's wavelet commutator estimate \cite[Lemma~2.14]{Mallat2012}, which is a key ingredient in the proof of deformation stability \cite[Theorem~2.12]{Mallat2012}. In the wavelet case, the multiplier derivative costs $2^{-j}$ and therefore cancels the carrier contribution $2^j$. The proof uses paradifferential methods \cite{Taylor-PDC}, with estimates adapted to the transverse scales of the wave-packet covering. We do not assert any uniformity of the constant as $\beta\uparrow1$; indeed, the proof exhibits degeneracies as $1-\beta\downarrow0$.

\begin{proposition} \label{prop:critical-vector-field-commutator}
    Under Assumption~\ref{ass:comm-filter-ass}, there exists a constant $C>0$ such that, for every vector field $b\in W^{1,\infty}(\R^2;\R^2)$ and $h\in H^{1-\beta}(\R^2)$,
    \begin{equation*}
        \left(\sum_{i\in I_\ast}\|[\calC_i,b\cdot\nabla]h\|_{L^2}^2\right)^{1/2} \le  C \|Db\|_{L^\infty}\|h\|_{H^{1-\beta}}.
    \end{equation*}
\end{proposition}

\begin{proof}
    We first prove the estimate for $b\in C_c^\infty(\R^2;\R^2)$ and $h\in\calS(\R^2)$. All implicit constants below may depend on the constants in Assumption~\ref{ass:comm-filter-ass}, on the fixed Littlewood--Paley cutoffs, on the auxiliary cutoffs introduced below, and on $\beta$, but not on $b$ or $h$.

    Fix an inhomogeneous Littlewood--Paley partition $(P_k)_{k\geq0}$ on frequency space, with $P_0$ supported near the origin and $P_k$ supported where $\abs{\xi}\simeq2^k$ for $k\geq1$. We write $P_{\leq N}=\sum_{0\leq k\leq N}P_k$. The finitely many indices in $I_{\rm fin}\coloneqq \bigcup_{1\leq j<j_{\rm reg}}I_{\ast,j}$ are handled by Lemma~\ref{lem:fixed-low-scale-commutators}. For $j\geq j_{\rm reg}$, set
    \begin{equation*}
        N_j\coloneqq \max\{0,\lfloor \beta j\rfloor\},\qquad b_j^{\rm lo}\coloneqq P_{\leq N_j}b,\qquad b_j^{\rm hi}\coloneqq b-b_j^{\rm lo},\qquad \lambda_j\coloneqq 2^{N_j}.
    \end{equation*}

    By the geometry of the covering, there are constants $0<c<C$ such that
    \begin{equation*}
        \supp m_i\subset Q_i\subset\set*{\xi\in\R^2\given c2^j\leq\abs{\xi}\leq C2^j},\qquad i\in I_{\ast,j}.
    \end{equation*}
    Since $\supp\partial^\nu m_i\subset\supp m_i$, the condition $\partial^\nu m_i(\eta+t\zeta)\neq0$ implies $c2^j\leq\abs{\eta+t\zeta}\leq C2^j$. Let $A_{\rm LP}\ge 1$ be such that $\supp\wh{P_{\leq N}f}\subset B_{A_{\rm LP}2^N}(0)$ for every $N\geq0$. Here the notation $\abs{\zeta}\lesssim\lambda_j$ below always means $\abs{\zeta}\leq A_{\rm LP}\lambda_j$, with this fixed constant. We shall apply Lemma~\ref{lem:scale-local-CM} with $\lambda=A_{\rm LP}\lambda_j$, $R_0=1$ and $R_1=2$. Choose $j_{\rm aux}\geq j_{\rm reg}$ so large that, for all $j\geq j_{\rm aux}$, $2 A_{\rm LP}\lambda_j\leq\eps2^j$ with $\eps<c/2$. This is possible because $\lambda_j/2^j\to0$ for $0\leq\beta<1$. Hence, whenever $\abs{\zeta}\leq 2 A_{\rm LP}\lambda_j$ and $\partial^\nu m_i(\eta+t\zeta)\neq0$ with $0\leq t\leq1$, we have
    \[(c-\eps)2^j\leq\abs{\eta}\leq(C+\eps)2^j.\]
    Consequently, there are constants $0<c'<C'$ such that, whenever $j\geq j_{\rm aux}$, $i\in I_{\ast,j}$, $\abs{\zeta}\leq 2 A_{\rm LP}\lambda_j$, $0\leq t\leq1$, and $\partial^\nu m_i(\eta+t\zeta)\neq0$ for some $1\leq\abs{\nu}\leq4$, one has
    \[\eta\in\set*{\xi\in\R^2\given c'2^j\leq\abs{\xi}\leq C'2^j}.\]
    Let $\widetilde P_j$ be a smooth annular Fourier projection which is equal to one on this annulus and supported in a fixed slightly larger annulus at the same dyadic scale. The family $(\widetilde P_j)_j$ has uniformly finite overlap. The remaining finitely many generations $j_{\rm reg}\leq j<j_{\rm aux}$ form the auxiliary set $I_{\rm aux}\coloneqq \bigcup_{j_{\rm reg}\leq j<j_{\rm aux}}I_{\ast,j}$ and are again covered by Lemma~\ref{lem:fixed-low-scale-commutators}. It remains to estimate the generations $j\geq j_{\rm aux}$.

    \bigskip
    \textbf{Step 1.} \textit{Low-frequency terms.}

    We first treat the low-frequency coefficient terms $[\calC_i,b_j^{\rm lo}\cdot\nabla]h$. For a single component $b_{j,s}^{\rm lo}\partial_s$, the Fourier representation is
    \begin{equation*}
        \wh{[\calC_i,b_{j,s}^{\rm lo}\partial_s]h}(\omega)=2\pi\iu\int_{\R^2}\wh b_{j,s}^{\rm lo}(\zeta)\eta_s\bigl(m_i(\eta+\zeta)-m_i(\eta)\bigr)\wh h(\eta)\deta,\qquad \zeta=\omega-\eta.
    \end{equation*}
    Since
    \begin{equation*}
        m_i(\eta+\zeta)-m_i(\eta)=\sum_{r=1}^2\zeta_r\int_0^1\partial_rm_i(\eta+t\zeta)\dt,\qquad \wh{P_{\leq N_j}\partial_rb_s}(\zeta)=2\pi\iu\,\zeta_r\wh b_{j,s}^{\rm lo}(\zeta),
    \end{equation*}
    we can write
    \begin{equation*}
        [\calC_i,b_j^{\rm lo}\cdot\nabla]h=\sum_{r,s=1}^2\int_0^1 B_{i,r,s,t,j}(P_{\leq N_j}\partial_rb_s,h)\dt,
    \end{equation*}
    where
    \begin{equation*}
        \wh{B_{i,r,s,t,j}(g,u)}(\omega)=\int_{\R^2}\wh g(\zeta)\sigma_{i,r,s,t,j}(\zeta,\eta)\wh u(\eta)\deta,\qquad \zeta=\omega-\eta,
    \end{equation*}
    and
    \begin{equation*}
        \sigma_{i,r,s,t,j}(\zeta,\eta)\coloneqq \eta_s\partial_rm_i(\eta+t\zeta),\qquad 0\leq t\leq1.
    \end{equation*}
    The Fourier support of $P_{\leq N_j}\partial_rb_s$ is contained in $B_{A_{\rm LP}\lambda_j}(0)$. Therefore, by the localization established above, whenever the integrand is nonzero, the variable $\eta$ lies in the annulus on which $\widetilde P_j$ is identically one. Thus $h$ may be replaced by $\widetilde P_jh$, and on the same support $\abs{\eta_s}\lesssim2^j$.

    For $\abs{\gamma}\leq3$,
    \begin{equation*}
        \partial_\zeta^\gamma\sigma_{i,r,s,t,j}(\zeta,\eta)=t^{\abs{\gamma}}\eta_s\partial^{\gamma+e_r}m_i(\eta+t\zeta).
    \end{equation*}
    For $\abs{\zeta}\leq 2A_{\rm LP}\lambda_j$, the localization established above shows that, whenever $\partial^{\gamma+e_r}m_i(\eta+t\zeta)\neq0$, one has $\abs{\eta_s}\lesssim2^j$. Since $\abs{\gamma}+1\leq4$, Assumption~\ref{ass:comm-filter-ass} gives, uniformly in $t\in[0,1]$,
    \begin{equation*}
        \left(\sum_{i\in I_{\ast,j}}\abs{\partial_\zeta^\gamma\sigma_{i,r,s,t,j}(\zeta,\eta)}^2\right)^{1/2}\lesssim2^j2^{-\beta j(\abs{\gamma}+1)}.
    \end{equation*}
    Since $A_{\rm LP}$ is fixed and $\lambda_j\lesssim2^{\beta j}$, this implies
    \begin{equation*}
        (A_{\rm LP}\lambda_j)^{\abs{\gamma}}\left(\sum_{i\in I_{\ast,j}}\abs{\partial_\zeta^\gamma\sigma_{i,r,s,t,j}(\zeta,\eta)}^2\right)^{1/2}\lesssim2^{(1-\beta)j},\qquad \abs{\gamma}\leq3,
    \end{equation*}
    uniformly for $\abs{\zeta}\leq 2 A_{\rm LP}\lambda_j$, $\eta\in\R^2$, and $t\in[0,1]$.
    Moreover,
    \begin{equation*}
        \norm{P_{\leq N_j}\partial_rb_s}_{L^\infty}\lesssim\norm{Db}_{L^\infty}.
    \end{equation*}
    Since $\supp\wh{P_{\leq N_j}\partial_rb_s}\subset B_{A_{\rm LP}\lambda_j}(0)$, Lemma~\ref{lem:scale-local-CM} applies with
    \begin{equation*}
        \lambda=A_{\rm LP}\lambda_j,\qquad R_0=1,\qquad R_1=2,\qquad g=P_{\leq N_j}\partial_rb_s,\qquad h=\widetilde P_jh.
    \end{equation*}
    Therefore,
    \begin{equation*}
        \left(\sum_{i\in I_{\ast,j}}\norm{B_{i,r,s,t,j}(P_{\leq N_j}\partial_rb_s,\widetilde P_jh)}_{L^2}^2\right)^{1/2}\lesssim\norm{Db}_{L^\infty}2^{(1-\beta)j}\norm{\widetilde P_jh}_{L^2},
    \end{equation*}
    uniformly in $r,s,t$, and $j$.
    Minkowski's inequality handles the integration in $t$, and the finite sum over $r,s$ only changes the implicit constant. Hence
    \begin{equation*}
        \left(\sum_{i\in I_{\ast,j}}\norm{[\calC_i,b_j^{\rm lo}\cdot\nabla]h}_{L^2}^2\right)^{1/2}\lesssim\norm{Db}_{L^\infty}2^{(1-\beta)j}\norm{\widetilde P_jh}_{L^2},\qquad j\geq j_{\rm aux}.
    \end{equation*}
    Squaring and summing over $j\geq j_{\rm aux}$, and using the uniformly finite overlap of $(\widetilde P_j)_j$, gives
    \begin{equation*}
        \sum_{j\geq j_{\rm aux}}\sum_{i\in I_{\ast,j}}\norm{[\calC_i,b_j^{\rm lo}\cdot\nabla]h}_{L^2}^2\lesssim\norm{Db}_{L^\infty}^2\sum_{j\geq j_{\rm aux}}2^{2(1-\beta)j}\norm{\widetilde P_jh}_{L^2}^2\lesssim\norm{Db}_{L^\infty}^2\norm{h}_{H^{1-\beta}}^2.
    \end{equation*}
    
    \bigskip
    \textbf{Step 2.} \textit{High-frequency terms.}

    We next estimate the high-frequency coefficient terms. Let $\Pi_j$ be a smooth annular Fourier projection such that $\Pi_j\equiv1$ on $\bigcup_{i\in I_{\ast,j}}\supp m_i$ and such that $\Pi_j$ is supported in a fixed annular enlargement at scale $2^j$. Since the filter bank satisfies the Littlewood--Paley condition, $\sum_{i\in I_{\ast,j}}\abs{m_i(\xi)}^2\leq1$. Therefore, for every $g\in L^2(\R^2)$,
    \begin{equation*}
        \sum_{i\in I_{\ast,j}}\norm{\calC_ig}_{L^2}^2=\int_{\R^2}\sum_{i\in I_{\ast,j}}\abs{m_i(\xi)}^2\abs{\wh{\Pi_jg}(\xi)}^2\dxi\leq\norm{\Pi_jg}_{L^2}^2.
    \end{equation*}
    Hence
    \begin{equation}\label{eq:genwise_square-fct_estimate_vs_proj}
        \left(\sum_{i\in I_{\ast,j}}\norm{\calC_ig}_{L^2}^2\right)^{1/2}\leq\norm{\Pi_jg}_{L^2}.
    \end{equation}
    Write $b_q\coloneqq P_qb$ and $h_k\coloneqq P_kh$. For $q\geq1$, Bernstein's inequality gives
    \[\norm{b_q}_{L^\infty}\lesssim2^{-q}\norm{\nabla b}_{L^\infty}.\]
    Indeed, on the support of $P_q$ one can write $b_q=2^{-q}\sum_{\ell=1}^2\widetilde P_{q,\ell}\partial_\ell b$ with multipliers $\widetilde P_{q,\ell}$ uniformly bounded on $L^\infty$.

    We decompose $b_j^{\rm hi}\partial_sh$ by Bony's paraproducts \cite{Bony1981}. Since $b_j^{\rm hi}=\sum_{q>N_j}b_q$ and $h=\sum_{k\geq0}h_k$,
    \begin{equation*}
        \Pi_j(b_j^{\rm hi}\partial_sh)=\mathrm{LH}_{j,s}+\mathrm{HL}_{j,s}+\mathrm{HH}_{j,s},
    \end{equation*}
    where
    \begin{equation*}
        \mathrm{LH}_{j,s}\coloneqq \sum_{\substack{q>N_j\\ q\leq k-4}}\Pi_j(b_q\partial_sh_k),\qquad \mathrm{HL}_{j,s}\coloneqq \sum_{\substack{q>N_j\\ k\leq q-4}}\Pi_j(b_q\partial_sh_k),\qquad \mathrm{HH}_{j,s}\coloneqq \sum_{\substack{q>N_j\\ \abs{q-k}\leq3}}\Pi_j(b_q\partial_sh_k).
    \end{equation*}
    In what follows, $C_0>0$ denotes a fixed integer depending only on the supports of the Littlewood--Paley cutoffs and on the auxiliary annular projections; its value may change from line to line.

    In the low-high region, $b_q\partial_sh_k$ is localized at frequency $\asymp2^k$, and $\Pi_j$ forces $\abs{k-j}\leq C_0$. Therefore
    \begin{equation*}
        \norm{\mathrm{LH}_{j,s}}_{L^2}\lesssim\norm{Db}_{L^\infty}\sum_{\abs{k-j}\leq C_0}2^k\norm{h_k}_{L^2}\sum_{\substack{q>N_j\\ q\leq k-4}}2^{-q}.
    \end{equation*}
    Since $\sum_{q>N_j}2^{-q}\lesssim2^{-N_j}\lesssim2^{-\beta j}$, we obtain
    \begin{equation*}
        \norm{\mathrm{LH}_{j,s}}_{L^2}\lesssim\norm{Db}_{L^\infty}2^{(1-\beta)j}\sum_{\abs{k-j}\leq C_0}\norm{h_k}_{L^2}.
    \end{equation*}
    In the high-low region, $b_q\partial_sh_k$ is localized at frequency $\asymp2^q$, and $\Pi_j$ forces $\abs{q-j}\leq C_0$. Hence
    \begin{equation*}
        \norm{\mathrm{HL}_{j,s}}_{L^2}\lesssim\norm{Db}_{L^\infty}\sum_{\substack{\abs{q-j}\leq C_0\\ k\leq q-4}}2^{-q}2^k\norm{h_k}_{L^2}\lesssim\norm{Db}_{L^\infty}\sum_{k\leq j+C_0}2^{-j}2^k\norm{h_k}_{L^2}.
    \end{equation*}
    In the resonant region $\abs{q-k}\leq3$, cancellations may generate lower output frequencies, so $\Pi_j(b_q\partial_sh_k)\neq0$ does not force $\abs{j-k}\leq C_0$. It does, however, imply the one-sided relation $j\leq k+C_0$. Therefore
    \begin{equation*}
        \norm{\mathrm{HH}_{j,s}}_{L^2}\lesssim\norm{Db}_{L^\infty}\sum_{k\geq j-C_0}\norm{h_k}_{L^2}.
    \end{equation*}
    \textit{Controlling $\calC_i(b_j^{\rm hi}\cdot\nabla h)$.} Set $u_k\coloneqq 2^{(1-\beta)k}\norm{h_k}_{L^2}$. Up to the common factor $\norm{Db}_{L^\infty}$, the preceding estimates become
    \begin{equation*}
        \norm{\mathrm{LH}_{j,s}}_{L^2}\lesssim\sum_{\abs{k-j}\leq C_0}2^{(1-\beta)(j-k)}u_k,\qquad \norm{\mathrm{HL}_{j,s}}_{L^2}\lesssim\sum_{k\leq j+C_0}2^{-j+\beta k}u_k,
    \end{equation*}
    and
    \begin{equation*}
        \norm{\mathrm{HH}_{j,s}}_{L^2}\lesssim\sum_{k\geq j-C_0}2^{-(1-\beta)k}u_k.
    \end{equation*}
    The low-high operator is bounded on $\ell^2(\Nzero)$ because it has finite bandwidth. The high-low operator is bounded on $\ell^2(\Nzero)$ by Schur's test, since
    \begin{equation*}
        \sup_j\sum_{k\leq j+C_0}2^{-j+\beta k}<\infty,\qquad \sup_k\sum_{j\geq k-C_0}2^{-j+\beta k}<\infty.
    \end{equation*}
    The same argument applies to the resonant kernel $\one_{\{k\geq j-C_0\}}2^{-(1-\beta)k}$: for fixed $j$,
    \[\sum_{k\geq j-C_0}2^{-(1-\beta)k}\lesssim 2^{-(1-\beta)j},\]
    while for fixed $k$,
    \begin{equation*}
        \sum_{j\leq k+C_0}2^{-(1-\beta)k}\lesssim(k+1)2^{-(1-\beta)k}\lesssim 1.
    \end{equation*}
    Thus the resonant contribution is also bounded on $\ell^2$.

    Consequently,
    \begin{equation*}
        \sum_{j\geq j_{\rm aux}}\norm{\Pi_j(b_j^{\rm hi}\cdot\nabla h)}_{L^2}^2\lesssim\norm{Db}_{L^\infty}^2\sum_{k\geq0}2^{2(1-\beta)k}\norm{h_k}_{L^2}^2\lesssim\norm{Db}_{L^\infty}^2\norm{h}_{H^{1-\beta}}^2.
    \end{equation*}
    Combining this with the estimate from \eqref{eq:genwise_square-fct_estimate_vs_proj} yields
    \begin{equation*}
        \sum_{j\geq j_{\rm aux}}\sum_{i\in I_{\ast,j}}\norm{\calC_i(b_j^{\rm hi}\cdot\nabla h)}_{L^2}^2\lesssim\norm{Db}_{L^\infty}^2\norm{h}_{H^{1-\beta}}^2.
    \end{equation*}
    \textit{Controlling $b_j^{\rm hi}\cdot\nabla\calC_ih$.} It remains to estimate the second high-frequency term, $b_j^{\rm hi}\cdot\nabla\calC_ih$. Bernstein's inequality gives
    \begin{equation*}
        \norm{b_j^{\rm hi}}_{L^\infty}\leq\sum_{q>N_j}\norm{b_q}_{L^\infty}\lesssim\norm{Db}_{L^\infty}\sum_{q>N_j}2^{-q}\lesssim\norm{Db}_{L^\infty}2^{-N_j}\lesssim\norm{Db}_{L^\infty}2^{-\beta j}.
    \end{equation*}
    Using the fixed-generation Littlewood--Paley bound with one derivative and the support condition $\supp m_i\subset\set*{\abs{\xi}\lesssim2^j}$ for $i\in I_{\ast,j}$, we obtain
    \begin{equation*}
        \left(\sum_{i\in I_{\ast,j}}\norm{\nabla\calC_ih}_{L^2}^2\right)^{1/2}\lesssim2^j\norm{\Pi_jh}_{L^2}.
    \end{equation*}
    Hence
    \begin{equation*}
        \left(\sum_{i\in I_{\ast,j}}\norm{b_j^{\rm hi}\cdot\nabla\calC_ih}_{L^2}^2\right)^{1/2}\lesssim\norm{Db}_{L^\infty}2^{(1-\beta)j}\norm{\Pi_jh}_{L^2}.
    \end{equation*}
    Squaring and summing in $j$, and using the uniformly finite overlap of $(\Pi_j)_j$, gives
    \begin{equation*}
        \sum_{j\geq j_{\rm aux}}\sum_{i\in I_{\ast,j}}\norm{b_j^{\rm hi}\cdot\nabla\calC_ih}_{L^2}^2\lesssim\norm{Db}_{L^\infty}^2\sum_{j\geq j_{\rm aux}}2^{2(1-\beta)j}\norm{\Pi_jh}_{L^2}^2\lesssim\norm{Db}_{L^\infty}^2\norm{h}_{H^{1-\beta}}^2.
    \end{equation*}
    
    \bigskip
    \textbf{Step 3.} \textit{Conclusion.}

    Putting together the low-frequency estimate, the two high-frequency estimates, and the finite contributions from $I_{\rm fin}\cup I_{\rm aux}$, we obtain
    \begin{equation*}
        \left(\sum_{i\in I_\ast}\norm{[\calC_i,b\cdot\nabla]h}_{L^2}^2\right)^{1/2}\leq C\norm{Db}_{L^\infty}\norm{h}_{H^{1-\beta}},
    \end{equation*}
    for all $b\in C_c^\infty(\R^2;\R^2)$ and $h\in\calS(\R^2)$, where $C>0$ is independent of $b$ and $h$.

    To extend this estimate to $b\in C_b^\infty(\R^2;\R^2)$, choose $\vartheta\in C_c^\infty(\R^2)$ with $0\leq\vartheta\leq1$ and $\vartheta=1$ near the origin, and set $b_R\coloneqq\vartheta(\cdot/R)b$ for $R\geq1$. Then $b_R\in C_c^\infty(\R^2;\R^2)$ and
    \begin{equation*}
        \norm{Db_R}_{L^\infty}\leq\norm{Db}_{L^\infty}+R^{-1}\norm{\nabla\vartheta}_{L^\infty}\norm{b}_{L^\infty}.
    \end{equation*}
    For each fixed $i\in I_\ast$ and $h\in\calS(\R^2)$, both $(b_R-b)\cdot\nabla h$ and $(b_R-b)\cdot\nabla\calC_i h$ tend to zero in $L^2$ by dominated convergence. Here $\nabla\calC_i h\in L^2$ because $m_i$ is bounded. Since $\calC_i$ is bounded on $L^2$, it follows that $[\calC_i,b_R\cdot\nabla]h\to[\calC_i,b\cdot\nabla]h$ in $L^2$. Applying the preceding estimate to $b_R$, restricting the left-hand sum to a finite subset of $I_\ast$, and letting $R\to\infty$ gives the same bound on that subset with $\norm{Db}_{L^\infty}$. Monotone convergence over finite subsets of $I_\ast$ proves the full estimate for $b\in C_b^\infty(\R^2;\R^2)$.

    It remains to pass from smooth vector fields to arbitrary $b\in W^{1,\infty}(\R^2;\R^2)$. Let $\rho\in C_c^\infty(\R^2)$ be a standard mollifier with $\rho\geq0$ and $\int_{\R^2}\rho(x)\dx=1$, set $\rho_\eps(x)\coloneqq \eps^{-2}\rho(x/\eps)$, and define $b_\eps\coloneqq \rho_\eps*b$ componentwise. Then $b_\eps\in C_b^\infty(\R^2;\R^2)$, $\norm{Db_\eps}_{L^\infty}\leq\norm{Db}_{L^\infty}$, and $b_\eps\to b$ locally uniformly, after choosing the Lipschitz representative of $b$. Fix $h\in\calS(\R^2)$ and set $F_\eps\coloneqq ([\calC_i,b_\eps\cdot\nabla]h)_{i\in I_\ast}$. The estimate already proved shows that $(F_\eps)_\eps$ is bounded in $\ell^2(I_\ast;L^2)$, uniformly in $\eps$. Hence, along any sequence $\eps_n\downarrow0$, we may extract a subsequence, not relabeled, such that $F_{\eps_n}$ converges weakly in $\ell^2(I_\ast;L^2)$ to some $F$. For each fixed $i$, the $i$-th component $[\calC_i,b_{\eps_n}\cdot\nabla]h$ converges distributionally to $[\calC_i,b\cdot\nabla]h$, because $b_{\eps_n}\to b$ locally uniformly and $h$ is Schwartz. On the other hand, the coordinate projection from $\ell^2(I_\ast;L^2)$ to $L^2$ is continuous, so the same component converges weakly in $L^2$ to $F_i$. Thus $F_i=[\calC_i,b\cdot\nabla]h$ as a distribution for every fixed $i$. Weak lower semicontinuity in $\ell^2(I_\ast;L^2)$ then gives
    \begin{equation*}
        \left(\sum_{i\in I_\ast}\norm{[\calC_i,b\cdot\nabla]h}_{L^2}^2\right)^{1/2}=\norm{F}_{\ell^2(I_\ast;L^2)}\leq\liminf_{n\to\infty}\norm{F_{\eps_n}}_{\ell^2(I_\ast;L^2)}\leq C\norm{Db}_{L^\infty}\norm{h}_{H^{1-\beta}}.
    \end{equation*}
    Therefore the desired estimate holds for every $b\in W^{1,\infty}(\R^2;\R^2)$ and $h\in\calS(\R^2)$, with the same constant $C$.

    Finally, let $h\in H^{1-\beta}(\R^2)$ and choose $h_n\in\calS(\R^2)$ with $h_n\to h$ in $H^{1-\beta}(\R^2)$. The preceding estimate applied to $h_n-h_m$ shows that $([\calC_i,b\cdot\nabla]h_n)_i$ is Cauchy in $\ell^2(I_\ast;L^2)$. Its limit defines $([\calC_i,b\cdot\nabla]h)_i$ by completion; the definition is independent of the chosen approximating sequence by the same difference estimate. Passing to the limit yields
    \begin{equation*}
        \left(\sum_{i\in I_\ast}\norm{[\calC_i,b\cdot\nabla]h}_{L^2}^2\right)^{1/2}\leq C\norm{Db}_{L^\infty}\norm{h}_{H^{1-\beta}},
    \end{equation*}
    which proves the claim.
\end{proof}

\subsection{Lifting to scattering}

We are now ready to transfer the commutator bound to a stability result for the scattering transform. Let us first transfer stability to high-pass scattering filters. We follow the commutator approach to scattering stability developed in Mallat \cite[Section~2.5, Lemmas~2.13--2.14]{Mallat2012}. Here, the high-pass deformation estimate follows from Proposition~\ref{prop:critical-vector-field-commutator} by conjugation along $F_t=\Id-t\tau$.

\begin{proposition} \label{prop:appendix-high-pass-deformation-commutator}
    Let $\kappa\in(0,1)$. Under Assumption~\ref{ass:comm-filter-ass}, there exists a constant $C>0$ such that, whenever $\tau\in W^{1,\infty}(\R^2;\R^2)$ satisfies $\norm{D\tau}_{L^\infty} \le \kappa$, we have, for every $h\in H^{1-\beta}(\R^2)$,
    \begin{equation*}
        \left(\sum_{i\in I_\ast}\|(L_\tau h)*\psi_i-L_\tau(h*\psi_i)\|_{L^2}^2\right)^{1/2} \le  C \|D\tau\|_{L^\infty}\|h\|_{H^{1-\beta}}.
    \end{equation*}
    Therefore, for every countable family $q=(q_p)_{p\in P}\in\ell^2(H^{1-\beta})$,
    \begin{equation*}
        \|\calU L_\tau q-L_\tau\calU q\|_{\ell^2(P\times I_\ast;L^2)} \le  C \|D\tau\|_{L^\infty}\|q\|_{\ell^2(P;H^{1-\beta})}.
    \end{equation*}
\end{proposition}

\begin{proof}
    For $0\leq t\leq1$, set $F_t\coloneqq \Id-t\tau$ and $V_tf\coloneqq f\circ F_t$. We first assume that $\tau\in C_b^\infty(\R^2;\R^2)$, with $\norm{D\tau}_{L^\infty}\leq\kappa$, and that $h\in\calS(\R^2)$. By Lemma~\ref{lem:composition-small-deformation}, the maps $F_t$ are bi-Lipschitz homeomorphisms, and the operators $V_t$ and $V_t^{-1}$ are bounded on $H^{1-\beta}(\R^2)$, uniformly for $0\leq t\leq1$. Moreover,
    \[\norm{V_t}_{L^2\to L^2}\leq (1-\kappa)^{-1}.\]
    Define the time-dependent vector field $v_t\coloneqq -\tau\circ F_t^{-1},$
    and let $X_t\coloneqq v_t\cdot\nabla$ denote the associated first-order differential operator.
    Since $\partial_tF_t=-\tau$, we have, for every smooth $f$,
    \begin{equation*}
        \frac{\dd}{\dd t}V_tf=V_tX_tf, \qquad \frac{\dd}{\dd t}V_t^{-1}f=-X_tV_t^{-1}f.
    \end{equation*}
    For fixed $i\in I_\ast$, consider
    \[Y_i(t)\coloneqq V_t\calC_iV_t^{-1}V_1h.\]
    Then $Y_i(0)=\calC_iV_1h=(L_\tau h)*\psi_i$ and $Y_i(1)=V_1\calC_ih=L_\tau(h*\psi_i)$. Using the two identities above, we obtain
    \begin{equation*}
        \frac{\dd}{\dd t}Y_i(t)=V_tX_t\calC_iV_t^{-1}V_1h-V_t\calC_iX_tV_t^{-1}V_1h=-V_t[\calC_i,X_t]V_t^{-1}V_1h.
    \end{equation*}
    Hence
    \begin{equation*}
        (L_\tau h)*\psi_i - L_\tau (h*\psi_i)=\calC_iV_1h-V_1\calC_ih=\int_0^1V_t[\calC_i,X_t]V_t^{-1}V_1h\dt.
    \end{equation*}
    We now estimate the right-hand side in $\ell^2(I_\ast;L^2)$. By Minkowski's inequality,
    \begin{equation*}
        \left(\sum_{i\in I_\ast}\norm{(L_\tau h)*\psi_i - L_\tau (h*\psi_i)}_{L^2}^2\right)^{1/2}\leq\int_0^1\norm{V_t}_{L^2\to L^2}\left(\sum_{i\in I_\ast}\norm{[\calC_i,X_t]V_t^{-1}V_1h}_{L^2}^2\right)^{1/2}\dt.
    \end{equation*}
    Since $X_t=v_t\cdot\nabla$, Proposition~\ref{prop:critical-vector-field-commutator} gives
    \begin{equation*}
        \left(\sum_{i\in I_\ast}\norm{[\calC_i,X_t]V_t^{-1}V_1h}_{L^2}^2\right)^{1/2}\leq C\norm{Dv_t}_{L^\infty}\norm{V_t^{-1}V_1h}_{H^{1-\beta}}.
    \end{equation*}
    Furthermore,
    \[Dv_t=-(D\tau\circ F_t^{-1})DF_t^{-1},\]
    and therefore
    \begin{equation*}
        \norm{Dv_t}_{L^\infty}\leq\norm{D\tau}_{L^\infty}\operatorname{Lip}(F_t^{-1})\leq(1-\kappa)^{-1}\norm{D\tau}_{L^\infty}.
    \end{equation*}
    Lemma~\ref{lem:composition-small-deformation} also implies
    \begin{equation*}
        \norm{V_t^{-1}V_1h}_{H^{1-\beta}}\leq C_\kappa\norm{h}_{H^{1-\beta}},
    \end{equation*}
    uniformly in $t\in[0,1]$. Combining these estimates yields
    \begin{equation*}
        \left(\sum_{i\in I_\ast}\norm{(L_\tau h)*\psi_i-L_\tau(h*\psi_i)}_{L^2}^2\right)^{1/2}\leq C_\kappa\norm{D\tau}_{L^\infty}\norm{h}_{H^{1-\beta}}.
    \end{equation*}
    We next remove the smoothness assumption on $\tau$. Let $\rho\in C_c^\infty(\R^2)$ be a standard mollifier, set $\rho_\eps(x)\coloneqq \eps^{-2}\rho(x/\eps)$, and put $\tau_\eps\coloneqq \rho_\eps*\tau$ componentwise. Choosing the Lipschitz representative of $\tau$, we have $\tau_\eps\to\tau$ locally uniformly and $\norm{D\tau_\eps}_{L^\infty}\leq\norm{D\tau}_{L^\infty}\leq\kappa$. The maps $\Id-\tau_\eps$ and $\Id-\tau$ have uniformly controlled bi-Lipschitz constants, and a density argument gives $L_{\tau_\eps}u \to L_\tau u$ in $L^2(\R^2)$ for every $u\in L^2(\R^2)$. Therefore, the estimate just proved applies to $\tau_\eps$, with a constant independent of $\eps>0$. Fix $h\in\calS(\R^2)$ and set
    \begin{equation*}
        G_\eps\coloneqq \left((L_{\tau_\eps}h)*\psi_i-L_{\tau_\eps}(h*\psi_i)\right)_{i\in I_\ast}.
    \end{equation*}
    The family $(G_\eps)_{\eps>0}$ is bounded in $\ell^2(I_\ast;L^2)$. Thus, along every sequence $\eps_n\downarrow0$, we may extract a subsequence, not relabeled, such that $G_{\eps_n}$ converges weakly in $\ell^2(I_\ast;L^2)$ to some $G$. For each fixed $i$, strongly in $L^2(\R^2)$ we have
    \begin{equation*}
        (L_{\tau_{\eps_n}}h)*\psi_i-L_{\tau_{\eps_n}}(h*\psi_i)\longrightarrow (L_\tau h)*\psi_i-L_\tau(h*\psi_i).
    \end{equation*}
    Since the coordinate projection $\ell^2(I_\ast;L^2)\to L^2$ is continuous, the $i$-th component of the weak limit is the same distribution. Therefore $G_i=(L_\tau h)*\psi_i-L_\tau(h*\psi_i)$ for every fixed $i$. Weak lower semicontinuity gives
    \begin{equation*}
        \left(\sum_{i\in I_\ast}\norm{(L_\tau h)*\psi_i-L_\tau(h*\psi_i)}_{L^2}^2\right)^{1/2}\leq C_\kappa\norm{D\tau}_{L^\infty}\norm{h}_{H^{1-\beta}}.
    \end{equation*}
    Thus the scalar estimate holds for every $\tau\in W^{1,\infty}(\R^2;\R^2)$ and $h\in\calS(\R^2)$.

    Finally, let $h\in H^{1-\beta}(\R^2)$ and choose $h_n\in\calS(\R^2)$ with $h_n\to h$ in $H^{1-\beta}(\R^2)$. The estimate applied to $h_n-h_m$ shows that the vectors
    \[\left((L_\tau h_n)*\psi_i-L_\tau(h_n*\psi_i)\right)_{i\in I_\ast}\]
    are Cauchy in $\ell^2(I_\ast;L^2)$. Since $L_\tau$ and convolution by each $\psi_i$ are bounded on $L^2$, their $i$-th coordinates converge to $(L_\tau h)*\psi_i-L_\tau(h*\psi_i)$ for every $i\in I_\ast$; their limit is the natural commutator vector for $h$, and is independent of the chosen approximating sequence by the same difference estimate. Passing to the limit thus gives the asserted scalar bound.

    It remains to prove the array estimate. For each $p\in P$ and $i\in I_\ast$,
    \begin{equation*}
        (\calU L_\tau q)_{p,i}=\abs{(L_\tau q_p)*\psi_i},\qquad (L_\tau\calU q)_{p,i}=L_\tau\abs{q_p*\psi_i}=\abs{L_\tau(q_p*\psi_i)}.
    \end{equation*}
    Hence the pointwise inequality $\abs{\abs{z}-\abs{w}}\leq\abs{z-w}$ gives
    \begin{equation*}
        \abs{(\calU L_\tau q)_{p,i}-(L_\tau\calU q)_{p,i}}\leq\abs{(L_\tau q_p)*\psi_i-L_\tau(q_p*\psi_i)}.
    \end{equation*}
    Squaring, summing first over $i\in I_\ast$ and then over $p\in P$, and applying the scalar estimate to each $q_p$, we obtain
    \begin{equation*}
        \norm{\calU L_\tau q-L_\tau\calU q}_{\ell^2(P\times I_\ast;L^2)}^2\leq C_\kappa^2\norm{D\tau}_{L^\infty}^2\sum_{p\in P}\norm{q_p}_{H^{1-\beta}}^2.
    \end{equation*}
    This proves the second estimate and completes the proof.
\end{proof}

The corresponding treatment of the output-generating low-pass scattering filters is obtained, up to minor modifications, by adapting the techniques of \cite[Lemma~2.11 and Appendix~B]{Mallat2012}. We include the argument for completeness.

\begin{lemma} \label{lem:appendix-low-pass-deformation-estimate}
    Let $\kappa\in (0,1)$. Then, there exists $C>0$ such that, for every $\tau\in W^{1,\infty}(\R^2;\R^2)$ satisfying $\norm{D\tau}_{L^\infty} \le \kappa$, we have, for every $q\in L^2(\R^2)$,
    \begin{equation*}
        \norm{(L_\tau q)*\chi-L_\tau(q*\chi)}_{L^2} \le  C(\|\tau\|_{L^\infty}+\norm{D\tau}_{L^\infty})\|q\|_{L^2}
    \end{equation*}
    and
    \begin{equation*}
        \norm{L_\tau(q*\chi)-q*\chi}_{L^2} \le  C\|\tau\|_{L^\infty}\|q\|_{L^2}.
    \end{equation*}
    If $P$ is countable, the same bounds hold componentwise for arrays $q=(q_p)_{p\in P}\in\ell^2(P;L^2)$ after summing in $\ell^2(P)$.
\end{lemma}

\begin{proof}
    We first note that the standing assumptions imply $\chi\in W^{1,1}(\R^2)$. Indeed, choose $\eta\in C_c^\infty(\R^2)$ with $\eta \equiv 1$ on $Q_0=B_4(0)$, and set $k=\calF^{-1}\eta$. Since $\supp(\wh{\chi})\subset B_4(0)$, we have $\chi=\chi*k$. Consequently, $\nabla\chi=\chi*\nabla k$ in the sense of distributions, and hence
    \begin{equation*}
        \norm{\nabla\chi}_{L^1}\leq\norm{\chi}_{L^1}\norm{\nabla k}_{L^1}<\infty.
    \end{equation*}
    In particular, we may work with the smooth representative $\chi*k$ of $\chi$. All constants below may depend on $\kappa$ and $\chi$, but not on $\tau$ or $q$.
    
    For the Lipschitz representative of $\tau$, set
    \[F_t\coloneqq\Id-t\tau,\qquad 0\leq t\leq1,\]
    and write $F=F_1$. Since $\norm{D\tau}_{L^\infty}\leq\kappa<1$, each $F_t$ is a bi-Lipschitz homeomorphism of $\R^2$ satisfying
    \begin{equation*}
        (1-t\kappa)\abs{x-y}\leq\abs{F_t(x)-F_t(y)}\leq(1+t\kappa)\abs{x-y}.
    \end{equation*}
    In particular,
    \[(1-t\kappa)^2\leq\abs{\det DF_t(x)}\leq(1+t\kappa)^2\]
    for almost every $x\in\R^2$.
    
    Let $G=F^{-1}$ and put
    \[J_G(z)\coloneqq\abs{\det DG(z)}.\]
    Changing variables in the first convolution gives
    \begin{equation*}
        \bigl((L_\tau q)*\chi-L_\tau(q*\chi)\bigr)(x)=\int_{\R^2}K_\tau(x,z)q(z)\dz,
    \end{equation*}
    where
    \[K_\tau(x,z)=J_G(z)\chi(x-G(z))-\chi(F(x)-z).\]
    We estimate the $L^2\to L^2$ operator norm of this convolution integral kernel by Schur's test. Since
    \[\det DF(y)=\det(\Id-D\tau(y)),\]
    the two-dimensional determinant formula gives
    \[\abs{\abs{\det DF(y)}-1}\lesssim\norm{D\tau}_{L^\infty}.\]
    Together with the lower Jacobian bound, this yields
    \[\abs{J_G(F(y))-1}\lesssim_\kappa\norm{D\tau}_{L^\infty}\]
    for almost every $y\in\R^2$.
    
    Fix $z\in\R^2$ and write $z=F(y)$. Then
    \begin{align*}
        \int_{\R^2}\abs{K_\tau(x,z)}\dx
        &\leq\abs{J_G(F(y))-1}\norm{\chi}_{L^1} \\
        &\quad+\int_{\R^2}\abs{\chi(x-y)-\chi(F(x)-F(y))}\dx.
    \end{align*}
    The fundamental theorem of calculus gives
    \begin{equation*}
        \chi(x-y)-\chi(F(x)-F(y))=\int_0^1\nabla\chi(F_t(x)-F_t(y))\cdot\bigl(\tau(x)-\tau(y)\bigr)\dt.
    \end{equation*}
    Since
    \[\abs{\tau(x)-\tau(y)}\leq2\norm{\tau}_{L^\infty},\]
    a change of variables and the Jacobian bounds yield
    \begin{align*}
        &\int_{\R^2}\abs{\chi(x-y)-\chi(F(x)-F(y))}\dx \\
        &\qquad\leq2\norm{\tau}_{L^\infty}\int_0^1\int_{\R^2}\abs{\nabla\chi(F_t(x)-F_t(y))}\dx\dt \\
        &\qquad\leq\frac{2\norm{\tau}_{L^\infty}}{(1-\kappa)^2}\norm{\nabla\chi}_{L^1}.
    \end{align*}
    It follows that
    \begin{equation*}
        \operatorname*{ess\,sup}_{z\in\R^2}\int_{\R^2}\abs{K_\tau(x,z)}\dx\lesssim_{\kappa,\chi}\norm{\tau}_{L^\infty}+\norm{D\tau}_{L^\infty}.
    \end{equation*}
    
    Next fix $x\in\R^2$ and change variables $z=F(y)$. Since
    \[J_G(F(y))\abs{\det DF(y)}=1\]
    for almost every $y\in\R^2$, we obtain
    \begin{align*}
        \int_{\R^2}\abs{K_\tau(x,z)}\dz
        &=\int_{\R^2}\abs{\chi(x-y)-\abs{\det DF(y)}\chi(F(x)-F(y))}\dy \\
        &\leq\int_{\R^2}\abs{\chi(x-y)-\chi(F(x)-F(y))}\dy \\
        &\quad+\int_{\R^2}\abs{1-\abs{\det DF(y)}}\abs{\chi(F(x)-F(y))}\dy.
    \end{align*}
    The first term is bounded by
    \begin{equation*}
        \frac{2\norm{\tau}_{L^\infty}}{(1-\kappa)^2}\norm{\nabla\chi}_{L^1}
    \end{equation*}
    by the same argument as above. For the second term, the Jacobian bounds and the change of variables $z=F(y)$ yield
    \begin{equation*}
        \int_{\R^2}\abs{1-\abs{\det DF(y)}}\abs{\chi(F(x)-F(y))}\dy\lesssim_\kappa\norm{D\tau}_{L^\infty}\norm{\chi}_{L^1}.
    \end{equation*}
    Thus,
    \begin{equation*}
        \operatorname*{ess\,sup}_{x\in\R^2}\int_{\R^2}\abs{K_\tau(x,z)}\dz\lesssim_{\kappa,\chi}\norm{\tau}_{L^\infty}+\norm{D\tau}_{L^\infty}.
    \end{equation*}
    Schur's test now yields
    \begin{equation*}
        \norm{(L_\tau q)*\chi-L_\tau(q*\chi)}_{L^2}\lesssim_{\kappa,\chi}\bigl(\norm{\tau}_{L^\infty}+\norm{D\tau}_{L^\infty}\bigr)\norm{q}_{L^2}.
    \end{equation*}
    
    For the second estimate, set $h\coloneqq q*\chi$. Since $\chi\in W^{1,1}(\R^2)$, Young's inequality gives
    \[\norm{\nabla h}_{L^2}\leq\norm{\nabla\chi}_{L^1}\norm{q}_{L^2}.\]
    Moreover, $h$ is bandlimited and hence admits a smooth representative. Thus,
    \begin{equation*}
        L_\tau h(x)-h(x)=h(F(x))-h(x)=-\int_0^1\tau(x)\cdot\nabla h(F_t(x))\dt.
    \end{equation*}
    By Minkowski's inequality and the Jacobian bounds,
    \begin{align*}
        \norm{L_\tau h-h}_{L^2}
        &\leq\norm{\tau}_{L^\infty}\int_0^1\norm{\nabla h\circ F_t}_{L^2}\dt \\
        &\leq\frac{\norm{\tau}_{L^\infty}}{1-\kappa}\norm{\nabla h}_{L^2} \\
        &\leq\frac{\norm{\nabla\chi}_{L^1}}{1-\kappa}\norm{\tau}_{L^\infty}\norm{q}_{L^2}.
    \end{align*}
    This proves the second estimate.
    
    Finally, the constants in both estimates are independent of the input function. Applying the estimates to each component $q_p$, squaring, and summing over $p\in P$ proves the corresponding bounds in $\ell^2(P;L^2)$.
\end{proof}

We are finally ready to prove the desired plugin commutator estimate. 

\begin{proof}[Proof of Proposition~\ref{prop:scattering-commutator-plugin}]
    The high-pass contribution is precisely the array estimate in Proposition~\ref{prop:appendix-high-pass-deformation-commutator}:
    \begin{equation*}
        \|\calU L_\tau q-L_\tau\calU q\|_{\ell^2(P\times I_\ast;L^2)} \lesssim \|D\tau\|_{L^\infty}\|q\|_{\ell^2(P;H^{1-\beta})}.
    \end{equation*}
    It remains to control the low-pass output operator $\calA $. For each path component $q_p$, we split
    \[(\calA L_\tau q)_p-(\calA q)_p=(L_\tau q_p)*\chi-q_p*\chi\]
    as
    \begin{equation*}
        (L_\tau q_p)*\chi-q_p*\chi=\bigl((L_\tau q_p)*\chi-L_\tau(q_p*\chi)\bigr)+\bigl(L_\tau(q_p*\chi)-q_p*\chi\bigr).
    \end{equation*}
    Applying Lemma~\ref{lem:appendix-low-pass-deformation-estimate} to each component and summing in $\ell^2(P;L^2)$ gives
    \begin{equation*}
        \|\calA L_\tau q-\calA q\|_{\lp} \leq C(\|\tau\|_{L^\infty}+\norm{D\tau}_{L^\infty})\|q\|_{\lp},
    \end{equation*}
    where the constant also depends on the fixed low-pass filter through $\|\chi\|_{L^1}$ and $\|\nabla\chi\|_{L^1}$.
    Combining the high-pass and low-pass bounds gives the desired result. 
\end{proof}

\addtocontents{toc}{\protect\setcounter{tocdepth}{0}}
\section*{Acknowledgments}

This work was supported by the Deutsche Forschungsgemeinschaft (DFG, German Research Foundation)---project number 442047500/SFB 1481 Sparsity and Singular Structures. M. G. is grateful to S. I. T. and the Department of Mathematics ``G. L. Lagrange'' (DISMA) for their hospitality during a research visit in April 2026, during which much of this work originated. The authors thank Hartmut F\"uhr and Sebastian Walcher for insightful comments on an earlier version of the manuscript.

\section*{Declaration of generative AI in the manuscript preparation process}

The main ideas of this work were developed through in-person discussions between the authors during the research visit acknowledged above. The authors used large language models to explore and assess possible approaches and to assist with reviewing mathematical arguments, drafting, and revising the manuscript. The authors reviewed and edited the AI-assisted content and validated the mathematical arguments and proofs. They take full responsibility for the content of the manuscript.

\addtocontents{toc}{\protect\setcounter{tocdepth}{1}}

\bibliographystyle{plain}
\bibliography{refs}

\end{document}